\documentclass{article}
\usepackage[british]{babel}
\usepackage{amsmath,amssymb,enumerate,latexsym,xcolor,theorem, dsfont, hyperref,graphicx,subcaption,float}
\usepackage[a4paper,portrait,
            left=3.5cm,
            right=3.5cm,
            top=3.5cm,
            bottom=3.5cm]{geometry}
\newcommand{\assign}{:=}

\newcommand{\nobracket}{}

\newcommand{\tmem}[1]{{\em #1\/}}
\newcommand{\tmop}[1]{\ensuremath{\operatorname{#1}}}

\DeclareMathOperator*{\esssup}{ess\,sup}
\newenvironment{enumeratealpha}{\begin{enumerate}[a{\textup{)}}] }{\end{enumerate}}

\newenvironment{enumerateroman}{\begin{enumerate}[i.] }{\end{enumerate}}
\newenvironment{proof}{\noindent\textbf{Proof\ }}{\hspace*{\fill}$\Box$\medskip}
\newtheorem{definition}{Definition}[section]
\newtheorem{theorem}[definition]{Theorem}
\newtheorem{lemma}[definition]{Lemma}
\newtheorem{remark}[definition]{Remark}
\newtheorem{example}[definition]{Example}
\newtheorem{proposition}[definition]{Proposition}
\newtheorem{corollary}[definition]{Corollary}
\numberwithin{equation}{section}

\newcommand{\RR}{\ensuremath{\mathbb{R}}}

\begin{document}

\title{A pathwise approach to semilinear\\
SPDEs with L\'evy drivers}

\author{
  Dirk Becherer\thanks{Humboldt-Universit\"at zu Berlin, Berlin, Germany.}
  \and
  Peter Friz\thanks{Technische Universit\"at Berlin and Weierstrass Institute, Berlin, Germany.}
  \and
  Yuchen Sun\thanks{Humboldt-Universit\"at zu Berlin and Technische
    Universit\"at Berlin, Berlin, Germany.}
}
\date{24th~September 2026}

\maketitle

\begin{abstract}
We extend the notion of robust viscosity solutions to semilinear rough partial differential equations (RPDEs) of the form
\[
\partial_t u + \mathcal{L}_t u + f(t,x,u,\sigma^{\top}\nabla_x u) + h(t,x,u)\,\diamond dW = 0, 
\quad u(T,x)=\xi(x),
\]
driven by a discontinuous path $W$ of finite $q$-variation for some $q<2$.
This includes, in particular, typical paths from a broad class of  L\'evy processes, for instance from $\alpha$-stable L\'evy processes with $\alpha<2$.
Our notion of solution is
formulated through a solution map that agrees with the classical viscosity
solution for smooth drivers and is continuous in the terminal condition
and the driver, with both the driver and the solution paths viewed in a suitable
Skorokhod-type decorated-path metric
\cite{chevyrev_superdiffusive_2024}. These requirements naturally imply that robust viscosity solutions must employ Marcus-type jumps, and we establish well-posedness, a flow property, and a stochastic representation in terms of the rough BSDEs with discontinuous Young drivers \cite{becherer_rough_2026}. For L\'evy processes 
of finite $q$-variation with $q<2$,
our notion of RPDE solutions provides a pathwise interpretation for corresponding SPDEs, which are seen to be Markov processes in an infinite-dimensional function space.
\end{abstract}



\section{Introduction}

Let $W$ be a deterministic c\`agl\`ad\footnote{Left-continuous paths with right limits are natural here, as we work with backward equations.} path with values in $\mathbb{R}^e$ and finite $q$-variation for some $q\in[1,2)$. Denote by $D^q([0,T];\mathbb{R}^e)$ the space of such paths. We study solutions $u=(u_i(t,x))_{i=1,\ldots,k}$, $t\in[0,T]$, $x\in\mathbb{R}^d$, to the system of semilinear rough PDEs
\begin{equation}
  d u_i + (\mathcal{L}_t u_i + f_i(t,x,u,\sigma^{\top} \nabla_x u_i)) \, dt + h_i(t,x,u) \diamond dW = 0, 
  \quad u(T,x) = \xi(x), \quad 1 \leqslant i \leqslant k, 
  \label{semilinear-RPDE}
\end{equation}
with the differential operator
\begin{equation*}
  \mathcal{L}_t
  =
  \frac12
  \sum_{m,n=1}^d
  (\sigma\sigma^\top)_{m,n}(t,x)\,
  \partial_{x_m x_n}^2
  +
  \sum_{m=1}^d
  b_m(t,x)\,\partial_{x_m}.
\end{equation*}
A wide range of stochastic processes have sample paths of finite $q$-variation almost surely, including fractional Brownian motion with Hurst parameter $H>1/2$ as a continuous example. A prominent discontinuous example relevant for this paper is given by $\alpha$-stable L\'evy processes with $\alpha < 2$ (see Section~\ref{section-markov-process} for further examples). Solutions to RPDEs \eqref{semilinear-RPDE} then offer a pathwise notion for solutions to the corresponding stochastic PDEs (SPDEs), which are driven by such stochastic processes.

We recall the notion of robust viscosity solutions. More precisely, a
solution map $(\xi,W)\mapsto u$ is called a robust viscosity solution map if it
recovers the classical viscosity solution of \eqref{semilinear-RPDE}, whenever
$W$ is smooth, and if it is continuous with respect to both the terminal condition
$\xi$ and the driving signal $W$. In particular, this implies continuity of
the solution map restricted to smooth drivers. This allows us to view the robust
viscosity solution map as the unique continuous extension of the classical
viscosity solution map to the closure of smooth drivers in a suitable topology;
see Remark~\ref{remark-unique-continuous-extension} for further details.
This concept is well established for PDEs with continuous rough drivers; see,
for example,
\cite{caruana_partial_2009,caruana_rough_2011,diehl_backward_2012,
friz_rough_2014,friz_eikonal_2017,diehl_backward_2017,
liang_multidimensional_2023,song_backward_2025-1}. But to the best of our knowledge, a
corresponding robust viscosity theory for discontinuous drivers has not yet
been developed.

In the discontinuous setting, the choice of topologies on the domain
and codomain of the solution map becomes particularly important. Following 
\cite{chevyrev_canonical_2019,chevyrev_superdiffusive_2024,becherer_rough_2026}, our choice is
guided by two considerations. First, topologies on both spaces need to allow discontinuous paths 
to be approximated by continuous ones. Second, we want
the topology on the solution space to retain sufficient information about how
a solution evolves across a jump. These considerations lead us to use a
Skorokhod-type topology on decorated paths introduced in \cite{chevyrev_superdiffusive_2024}; see
Section~\ref{section-decorated-path}.

The main goal of this paper is to extend the notion of robust viscosity
solutions to semilinear PDEs with discontinuous drivers of finite
$q$-variation in the Young regime $q<2$.

We introduce a precise notion of robust viscosity solutions for discontinuous
drivers in Definition~\ref{def-robust-viscosity-solution}. Our main results,
which are formulated compactly in Theorem~\ref{main-result}, establish the
well-posedness of the RPDE in this class, together with a flow property and a
stochastic representation in terms of the rough BSDEs recently developed in
\cite{becherer_rough_2026}. To substantiate the interpretation of this
pathwise solution as a solution to a corresponding SPDE, we further show in
Proposition~\ref{Prop-randomise} that the RPDE solution can be randomised to
obtain a stochastic process adapted to the filtration generated by the driving
process. Finally, when the driver is a L\'evy process,
Theorem~\ref{Theorem-Markov} shows that the randomised solution defines an infinite-dimensional
Markov process.

For our analysis, the choice of integration for jumps is crucial. The notation $\diamond\, dW$ indicates that jumps are integrated in the Marcus sense (see \cite{marcus_modeling_1980,kurtz_stratonovich_1995,kunita_stochastic_2004,applebaum_levy_2009}), meaning that the jump of $u$ is given by
\begin{align}\label{marcus-type-jump}
\Delta u(t,x) := u(t+,x) - u(t,x) = u(t+,x) - \varphi\big(-h(t,x,\cdot)\,\Delta W_t,\, u(t+,x), 0\big),
\end{align}
where, for a vector field $V$, the map $r \mapsto \varphi(V,x,r)$ denotes the flow associated to the ODE
\begin{align}\label{Marusjump-ODE}
  \frac{dy}{dr}(r) = V(y(r)), \quad y(1)=x, \quad 0 \le r \le 1.
\end{align}
We refer to \eqref{marcus-type-jump} as the \textbf{Marcus jump}. It differs from the more commonly used forward-type jump
\begin{align} \label{forward-jump}
    \Delta u(t,x) = - h(t,x,u(t+,x))\,\Delta W_t,
\end{align}
which describes an instantaneous displacement, whereas the Marcus jump corresponds to evolution along the vector field over a fictitious time interval. 
This distinction is important both from a modelling and an analytical point of view. In particular, it is known that SDEs with jumps interpreted in the Marcus sense enjoy properties similar to Stratonovich SDEs in the continuous case, such as the existence of Wong--Zakai-type smooth approximations. 
We further refer to \cite{chechkin_marcus_2014} for a discussion from an application point of view.

For the present work, however, the choice of the Marcus jump structure has
effectively already been made once we decide to work with robust viscosity
solutions and choose a topology on the driver space that is weak enough to
admit smooth approximations of $W$. Indeed, if $W^n$ is any sequence of smooth
paths converging to $W$ in this topology, continuity of the solution map
requires the corresponding classical viscosity solutions to converge to the
robust viscosity solution of the RPDE driven by $W$. As the following example
shows, even the weak requirement of pointwise convergence at the continuity
points of $W$ generally rules out forward jumps in the limiting solution. In
fact, as we prove in
Proposition~\ref{prop-robust-solution-marcus-jumps} and discuss further in
Remark~\ref{remark-marcus-lift-not-circular}, the notion of robust viscosity solution
naturally requires Marcus jump dynamics. Thus, in the discontinuous
setting considered here, requiring consistency with classical viscosity
solutions for smooth drivers together with continuity of the solution map also
amounts to selecting the Marcus jump structure.

\begin{example}  \label{example}
Given a sequence of smooth paths $(W^n)_{n \in \mathbb{N}} \subset C^\infty([0,T];\mathbb{R})$ converging pointwise to a c\`agl\`ad path $W$ at every $t \in [0,T]$ at which $W$ is continuous, consider the PDE with smooth driver
\[
\partial_t u^n(t,x)  + \Delta u^n(t,x) = \mathcal{R} u^n(t,x)\, \dot{W^n_t},
\]
with $W_T = W^n_T = 0$, $\mathcal{R} = \begin{pmatrix}
0 & 1\\
-1 & 0
\end{pmatrix}$, $u^n(T,x) = \sin(x) \begin{pmatrix}
a\\
b
\end{pmatrix}$. 
A direct computation shows that the solution is given by
\[
u^n(t,x) = e^{t-T} \sin(x)
\begin{pmatrix}
\cos(W^n_t) & \sin(W^n_t)\\
-\sin(W^n_t) & \cos(W^n_t)
\end{pmatrix}
\begin{pmatrix}
a\\
b
\end{pmatrix}.
\]
In the case $(a,b) = (1,0)$, this simplifies to
\[
u^n(t,x) = e^{t-T} \sin(x)
\begin{pmatrix}
\cos(W^n_t)\\
-\sin(W^n_t)
\end{pmatrix}.
\]
As $n \to \infty$, we obtain convergence of $u^n$ to
\begin{align} \label{solution-example}
u(t,x) = e^{t-T} \sin(x)
\begin{pmatrix}
\cos(W_t)\\
-\sin(W_t)
\end{pmatrix},
\end{align}
for all continuity points $t$ of $W$. Since we work with c\`agl\`ad paths, this already uniquely determines for all $t \in [0,T]$ the limiting solution $u$ of the RPDE
\begin{align} \label{example-RPDE}
    \partial_t u(t,x) + \Delta u(t,x) = \mathcal{R} u(t,x)\, \diamond dW_t.
\end{align}
At each jump time $t$, the values $u(t)$ and $u(t+)$ are respectively the
initial and terminal values of the ODE
\begin{align*}
\frac{d \varphi}{dr}(r) = \mathcal{R} \Delta W_{t}  \varphi(r), \quad \varphi(1)=u(t+)=e^{t-T} \sin(x)\, (\cos(W_{t+}),\; -\sin(W_{t+}))^{\top}, \quad 0 \le r \le 1.
\end{align*}
This means that the limiting solution $u$ in \eqref{solution-example} does not admit forward jumps as in \eqref{forward-jump}, but instead exhibits Marcus-type jumps as in \eqref{marcus-type-jump}.
We emphasize that pointwise convergence at continuity points alone is not
sufficient for our purposes. We therefore prove convergence in a
Skorokhod-type topology on the space of decorated paths; see
Section~\ref{section-decorated-path} and Section~\ref{section-example-revisited}.
\end{example}

The RPDE theory developed here naturally yields pathwise solutions to
semilinear SPDEs driven by discontinuous stochastic processes whose sample
paths almost surely have finite $q$-variation, $q<2$. Importantly, this
includes a large class of L\'evy processes. L\'evy-driven SPDEs
have been studied extensively; however, while the theory of SPDEs with forward
jumps is by now well developed (cf.\ the monograph by Peszat and Zabczyk
\cite{peszat_stochastic_2007}), SPDEs with Marcus jumps have only recently
begun to receive increased attention. We mention \cite{hartmann_first-order_2023}, where the authors show that, for first-order linear SPDEs, the Marcus interpretation is the appropriate framework to obtain solutions via the method of {stochastic characteristics, originally developed by Kunita \cite{kunita1990stochastic} for first-order linear SPDEs driven by Brownian motion}. The BSDE approach used in this paper and described below could be viewed as the analogue of the method of characteristics for second-order semilinear PDEs. We also mention the series of papers by Brzeźniak and Manna with coauthors \cite{brzezniak_weak_2019, brzezniak_wongzakai_2019, brzezniak_martingale_2019, brzezniak_weak_2020}, where different SPDEs driven by L\'evy noise arising in mathematical physics are studied. In these works, the use of Marcus jumps allows the solutions to preserve certain invariance properties that are intrinsic to the underlying physical models and are consistent with the corresponding continuous dynamics. The authors also point out that rough path techniques may provide a useful framework for the analysis of such equations. Although the SPDEs considered there differ from those studied in this paper, our results can be viewed as contributing to this line of investigation.

Besides being one of our main results, the rough FBSDE representation is also
one of the principal tools used in our analysis of the RPDE. In a nutshell, we define a candidate
for the robust viscosity solution by $u(t,x)\assign Y^{t,x}_t$, where
$(X^{t,x},Y^{t,x},Z^{t,x})$ solves the rough forward-backward
SDE (FBSDE)
\begin{align}
  X^{t,x}_s = & \; x + \int_t^s b(r,X^{t,x}_r)\,dr + \int_t^s \sigma(r,X^{t,x}_r)\,dB_r,\label{rough-FBSDE} \\
  Y^{t,x}_s = & \; \xi(X^{t,x}_T) + \int_s^T f(r,X^{t,x}_r,Y^{t,x}_r,Z^{t,x}_r)\,dr + \int_s^T h(r,X^{t,x}_r,Y^{t,x}_r) \diamond dW_r  - \int_s^T Z^{t,x}_r\,dB_r,  \nonumber
\end{align}
which is multidimensional in the $Y$-component, with $k>1$.
To show that the function $u$ defined in this way is indeed a robust viscosity
solution, it remains to verify the two defining properties: consistency with
classical viscosity solutions for smooth drivers and continuity with respect
to the terminal condition and the driving signal.

For a smooth driver $W$, the rough FBSDE and RPDE reduce to their classical
counterparts. We can therefore use the multidimensional BSDE and viscosity
solution theory developed by Pardoux in \cite{decreusefond_backward_1998},
which shows that $u(t,x)=Y_t^{t,x}$ is a viscosity solution of the
corresponding PDE. Since viscosity solution theory for multidimensional
systems appears less familiar than its scalar counterpart, we revisit it in
Section~\ref{Chap-ApproxPDE}, where we complement the existing theory by a
short argument for uniqueness in Theorem~\ref{multidimentional-visco}. This establishes consistency of our
RPDE solution with classical viscosity solutions for smooth drivers.

Establishing continuity of the solution map is considerably more delicate and
constitutes one of the main technical difficulties of the paper. Given paths
$(W^n)_{n\in\mathbb{N}}$ converging to $W$, we consider the corresponding
solutions $u^n(t,x)=Y_t^{n,t,x}$, where
$(X^{t,x},Y^{n,t,x},Z^{n,t,x})$ solves the rough FBSDE
\eqref{rough-FBSDE} driven by $W^n$. The goal is to prove that $u^n$ converges
to $u(t,x)=Y_t^{t,x}$. The stability theory for rough BSDEs, however, yields
only pointwise convergence of $u^n(t,x)$ to $u(t,x)$ at the continuity points
of $W$. For continuous drivers, an Arzel\`a--Ascoli argument would extend such
pointwise convergence to locally uniform convergence. However, such an argument cannot be directly applied in our setting, since the driving paths and hence the solutions are non-continuous. We therefore work with the Skorokhod-type topology on the space of decorated
paths. Heuristically, a
decorated path consists of the original path together with additional data at
each jump describing how the path would evolve during an inserted fictitious time
interval. We refer to these additional path segments as \emph{excursions}. To
lift an ordinary path to the decorated-path space, one must specify these
excursions. In our setting, the excursion of $u^n$ at a jump time $t$ is chosen as a
suitable parametrization of the Marcus trajectory
\[
  r\longmapsto
  \varphi\bigl(-h(t,x,\cdot)\,\Delta W^n_t,u^n(t+,x),r\bigr),
  \qquad r\in[0,1],
\]
where $\varphi$ is the flow defined by \eqref{Marusjump-ODE}.
We show that adjoining these excursions to $u^n$ yields a continuous function
and allows us to apply Arzel\`a--Ascoli type arguments. To this end, two further
difficulties must be overcome. First, Skorokhod-type convergence
also involves subtle time reparametrizations, which we study in greater
detail in Section~\ref{section-decorated-path}. Second, in the multidimensional
setting $k>1$, proving the equicontinuity required by the Arzel\`a--Ascoli
argument becomes substantially more delicate. It would be considerably
simpler if the BSDE solution $Y$ depended Lipschitz continuously on its
terminal condition in an $L^2$-sense. Since such results appear less readily available for multidimensional BSDEs in our problem, we
prove and apply a tailored pairwise stability estimate for rough FBSDEs
in Proposition~\ref{prop:uniform-pairwise-stability}.

Even after such compactness arguments, the limiting decorated path is not yet
fully identified. Pointwise convergence determines its underlying c\`agl\`ad
path as $u$, but it does not show that its limiting excursions are still
Marcus trajectories. Proving this identification requires a finer
stability analysis of the RPDE on the inserted fictitious
time intervals. Using the Markovian structure and the backward flow property
of the rough FBSDE, we show that the limiting dynamics on these intervals are
deterministic and adhere to the Marcus ODE.

Identifying the excursions is important in that it validates the choice made
above to equip RPDE solutions with Marcus excursions. To see that this choice
is not imposed artificially, one may start with smooth drivers $W^n$ and the
corresponding continuous solutions $u^n$. Since the paths $u^n$ are
continuous, they have no jump excursions to specify, and the choice of
decoration therefore has no effect. Nevertheless, the identification above
shows that their decorated-path limit carries Marcus excursions. Thus, smooth
approximation itself selects the Marcus decoration and justifies its use for
RPDE solutions with discontinuous drivers. For further details on this convergence and on the continuity of the RPDE
solution map, we refer to Theorem~\ref{theorem-Skorokhod} and
Proposition~\ref{Prop-Continuity}.

The rough FBSDE representation can be viewed as a nonlinear
Feynman--Kac result for RPDEs with discontinuous rough drivers, extending the classical
correspondence between BSDEs and semilinear parabolic PDEs due to
Pardoux and Peng
\cite{rozovskii_backward_1992}. For
continuous RPDEs, analogous BSDE methods have been applied in
\cite{diehl_backward_2017,liang_multidimensional_2023,
song_backward_2025-1}. Compared to approaches based on flow transformations
(see \cite[Chapter 12.2.4]{friz_course_2020} and the references therein), the BSDE approach
directly provides a stochastic representation of the solution. 

The paper is structured in the following way. The precise framework for Skorokhod-type topology on the space of decorated paths
is developed in Section~\ref{section-decorated-path}, where we recall results
on decorated paths and establish
several additional results needed in the proofs. Section~\ref{section-mainresult}
introduces our notion of solution to the RPDE \eqref{semilinear-RPDE} and
states the main results on well-posedness, the flow property, and the
stochastic representation via rough BSDEs. The proofs of these results are
given in Section~\ref{section-proof}. Finally,
Section~\ref{section-markov-process} turns to the stochastic setting. In
Section~\ref{Section-measurability}, we verify the measurability of the
solution map $(\xi,W)\mapsto u$, which allows us to randomise the pathwise
solution and obtain an adapted solution of the corresponding SPDE.
Section~\ref{Section-SPDE-Markov} then establishes the Markov property when the driving signal is a L\'evy process.

\section{Skorokhod-type Rough Path Metrics}\label{section-decorated-path}

We work with the framework of decorated paths introduced in \cite{chevyrev_superdiffusive_2024}. Subsection~\ref{Chap-decorated-paths} revisits their main results in a slightly modified framework, while Subsection~\ref{Chap-new-results-decorated-paths} develops new properties of decorated paths and of the associated Skorokhod-type metrics.

\subsection{Skorokhod-type Metrics on the Space of Decorated
Paths}\label{Chap-decorated-paths}

Our setting differs from that of \cite{chevyrev_superdiffusive_2024} in two respects. First, we work with
c{\`a}gl{\`a}d paths instead of c{\`a}dl{\`a}g paths. As a consequence, several definitions require minor modifications. Second, we allow paths to take values in a metrizable vector space $(E,d_E)$ instead of $\mathbb{R}^e$. This extension is motivated by our later analysis of the solution $u$ to the RPDE \eqref{semilinear-RPDE}, which takes values in $D ([0, T], E)$ with $E =
C (\mathbb{R}^d;\mathbb{R}^k)$, where $C(\mathbb{R}^d;\mathbb{R}^k)$ is equipped with the compact-open topology.

We start by recalling the definition of $p$-variation. Let $(E,d_E)$ be a
metric space and let $I=[a,b]\subset\mathbb R$. We denote by
$\mathcal P(I)$ the collection of all finite partitions
$\pi=(a=t_0<\cdots<t_n=b)$ of $I$. For a path $x:I\to E$ and $p>0$, define
the $E$-valued $p$-variation seminorm by
\[
\|x\|_{p;[a,b];E}
\assign
\left(
\sup_{(t_i)\in\mathcal P(I)}
\sum_i d_E(x_{t_i},x_{t_{i+1}})^p
\right)^{1/p}.
\]
We omit the target space $E$ whenever it is clear from the context. The
space $D^p(I;E)$ consists of all c{\`a}gl{\`a}d paths $x:I\to E$ with
finite $p$-variation.

\begin{definition}
  \label{Def-decorated}Let $I = [a, b] \subset \mathbb{R}$ be a closed interval and let $\Phi : I \rightarrow D ([0, 1] ; E)$ be a path taking values in the space of c{\`a}gl{\`a}d paths on $[0, 1]$. We say $t \in I$ is a stationary point of $\Phi$ if $\Phi(t)$ is constant on $[0,1]$. For $\Pi \subset I$, we denote by $\bar{\mathcal{D}}(I;E)$ the space of pairs $(\Phi,\Pi)$ satisfying the following properties:
  \begin{enumerateroman}
    \item the map $t \mapsto \Phi (t) (0)$ lies in $D (I)$,
    
    \item $\Pi$ is at most countable and contains all non-stationary points of
    $\Phi$,
    
    \item for all $\varepsilon > 0$, there exist only finitely many points $t
    \in \Pi$ such that
    \[ \sup_{s \in [0, 1]} d_E (\Phi (t) (s), \Phi (t) (0)) |  > \varepsilon.
    \]
  \end{enumerateroman}
\end{definition}
Whenever no confusion arises, we write $\bar{\mathcal{D}} (I) \equiv
\bar{\mathcal{D}} (I ; E).$ \\
For two compact intervals $I_1$ and $I_2$, let
$\Lambda_{I_1;I_2}$ denote the set of strictly increasing bijections from
$I_1$ to $I_2$, and write $\Lambda_I\assign\Lambda_{I;I}$. For paths
$\varphi^1\in D(I_1)$ and $\varphi^2\in D(I_2)$, define their Fr{\'e}chet
distance by
\[
  d_{\mathcal F}(\varphi^1,\varphi^2)
  \assign
  \inf_{\lambda\in\Lambda_{I_2;I_1}}
  \|\varphi^1\circ\lambda-\varphi^2\|_\infty.
\]
We say that $\varphi^1$ and $\varphi^2$ are reparametrizations of each other
if $d_{\mathcal F}(\varphi^1,\varphi^2)=0$.

We now introduce an equivalence relation on $\bar{\mathcal{D}} (I)$. We say that
$(\Phi^1, \Pi^1), (\Phi^2, \Pi^2) \in \bar{\mathcal{D}} (I)$ are equivalent if, for every $t \in I$, 
the functions
\[ s \mapsto \Phi^i (t) (s)  \hspace{0.17em} \mathds{1}_{\{s < 1\}} + \Phi^i
   (t +) (0)  \hspace{0.17em} \mathds{1}_{\{s = 1\}} \]
are reparametrizations of each other. In particular, this equivalence
implies $\Phi^1 (t -) (0) = \Phi^1 (t) (0) = \Phi^2 (t) (0) = \Phi^2 (t -)
(0)$ and $\Phi^1 (t +) (0) = \Phi^2 (t +) (0)$ for all $t \in I$. Observe that this equivalence relation is independent of the choice of 
$\Pi^1$ and $\Pi^2$.

\begin{definition}
  We define
\[
\mathfrak{D}(I;E) := \bar{\mathcal{D}}(I;E) / \sim
\]
to be the space of equivalence classes of $\bar{\mathcal{D}}(I;E)$ under the relation $\sim$. 
Elements of $\mathfrak{D}(I;E)$ are called \textbf{decorated paths}\footnote{Compared with \cite{chevyrev_superdiffusive_2024}, we use this terminology slightly more broadly: when no confusion can arise, we also refer to elements of the non-quotient space $\bar{\mathcal{D}}(I;E)$ as decorated paths.}.
\end{definition}

Let $(\Phi, \Pi) \in \bar{\mathcal{D}} ([a, b])$ and suppose that $\Pi = \{t_k \}_{k = 1,
\ldots, m}$. We define the \textbf{$\delta$-extension} of $\Phi$, denoted by 
$\Phi^\delta \in D([a,b+\delta];E)$, by adding a fictitious time interval of length $\delta > 0$ as follows. If $\Pi = \emptyset$, we set
\[ \Phi^{\delta} (t) = \Phi (t) (0)  \quad \text{for } t \in [a, b]  \quad
   \text{and} \quad \Phi^{\delta} (t) = \Phi (b)(0)  \quad \text{for } t \in (b,
   b + \delta].\]
Otherwise, define
\[ r = \sum_{k = 1}^m 2^{- k}  \quad \text{and} \quad r_k = \frac{2^{- k}
   \delta}{r} . \]
We also define a strictly increasing c{\`a}gl{\`a}d function
\begin{equation*}
  \tau^{\delta} : [a, b] \to [a, b + \delta], \qquad \tau^{\delta} (t) = t +
  \sum_{k = 1}^m r_k  \hspace{0.17em} \mathds{1}_{\{t_k < t\}}.
\end{equation*}
and define $\Phi^{\delta}$ by\footnote{If $\Phi(t)(\cdot) \in C([0,1])$ and $\Phi(t+)(1)=\Phi(t)(1)$ for all 
$t \in [a,b]$, then this construction coincides with that of 
\cite{chevyrev_canonical_2019}.} 
\begin{equation}
  \Phi^{\delta}_s = \left\{\begin{array}{ll}
    \Phi (t) (0), & \text{if } s = \tau^{\delta}_t  \text{ for some } t \in [a,
    b],\\
    \Phi (t_k)  ((s - \tau^{\delta}_{t_k}) / r_k), \nobracket & \text{if }
    s \in (\tau^{\delta}_{t_k}, \tau^{\delta}_{t_k +} \nobracket]  \text{ for
    some } 1 \leqslant k \leqslant m.
  \end{array}\right. \label{delta-extention}
\end{equation}
Following \cite{chevyrev_superdiffusive_2024} we
refer to the additional path segments on the
intervals\linebreak $(\tau^{\delta} (t_k), \tau^{\delta} (t_k +)]$, $1 \leqslant k \leqslant m$,
as \textbf{excursions} of the path.
For $(\Phi^1, \Pi^1), (\Phi^2, \Pi^2) \in \bar{\mathcal{D}} (I ; E)$ we introduce the pseudo metric
\[ \alpha_{\infty ; [a, b]} ((\Phi^1, \Pi^1), (\Phi^2, \Pi^2)) \assign
   \lim_{\delta \to 0} \inf_{\lambda \in \Lambda_{[a, b + \delta]}} \max
   \left\{ \| \lambda - \mathrm{id} \|_{\infty}, \hspace{0.27em} \| \Phi^{1,
   \delta} \circ \lambda - \Phi^{2, \delta} \|_{\infty ; [a, b + \delta]}
   \right\} . \]
\cite[Lemma~8.12]{chevyrev_superdiffusive_2024} shows that the above limit 
exists and is independent of the specific choices made in the construction, 
such as the ordering of the jumps, the definition of $r_k$, and, in particular, 
the choice of $\Pi^1$ and $\Pi^2$ (provided that Definition~\ref{Def-decorated}(ii) is satisfied).\\
Therefore, we suppress the explicit reference to $\Pi$ and speak only of $\Phi \in \bar{\mathcal{D}} ([a, b])$, but when necessary, we state it explicitly. Moreover, it is shown that
$\alpha_{\infty ; [a, b]} (\Phi^1, \Phi^2) = 0$ if and only if $\Phi^1 \sim
\Phi^2$. Thus, although $\alpha_{\infty}$ is not a metric on 
$\bar{\mathcal{D}}(I)$, it induces a genuine metric on the quotient space 
$\mathfrak{D}(I)$.
The authors of {\cite{chevyrev_superdiffusive_2024}} also show that
the above 
construction extends to the $p$-variation setting by replacing the uniform 
norm with the $p$-variation norm. We write
\[
  \bar{\mathcal{D}}^{p\text{-var}}(I)
  =
  \left\{ \hspace{0.17em} \Phi \in \bar{\mathcal{D}}(I)
  \hspace{0.27em} \mid \hspace{0.27em}
  \| \Phi^\delta \|_{p\text{-var}} < \infty \right\}
\]
for the corresponding non-quotient space. The quotient space
$\mathfrak{D}^{p\text{-var}}(I)$ is defined from
$\bar{\mathcal{D}}^{p\text{-var}}(I)$ using the same equivalence relation as above and is
equipped with the $p$-variation-type Skorokhod metric $\alpha_p$ defined by
\begin{equation}
  \begin{aligned}
  &\alpha_{p ; [a, b]} ((\Phi^1, \Pi^1), (\Phi^2, \Pi^2))\\
  &\quad\assign \lim_{\delta
  \to 0} \inf_{\lambda \in \Lambda_{[a, b + \delta]}} \max \left\{
  \begin{aligned}
    &\|\lambda-\mathrm{id}\|_\infty,\\
    &{d_E(\Phi^1(a)(0),\Phi^2(a)(0))}\mathbin{+}
    \|\Phi^{1,\delta}\circ\lambda-\Phi^{2,\delta}\|_{p;[a,b+\delta]}
  \end{aligned}\right\} .
  \end{aligned}
  \label{Skorokhod-metric}
\end{equation}
{The initial-value term distinguishes paths differing by a
constant. For $p=\infty$,
we omit the initial-value term in the metric formulas, since it is already captured by the uniform norm.}

There are different ways to embed c\`agl\`ad functions into the space of decorated paths. Depending on the chosen embedding, the Skorokhod metric $\alpha_{p;[a,b]}$ may recover previously studied Skorokhod-type metrics on the path space $D(I)$. Two natural embeddings are given by
\phantomsection\label{embedding-i}
\begin{align}
  \imath : D(I) \hookrightarrow \bar{\mathcal{D}}(I),
  \qquad
  (\imath h)(t)(s)
  &= h(t), \nonumber\\
  \jmath : D(I) \hookrightarrow \bar{\mathcal{D}}(I),
  \qquad
  (\jmath h)(t)(s)
  &= s h(t+) + (1-s) h(t). \nonumber
\end{align}
where the latter is simply the linear path connecting $h(t)$ and $h(t+)$. Notice that we suppressed the set $\Pi$ in the above definition, but when necessary, it may always be chosen as any countable set containing all discontinuities of $h$. \\
For later purposes, we also introduce the continuous non-decreasing surjection
$c^{\delta} : [a, b + \delta] \rightarrow [a, b]$ defined by $c^{\delta}
\equiv (\imath \tmop{id}_{[a, b]})^{\delta}$ in the sense of
\eqref{delta-extention}. One easily verifies that $c^\delta$ admits the equivalent representation
\[ c^{\delta} (t) := \inf \{ s \in [a, b] : \tau^{\delta} (s) \geqslant
   t \}, \]
that is, $c^\delta$ is the left inverse of $\tau^{\delta}$, i.e. $c^{\delta} \circ
\tau^{\delta} = \tmop{id}_{[a, b]}.$ \\

Returning to these two embeddings, we recall their relation with the classical
Skorokhod metrics.
Let $\sigma^p_{J1}$ denote the $p$-variation-type $J1$ metric on $D^p(I)$ (see \cite{friz_differential_2018}) and $\sigma^p_{M1}$ the $p$-variation-type $M1$ metric on $D^p(I)$ (see \cite{chevyrev_canonical_2019}). In the special case $p=\infty$, these reduce to the classical Skorokhod $J1$ metric on $D(I)$ (see \cite[Chap.~12]{billingsley_convergence_1999} or \cite[Chap.~VI]{jacod_limit_2003}) and the classical $M1$ metric on $D(I)$ (see \cite[Chap.~12.3]{whitt_stochastic-process_2002}). It is straightforward to verify that $(D^p(I),\sigma^p_{J1})$ is {homeomorphic} to $(\imath D^p(I),\alpha_p)$, while $(D^p(I),\sigma^p_{M1})$ is {homeomorphic} to $(\jmath D^p(I),\alpha_p)$. Thus, the Skorokhod metric on the space of decorated paths provides a unified framework for both $J1$- and $M1$-type Skorokhod metrics on path spaces.

The decorated-path framework also allows us to go beyond the
classical $J 1 / M 1$ settings. Given a family of complete vector fields $\{
U_t : t \in I \}$, define
\begin{equation} \label{def-kappa-V}
  ({\kappa_U} h) (t) (s) \assign \left\{\begin{array}{ll}
    h (t), & s = 0,\\
    \varphi (U_t, h (t +), s), & s \in (0, 1],
  \end{array}\right.
\end{equation}
where $(x, s) \mapsto \varphi (U_t, x, s)$ denotes the (backward) flow
associated with the ODE

\begin{align*}
  \frac{dy}{ds} = U_t (y (s)),  \qquad y (1) = x, \qquad 0 <
  s \le 1.
\end{align*}
If $h$ is continuous at $t$ and $U_t \equiv 0$ then $({\kappa_U} h) (t)$ is
continuous. \\
In the context of RDEs, such embeddings have the following special form. We
will later also see in
Definition~\ref{Def-Marcus-lift} how the same idea works in the RPDE setting.
\begin{example}[Marcus embedding]\label{Marcus-embedding} ~\
{
Let $W\in D(I;\mathbb R^e)$ be c{\`a}gl{\`a}d and let
$V_t:\mathbb R^k\to\mathcal L(\mathbb R^e,\mathbb R^k)$, $t\in I$.
Set $\Delta W_t=W(t+)-W(t)$ and $U_t=V_t\Delta W_t$, and define
$\kappa_{V,W}h:=\kappa_Uh$.
The excursion $(\kappa_{V,W}h)(t)$ is continuous if and only if
$h(t)=\varphi(V_t\Delta W_t,h(t+),0)$ and this is exactly the Marcus jump condition for $dh_t=V_t(h_t)\diamond dW_t$;
see \cite{chevyrev_canonical_2019,chevyrev_superdiffusive_2024}.
In particular, $\Delta W_t=0$ implies $\Delta h_t=0$, where
$\Delta h_t=h(t+)-h(t)$.
For the backward equation $-dh_t=V_t(h_t)\diamond dW_t$, we apply
the same construction with $-V$, giving the lift $\kappa_{-V,W}h$.
}
\end{example}

We next provide an alternative, but equivalent,  definition of the metric
$\alpha_p$ in \eqref{Skorokhod-metric}. Readers familiar with the classical
Skorokhod $M1$ theory will recognize that this formulation is conceptually similar to the definition
using the completed graph; see \cite{whitt_stochastic-process_2002}.

Define $\mathcal{D}= \bigcup_{a < b} D [a, b]$ to be the space of c{\`a}dl{\`a}g
paths parametrized by compact intervals. For $h_1 \in D [a_1, b_1]$ and
$h_2 \in D [a_2, b_2]$, write $h_1
\sim_{\mathcal{F}} h_2$ if $d_{\mathcal{F}} (h_1, h_2) = 0$. Denote the corresponding equivalence class by
\[ [h] = \{h' \in \mathcal{D}: h \sim_{\mathcal{F}} h' \} . \]
Similarly, we define the $p$-variation-type {\tmem{Fr{\'e}chet distance}} by
\begin{equation}
  d_{\mathcal{F},p} (h_1, h_2) = \inf_{\rho \in
  \Lambda_{[a_2,b_2];[a_1,b_1]}}
  \left( {d_E(h_1(a_1),h_2(a_2))}\mathbin{+}
  \|h_1 \circ \rho - h_2 \|_p\right) .
  \label{Frechet-pvar-distance}
\end{equation}

Given $(x, \Pi) \in \bar{\mathcal{D}} (I)$, we define
\[ \psi^{\delta} (t) = (x^{\delta}, c^{\delta, \Pi}) \]
and refer to $\psi^{\delta}$ as the $\delta$-parametric representation of $(x, \Pi)$.
 Similarly to \cite[Lemma~8.13]{chevyrev_superdiffusive_2024} we obtain the following result.

\begin{lemma}
  \label{lemma-alternative-alpha-p}Let $(x_1, \Pi^1), (x_2, \Pi^2) \in
  \mathfrak{D}^{p \text{-var}}$ for some $p \in \{0\} \cup [1, \infty]$. Then, for all $\delta
  > 0$,
  \begin{equation*}
    \alpha_p (x_1, x_2)
    = \inf_{\rho \in \Lambda_{[0,T+\delta]}}\max\left\{
    \begin{aligned}
      &{d_E(x_1(0)(0),x_2(0)(0))}\mathbin{+}
      \|x_1^\delta\circ\rho-x_2^\delta\|_p,\\
      &\|c^{\delta,\Pi^1}\circ\rho-c^{\delta,\Pi^2}\|_\infty
    \end{aligned}\right\}.
  \end{equation*}
\end{lemma}

\begin{proof}
 The argument is analogous to the proof of
\cite[Lemma~8.13]{chevyrev_superdiffusive_2024}.
\end{proof}

\subsection{Skorokhod Metrics on Decorated Paths: New Results}\label{Chap-new-results-decorated-paths}
The following results provide a deeper understanding of the metric $\alpha_p$ and will be crucial for the proofs throughout the paper. Lemmas~\ref{Lemma-localuniform} and \ref{restriction-alpha-metric} extend important properties of classical Skorokhod metrics (see, e.g., \cite{billingsley_convergence_1999, whitt_stochastic-process_2002}) to the setting of decorated paths. We also emphasize the role of Lemma~\ref{alternative-a-metric}, despite its short proof. When considering convergence of a sequence $(x_n,\Pi^n)_{n\in\mathbb{N}} \subset \mathfrak{D}^{p\text{-var}}([0,T])$, working directly with the definition of the Skorokhod metric $\alpha_p$ in \eqref{Skorokhod-metric} can be cumbersome, as it involves several layers of limiting procedures: the limit $n \to \infty$, a vanishing fictitious time parameter $\delta \to 0$, and an infimum over all possible time reparametrizations. The next lemma reduces these to a single limiting procedure, which is particularly useful. For instance, it allows us to apply the Arzelà–Ascoli theorem in Theorem~\ref{theorem-Skorokhod}.

\begin{lemma} \label{alternative-a-metric} Let $p \in \{0\} \cup [1,\infty] $ and let $(x,\Pi), (x_n,\Pi^n) \subset \mathfrak{D}^{p\text{-var}}([0,T])$. Then $(x_n)_{n
  \in \mathbb{N}}$ converge to $x$ in $\alpha_p$ if and only if for any $\delta >
  0$\footnote{For the backward implication, it is in fact enough to assume
  \eqref{alternativ-alpha} to hold for some $\delta > 0$ as one can see from
  the proof.} there exists a sequence $(\lambda^{n, \delta})_{n \in
  \mathbb{N}} \subset \Lambda_{[0, T + \delta]}$ such that
  \begin{equation}
    \lim_{n \rightarrow \infty}
    {d_E(x(0)(0),x_n(0)(0))}
    +\| x^{\delta} - x^{\delta}_n \circ \lambda^{n,
    \delta} \|_p + \| c^{\delta, \Pi} - c^{\delta, \Pi^n} \circ \lambda^{n,
    \delta} \|_{\infty} = 0. \label{alternativ-alpha}
  \end{equation}
  In the special case where $(x_n)_{n \in \mathbb{N}}$ consists of continuous paths, the
  sequence $(x_n)_{n \in \mathbb{N}}$ converges to $x$ in $\alpha_p$ if and
  only if for any $\delta > 0$, there exists a sequence $(\lambda^{n,
  \delta})_{n \in \mathbb{N}} \in \Lambda_{[0, T + \delta] ; [0, T]}$ such
  that
  \begin{equation}
    \lim_{n \rightarrow \infty}
    {d_E(x(0)(0),x_n(0))}
    +\| x^{\delta} - x_n \circ \lambda^{n, \delta}
    \|_p + \| c^{\delta, \Pi} - \lambda^{n, \delta} \|_{\infty} = 0.
    \label{alternativ-alpha-special}
  \end{equation}
\end{lemma}

\begin{remark}
The convergence criterion \eqref{alternativ-alpha} was already stated in
\cite[Lemma~4.4]{becherer_rough_2026}, see also there for additional properties
that may be imposed on the choice of the reparametrizations
$(\lambda^{n,\delta})_{n\in\mathbb N}$. The special case
\eqref{alternativ-alpha-special} is not used directly in this paper, but it is
useful for understanding the behaviour of the $\alpha_p$ metric. In particular,
it could be used to simplify the proof of Theorem~\ref{theorem-Skorokhod}, for
instance in \eqref{property-lambda-n}. We nevertheless work with the more general form \eqref{alternativ-alpha}, since
we reuse the same argument to prove the continuity of the solution map in
Proposition~\ref{Prop-Continuity}.
\end{remark}

\begin{proof}
  \space Fix $\delta > 0$. By Lemma~\ref{lemma-alternative-alpha-p}, the convergence $\alpha_p(x_n,x) \to 0$ is equivalent to \[\lim_{n \to \infty}
\inf_{\rho \in \Lambda_{[0,T+\delta]}}
\max \left\{
{d_E(x_n(0)(0),x(0)(0))}\mathbin{+}
\|x_n^\delta \circ \rho - x^\delta\|_p,\;
\|c^{\delta,\Pi^n} \circ \rho - c^{\delta,\Pi}\|_\infty
\right\}
= 0.
\] 
By a simple diagonalization argument, one can see
that this is obviously equivalent to \eqref{alternativ-alpha}.
  
  To see \eqref{alternativ-alpha-special}, we define a sequence
  $(\rho_m)_{m \in \mathbb{N}} \subset \Lambda_{[0, T + \delta] ; [0, T]}$
  with $\rho_m (t) = t \left( 1 - \frac{1}{m} \right)$ for $t \leqslant T$ and
  $\rho_m (t) = \left( 1 - \frac{1}{m} \right) + \frac{1}{m} (1 - t)$ for $T
  \leqslant t \leqslant T + \delta$. Then it holds 
  \[ \lim_{m \rightarrow \infty} \| x^{\delta}_n - x_n \circ \rho_m \|_p + \|
     c^{\delta, \emptyset} - \tmop{id} \circ \rho_m \|_{\infty} = 0, \]
  combining this with \eqref{alternativ-alpha} yields
  \eqref{alternativ-alpha-special}.
\end{proof}

\begin{lemma}
  \label{equicontinuity-p-var}
  Let $p \in \{0\} \cup [1,\infty]$ and let $(x_n,\Pi^n)_{n \in \mathbb{N}}
  \in \mathfrak{D}^{p \text{-var}}$ converge in $\alpha_p$ to
  $(x,\Pi) \in \mathfrak{D}^{p \text{-var}}$. Let
  $(\lambda^{n,\delta})_{n \in \mathbb{N}}$ be the sequence given by
  Lemma~\ref{alternative-a-metric}. Then, for any $\delta > 0$ and any
  $\varepsilon > 0$, there exist $\gamma > 0$ and $N \in \mathbb{N}$ such that
  for all $n \geqslant N$ and all $s \in [0,T+\delta]$,
  \[
    \| x^{\delta}_n \circ \lambda^{n,\delta} \|_{p ; [s, (s+\gamma)\wedge(T+\delta)]}
    +
    \| c^{\delta,\Pi^n} \circ \lambda^{n,\delta} \|_{0 ; [s, (s+\gamma)\wedge(T+\delta)]}
    \leqslant \varepsilon .
  \]
\end{lemma}

\begin{proof}
  Given a fixed $\varepsilon > 0$, by
  \cite[Lemma~4.7]{dudley_introduction_1998-1}, there exists a finite partition
  $\{t_i\}_{i=1,\ldots,m}$ such that
  \[
    \| x^{\delta} \|_{p ; [t_i, t_{i+1}]}
    +
    \| c^{\delta,\Pi} \|_{0 ; [t_i, t_{i+1}]}
    \leqslant \frac{\varepsilon}{2^{p+1}},
    \qquad i=1,\ldots,m-1.
  \]
  Set $\gamma := \min_{i=1,\ldots,m-1} |t_{i+1}-t_i|.$ For some $j=1,\ldots,m-1$ we have $s\in[t_j,t_{j+1}]$, and it holds
  \begin{align}
    \| x^{\delta} \|_{p ; [s,(s+\gamma)\wedge(T+\delta)]}
    &\leqslant
    \| x^{\delta} \|_{p ; [t_j,t_{j+2}]}
    \leqslant
    2^{p-1}
    \big(
      \| x^{\delta} \|_{p ; [t_j,t_{j+1}]}
      +
      \| x^{\delta} \|_{p ; [t_{j+1},t_{j+2}]}
    \big)
    \leqslant \frac{\varepsilon}{4},
    \label{uniform-cts-pvar-x}
    \\
    \| c^{\delta,\Pi} \|_{0 ; [s,(s+\gamma)\wedge(T+\delta)]}
    &\leqslant
    \| c^{\delta,\Pi} \|_{0 ; [t_j,t_{j+2}]}
    \leqslant
    \| c^{\delta,\Pi} \|_{0 ; [t_j,t_{j+1}]}
    +
    \| c^{\delta,\Pi} \|_{0 ; [t_{j+1},t_{j+2}]}
    \leqslant \frac{\varepsilon}{4}.
    \label{uniform-cts-pvar-c}
  \end{align}
  Lemma~\ref{alternative-a-metric} implies that for any fixed $\delta > 0$ there
  exists a sequence $(\lambda^{n,\delta})_{n\in\mathbb N}$ such that
  \[
    \lim_{n\rightarrow\infty}
    \| x^{\delta} - x^{\delta}_n \circ \lambda^{n,\delta} \|_{p ; [0,T+\delta]}
    +
    \| c^{\delta,\Pi} - c^{\delta,\Pi^n} \circ \lambda^{n,\delta}
    \|_{\infty ; [0,T+\delta]}
    =0.
  \]
  Then we can find some $N\in\mathbb N$ such that for all $n\geqslant N$ and
  for all $s\in[0,T+\delta]$,
  \begin{align*}
    \| x^{\delta}_n \circ \lambda^{n,\delta} \|_{p ; [s,(s+\gamma)\wedge(T+\delta)]}
    &\leqslant
    \| x^{\delta} \|_{p ; [s,(s+\gamma)\wedge(T+\delta)]}
    +
    \| x^{\delta} - x^{\delta}_n \circ \lambda^{n,\delta}
    \|_{p ; [0,T+\delta]}
    \\
    &\leqslant
    \| x^{\delta} \|_{p ; [s,(s+\gamma)\wedge(T+\delta)]}
    +
    \frac{\varepsilon}{4},
    \\
    \| c^{\delta,\Pi^n} \circ \lambda^{n,\delta} \|_{0 ; [s,(s+\gamma)\wedge(T+\delta)]}
    &\leqslant
    \| c^{\delta,\Pi} \|_{0 ; [s,(s+\gamma)\wedge(T+\delta)]}
    +
    2
    \| c^{\delta,\Pi} - c^{\delta,\Pi^n} \circ \lambda^{n,\delta}
    \|_{\infty ; [0,T+\delta]}
    \\
    &\leqslant
    \| c^{\delta,\Pi} \|_{0 ; [s,(s+\gamma)\wedge(T+\delta)]}
    +
    \frac{\varepsilon}{4}.
  \end{align*}
  Combining this with
  \eqref{uniform-cts-pvar-x}--\eqref{uniform-cts-pvar-c} proves the claim.
\end{proof}

It is well-known that (see {\cite{whitt_stochastic-process_2002}}) convergence in
Skorokhod metrics (such as $J1$ or $M1$) implies local uniform convergence at the continuity points of the limit. In the following lemma, we show that this property is independent of the choice of the "excursions" and extends to the Skorokhod metrics on decorated paths.

\begin{lemma}\label{Lemma-localuniform}
Let $p \in \{0\} \cup [1,\infty] $ and $x, x^1, x^2, \ldots \in
  \bar{\mathcal{D}}^p ([0,T])$. Assume that
  \[
    \lim_{n \to \infty} \alpha_{p ; [0,T]} (x^n,x)=0.
  \]
  Then, for every $t \in \mathcal{C} (x)$, there exists a $\gamma > 0$ such that 
  \[ \lim_{n \to \infty} \| x^n - x \|_{\infty ; [t - \gamma, t + \gamma]} =
     0. \]
  \begin{proof}
    We fix $t \in \mathcal C(x)$ and argue by contradiction. Suppose that there exist $\varepsilon > 0$
    and sequences $s, s_1, s_2, \ldots \in [t - \gamma, t + \gamma]$ such that $|
    s_n - s | \rightarrow 0$ and $| x^n (s_n) - x (s) | > \varepsilon$ for all
    $n \in \mathbb{N}$. By Lemma~\ref{alternative-a-metric}, for every $\delta > 0$
    there exists a sequence $(\lambda^{n, \delta})_{n \in \mathbb{N}}
    \subset \Lambda_{[0, T + \delta]}$ such that
    \begin{equation}
      \lim_{n \rightarrow \infty} | x^{\delta} (\lambda^{n, \delta}
      (\tau^{\delta, \Pi^n} (s_n))) - x^{\delta}_n (\tau^{\delta, \Pi^n}
      (s_n)) | + | c^{\delta, \Pi} (\lambda^{n, \delta} (\tau^{\delta, \Pi^n}
      (s_n))) - c^{\delta, \Pi^n} (\tau^{\delta, \Pi^n} (s_n)) | = 0.
      \label{proof-localuniform}
    \end{equation}
    In particular, $| c^{\delta, \Pi} (\lambda^{n, \delta}
    (\tau^{\delta, \Pi^n} (s_n))) - s_n | \rightarrow 0$, and hence
    \begin{equation}
      \lim_{n \rightarrow \infty} \lambda^{n, \delta} (\tau^{\delta, \Pi^n}
      (s_n)) = (c^{\delta, \Pi})^{- 1} (s) = \tau^{\delta, \Pi} (s),
      \label{proof-localuniform2}
    \end{equation}
    here we use the fact that the continuity of $x$ at $s$ ensures that $(c^{\delta, \Pi})^{- 1}(s)$ is well-defined. Finally, 
    \begin{align*}
      &| x^n (s_n) - x (s) |\\
      &\quad= | x^{n, \delta} (\tau^{\delta, \Pi^n} (s_n)) -
      x^{\delta} (\tau^{\delta, \Pi} (s)) |\\
      &\quad\leqslant | x^{n, \delta} (\tau^{\delta, \Pi^n} (s_n)) - x^{\delta}
      (\lambda^{n, \delta} (\tau^{\delta, \Pi^n} (s_n))) |\\
      &\qquad+ | x^{\delta}
      (\lambda^{n, \delta} (\tau^{\delta, \Pi^n} (s_n))) - x^{\delta}
      (\tau^{\delta, \Pi} (s)) |.
    \end{align*}
    The first term converges to $0$ by \eqref{proof-localuniform}, and
    the second converges to zero by \eqref{proof-localuniform2}. This contradicts $| x^n
    (s_n) - x (s) | > \varepsilon$.
  \end{proof}
\end{lemma}
Finally, we show that the Skorokhod metric $\alpha_{p ; [0, T]}$ defined in
\eqref{Skorokhod-metric} behaves well under restriction to $[t, T]$ for $t \in
\mathcal{C} (x)$. Here $\mathcal{C} (x) \subset [0,T]$ denotes the set of continuity points of $x$, i.e.
\begin{align*}
    \mathcal{C} (x) \assign \left\{ t \in [0, T] : x \text{ is continuous at time } t \right\}.
\end{align*}
Although this property is well known for classical Skorokhod topologies, its proof in the present setting is surprisingly delicate. The argument resonates with \cite[Theorem~16.2]{billingsley_convergence_1999} for the classical $J1$ topology, while containing the $J1$- and $M1$-type cases as special instances.

\begin{lemma}
  \label{restriction-alpha-metric}Let $p>0$ and $x, x^1, x^2, \ldots \in
  \bar{\mathcal{D}}^p ([0,T])$. Assume that $\lim_{n \to \infty} \alpha_{p ; [0,
  T]} (x^n, x) = 0$. Then, for all $t \in [0, T]$ such that $x(t)$ is constant, we have $\lim_{n \to \infty} \alpha_{p ; [t, T]} (x^n, x) =
  0$.
\end{lemma}

\begin{proof}
Fix $t \in \mathcal C(x)$ and $\delta > 0$.
{By Lemma~\ref{Lemma-localuniform},
$d_E(x^n(t)(0),x(t)(0))\to0$, so it remains to control other terms.}
Choose a countable set $\Pi \subset [0,T]$ containing all non-stationary 
points of $x,x^1,x^2,\dots$, and set $\Pi_t := \Pi \cap [t,T]$. Furthermore, we introduce the short notation $x^{\delta}_{[t]}$ for the
  $\delta$-extension $x^{\delta, \Pi_t} \in D ([t, T + \delta])$, analogously
  also $x^{n, \delta}_{[t]}$ and $c^{\delta}_{[t]} = c^{n, \delta}_{[t]} \in C
  ([t, T + \delta])$ for all $t \in [0, T]$.
  
  By Lemma~\ref{alternative-a-metric} there exists for any $\delta > 0$ a sequence $(\lambda^{n, \delta}_{[T]})_{n \in \mathbb{N}}
  \Lambda_{[0, T + \delta]}$ such that
  \begin{equation}
    \lim_{n \rightarrow \infty} \| x^{\delta}_{[T]} \circ \lambda^{n,
    \delta}_{[T]} - x^{n, \delta}_{[T]} \|_{p ; [0, T + \delta]} + \|
    c^{\delta}_{[T]} \circ \lambda^{n, \delta}_{[T]} - c^{\delta}_{[T]}
    \|_{\infty ; [0, T + \delta]} = 0. \label{T-convergence}
  \end{equation}
  The goal is to construct a sequence $(\lambda^{n, \delta}_{[t]})_n
  \subset \Lambda_{[t, T + \delta]}$ such that
  \[ \lim_{n \rightarrow \infty} \| x^{\delta}_{[t]} \circ \lambda^{n,
     \delta}_{[t]} - x^{n, \delta}_{[t]} \|_{p ; [t, T + \delta]} + \|
     c^{\delta}_{[t]} \circ \lambda^{n, \delta}_{[t]} - c^{\delta}_{[t]}
     \|_{\infty ; [t, T + \delta]} = 0. \]
  We start by deriving a useful estimate for $\lambda^{n, \delta}_{[T]}$. It
  follows from \eqref{T-convergence} and the continuity of $c^{\delta}_{[T]}$
  that for every $\varepsilon > 0$ there exists an $N \in \mathbb{N}$ such that
  for all $n \geqslant N$ it holds
  
  \begin{align}
    | c^{\delta}_{[T]} \circ \lambda^{n, \delta}_{[T]} (\tau^{\delta} (t)) -
    c^{\delta}_{[T]} (\tau^{\delta} (t)) | & < \varepsilon \nonumber\\
    | c^{\delta}_{[T]} \circ \lambda^{n, \delta}_{[T]} (\tau^{\delta} (t) + 1
    / n) - c^{\delta}_{[T]} (\tau^{\delta} (t) + 1 / n) | & < \varepsilon / 2,
    \nonumber\\
    c^{\delta}_{[T]} (\tau^{\delta} (t) + 1 / n) - c^{\delta}_{[T]}
    (\tau^{\delta} (t)) & < \varepsilon / 2, \nonumber
  \end{align}
  Combining them yields$| c^{\delta}_{[T]} \circ \lambda^{n, \delta}_{[T]}
  (\tau^{\delta} (t)) - t | < \varepsilon$ and $| c^{\delta}_{[T]} \circ
  \lambda^{n, \delta}_{[T]} (\tau^{\delta} (t) + 1 / n) - t | < \varepsilon$,
  which implies that
  \begin{equation}
    \lambda^{n, \delta}_{[T]} (\tau^{\delta} (t) + 1 / n) \leqslant
    \tau^{\delta} (t + \varepsilon) . \label{lambda-property}
  \end{equation}
  Notice that because we choose the same $\Pi$ for all $\delta$-extensions, we have
  that $c^{\delta}_{[t]}$ (resp. $x^{n, \delta}_{[t]}$ or $x^{\delta}_{[t]}$)
  compared to $c^{\delta}_{[T]} \mid_{[\tau^{\delta} (t), T + \delta]}$(resp.
  $x^{n, \delta}_{[T]}$ or $x^{\delta}_{[T]}$) only differs by the length of
  the ``excursions'', in particular, there exists some $\rho \in
  \Lambda_{[\nobracket \tau^{\delta} (t), T + \delta] ; [t, T + \delta]
  \nobracket}$ such that
  \begin{equation}
    c^{\delta}_{[T]} (s) = c^{\delta}_{[t]} \circ \rho (s), \quad x^{n,
    \delta}_{[T]} (s) = x^{n, \delta}_{[t]} \circ \rho (s), \quad
    x^{\delta}_{[T]} (s) = x^{\delta}_{[t]} \circ \rho (s) \quad \forall s \in
    [\tau^{\delta} (t), T + \delta] . \label{c-timechange}
  \end{equation}
  We distinguish between three different cases:
  \begin{enumerateroman}
    \item If $\lambda^{n, \delta}_{[T]} (\tau^{\delta} (t)) > \tau^{\delta}
    (t)$, then let $t_n = \tau^{\delta} (t) + 1 / n$;
    
    \item if $\lambda^{n, \delta}_{[T]} (\tau^{\delta} (t)) = \tau^{\delta}
    (t)$, then let $t_n = \tau^{\delta} (t)$;
    
    \item if $\lambda^{n, \delta}_{[T]} (\tau^{\delta} (t)) < \tau^{\delta}
    (t)$, then let $t_n = (\lambda^{n, \delta}_{[T]})^{- 1} (\tau^{\delta} (t)
    + 1 / n)$.
  \end{enumerateroman}
  Through this construction, we can ensure that $\lambda^{n, \delta}_{[T]} (t_n)
  \geqslant \tau^{\delta} (t)$ and $t_n \geqslant \tau^{\delta} (t)$. We
  define $\tilde{\lambda}^{n, \delta}_{[\tau^{\delta} (t)]} : [\tau^{\delta}
  (t), T + \delta] \rightarrow [\tau^{\delta} (t), T + \delta]$ to be equal to
  $\lambda^{n, \delta}_{[T]}$ on $[t_n, T + \delta]$, while on $[\tau^{\delta}
  (t), t_n]$ define $\tilde{\lambda}^{n, \delta}_{[\tau^{\delta} (t)]}$ to be
  the linear interpolation between $(\tau^{\delta} (t), \tau^{\delta} (t))$
  and $(t_n, \lambda^{n, \delta}_{[T]} (t_n))$. By construction we have
  $\tilde{\lambda}^{n, \delta}_{[\tau^{\delta} (t)]} \in
  \Lambda_{[\tau^{\delta} (t), T + \delta]}$. Finally, we introduce the
  reparametrization
  \begin{equation}
    \lambda^{n, \delta}_{[t]} = \rho  \circ \tilde{\lambda}^{n,
    \delta}_{[\tau^{\delta} (t)]} {\circ \rho^{- 1}}  \in \Lambda_{[t, T +
    \delta]} . \label{lambda-restricted}
  \end{equation}
  It holds by combining \eqref{c-timechange} and \eqref{lambda-restricted}
  that
  \begin{align}
    \| c^{\delta}_{[t]} \circ \lambda^{n, \delta}_{[t]} - c^{\delta}_{[t]}
    \|_{\infty ; [t, T + \delta]} = & \| c^{\delta}_{[T]} \circ
    \tilde{\lambda}^{n, \delta}_{[\tau^{\delta} (t)]} - c^{\delta}_{[T]}
    \|_{\infty ; [\tau^{\delta} (t), T + \delta]} ;  \label{t-convergence-c}\\
    \| x^{\delta}_{[t]} \circ \lambda^{n, \delta}_{[t]} - x^{n, \delta}_{[t]}
    \|_{p ; [t, T + \delta]} = & \| x^{\delta}_{[T]} \circ
    \tilde{\lambda}^{n, \delta}_{[\tau^{\delta} (t)]} - x^{n, \delta}_{[T]}
    \|_{p ; [\tau^{\delta} (t), T + \delta]} .  \label{t-convergence-x}
  \end{align}
  We only show the convergence of \eqref{t-convergence-x} to zero as $n
  \rightarrow \infty$. The convergence of \eqref{t-convergence-c} is simpler
  and can be derived analogously. The triangle inequality gives us
  
  \begin{align*}
    &\| x^{\delta}_{[T]} \circ \tilde{\lambda}^{n, \delta}_{[\tau^{\delta}
    (t)]} - x^{n, \delta}_{[T]} \|_{p ; [\tau^{\delta} (t), T + \delta]}
    \\
    &\quad\leqslant \| x^{\delta}_{[T]} \circ \tilde{\lambda}^{n,
    \delta}_{[\tau^{\delta} (t)]} - x^{\delta}_{[T]} \circ \lambda^{n,
    \delta}_{[T]} \|_{p ; [\tau^{\delta} (t), t_n]}\\
    & + \| x^{\delta}_{[T]} \circ \lambda^{n, \delta}_{[T]} - x^{n,
    \delta}_{[T]} \|_{p ; [\tau^{\delta} (t), t_n]}\\
    & + \| x^{\delta}_{[T]}
    \circ \tilde{\lambda}^{n, \delta}_{[\tau^{\delta} (t)]} - x^{n,
    \delta}_{[T]} \|_{p ; [t_n, T + \delta]} .
  \end{align*}
  The convergence of the latter two follows from the fact that
  $\tilde{\lambda}^{n, \delta}_{[\tau^{\delta} (t)]} = \lambda^{n,
  \delta}_{[T]}$ on $[t_n, T + \delta]$ and \eqref{T-convergence}. For the first term, we use the invariance of $p$-variation under 
reparametrization, as
  \begin{align*}
    &\| x^{\delta}_{[T]} \circ \tilde{\lambda}^{n, \delta}_{[\tau^{\delta}
    (t)]} - x^{\delta}_{[T]} \circ \lambda^{n, \delta}_{[T]} \|_{p ;
    [\tau^{\delta} (t), t_n]}\\
    &\quad= \| x^{\delta}_{[T]} - x^{\delta}_{[T]}
    \circ \lambda^{n, \delta}_{[T]} \circ (\tilde{\lambda}^{n,
    \delta}_{[\tau^{\delta} (t)]})^{- 1} \|_{p ; [\tau^{\delta} (t),
    \lambda^{n, \delta}_{[\tau^{\delta} (t)]} (t_n)]}\\
    &\quad= \| x^{\delta}_{[T]} \|_{p ; [\tau^{\delta} (t), \lambda^{n,
    \delta}_{[\tau^{\delta} (t)]} (t_n)]}\\
    &\qquad+ \| x^{\delta}_{[T]} \|_{p ;
    [\lambda^{n, \delta}_{[T]} (\tau^{\delta} (t)), \lambda^{n, \delta}_{[T]}
    (t_n)]} .
  \end{align*}
  Consider $n$ large enough such that $\tau^{\delta} (t) + 1 / n \leqslant
  \tau^{\delta} (t +)$ and study the three different cases of $t_n$
  separately.\\
  \textit{For $n$ satisfying (i)} and sufficiently large to yield
  \eqref{lambda-property}, it holds
  \begin{align}
    &\| x^{\delta}_{[T]} \circ \tilde{\lambda}^{n, \delta}_{[\tau^{\delta}
    (t)]} - x^{\delta}_{[T]} \circ \lambda^{n, \delta}_{[T]} \|_{p ;
    [\tau^{\delta} (t), t_n]}\nonumber\\
    &\quad\leqslant 2 \| x^{\delta}_{[T]} \|_{p ;
    [\tau^{\delta} (t), \lambda^{n, \delta}_{[\tau^{\delta} (t)]}
    (\tau^{\delta} (t) + 1 / n)]}\nonumber\\
    &\quad\leqslant 2 \| x^{\delta}_{[T]} \|_{p ;
    [\tau^{\delta} (t +), \nobracket \tau^{\delta} (t \nobracket +
    \varepsilon)]}, \nonumber
  \end{align}
  where for the last inequality, we used $\| x^{\delta}_{[T]} \|_{p ;
  [\tau^{\delta} (t), \nobracket \tau^{\delta} (t \nobracket +)]} = 0$ (since
  $x (t) \equiv \tmop{const}$).\\
  \textit{For $n$ satisfying ii)}, it holds trivially that $\| x^{\delta}_{[T]}
  \circ \tilde{\lambda}^{n, \delta}_{[\tau^{\delta} (t)]} - x^{\delta}_{[T]}
  \circ \lambda^{n, \delta}_{[T]} \|_{p ; [\tau^{\delta} (t), t_n]} = 0$.\\
  \textit{For $n$ satisfying iii)} , it holds that
  \begin{align}
    \| x^{\delta}_{[T]} \circ \tilde{\lambda}^{n, \delta}_{[\tau^{\delta}
    (t)]} - x^{\delta}_{[T]} \circ \lambda^{n, \delta}_{[T]} \|_{p ;
    [\tau^{\delta} (t), t_n]} & \leqslant 2 \| x^{\delta}_{[T]} \|_{p ;
    [\lambda^{n, \delta}_{[T]} (\tau^{\delta} (t)), \tau^{\delta} (t) + 1 /
    n]} \nonumber\\
    & \leqslant 2 \| x^{\delta}_{[T]} \|_{p ; [\lambda^{n, \delta}_{[T]}
    (\tau^{\delta} (t)), \tau^{\delta} (t)]} + 2 \| x^{\delta}_{[T]} \|_{p ;
    [\tau^{\delta} (t), \tau^{\delta} (t) + 1 / n]} . \nonumber
  \end{align}
  It follows from \eqref{T-convergence} that for every $\varepsilon > 0$ there
  exists an $N \in \mathbb{N}$ such that for all $n \geqslant N$ we have
  \[ c^{\delta}_{[T]} (\tau^{\delta} (t)) - c^{\delta}_{[T]} \circ
     \lambda^{n, \delta}_{[T]} (\tau^{\delta} (t)) = t - c^{\delta}_{[T]}
     \circ \lambda^{n, \delta}_{[T]} (\tau^{\delta} (t)) < \varepsilon, \]
  this implies $\lambda^{n, \delta}_{[T]} (\tau^{\delta} (t)) \geqslant
  \tau^{\delta} (t - \varepsilon)$.\\
  Combining the estimates for the three cases gives us the following generous
  bound: For any $\varepsilon > 0$, there exists an $N \in \mathbb{N}$ such
  that for all $n \geqslant N$ it holds
  \begin{align}
    \| x^{\delta}_{[T]} \circ \tilde{\lambda}^{n, \delta}_{[\tau^{\delta}
    (t)]} - x^{\delta}_{[T]} \circ \lambda^{n, \delta}_{[T]} \|_{p ;
    [\tau^{\delta} (t), t_n]} \leqslant & 2 \| x^{\delta}_{[T]} \|_{p ;
    [\tau^{\delta} (t +), \nobracket \tau^{\delta} (t \nobracket +
    \varepsilon)]} + 2 \| x^{\delta}_{[T]} \|_{p ; [\tau^{\delta} (t -
    \varepsilon), \tau^{\delta} (t)]} \nonumber\\
    & + 2 \| x^{\delta}_{[T]} \|_{p ; [\tau^{\delta} (t), \tau^{\delta} (t) +
    1 / n]} . \nonumber
  \end{align}
  Since the choice of $\varepsilon$ is arbitrary, we are done with the proof.
\end{proof}

\section{Main Results} \label{section-mainresult}
In this section, we clarify our notion of solution to RPDEs with discontinuous
drivers and state our first main results in Theorem~\ref{main-result}. We now
return to the RPDE \eqref{semilinear-RPDE}, which we restate for convenience:
\[
  d u_i + (\mathcal{L}_t u_i + {f_i}(t,x,u,\sigma^{\top} \nabla_x u_i)) \, dt + {h_i}(t,x,u) \diamond dW = 0, 
  \quad u(T,x) = \xi(x), \quad 1 \leqslant i \leqslant k.
\]
and for a smooth approximation $W^n$ of $W$, the corresponding approximating PDE
\begin{equation} 
  \frac{\partial u^n_i}{\partial t} +\mathcal{L}_t u^n_i + {f_i} (t, x, u^n,
  \sigma^{\top} \nabla_x u^n_i) + {h_i} (t, x, u^n) \dot{W}^n = 0, \quad u^n (T, x)
  = \xi (x), \quad 1 \leqslant i \leqslant k.
  \label{semilinear-smoothPDE-main}
\end{equation}
The RPDE \eqref{semilinear-RPDE} is understood in the Marcus sense:
at every jump time of $W$, the solution satisfies
\[
  u(t,x)=\varphi\bigl(-h(t,x,\cdot)\,\Delta W_t,u(t+,x),0\bigr),
\]
where $\varphi$ is defined through the ODE \eqref{Marusjump-ODE}. This suggests the
following lift of a c\`agl\`ad solution path to the space of decorated paths.
\begin{definition}
  \label{Def-Marcus-lift}
  Let $u\in D([0,T];C(\mathbb{R}^d;\mathbb{R}^k))$, let
  $W\in D([0,T];\mathbb{R}^e)$, and let $h$ be a vector field acting on the
  value variable of $u$. We define the \textbf{Marcus lift}
  $\kappa_{h,W}(u)$ by
  \[
    \bigl(\kappa_{h,W}(u)\bigr)(t)(s)(x)
    \assign
      \begin{cases}
        u(t)(x), & s=0,\\
        \varphi\bigl(h(t,x,\cdot)\,\Delta W_t,u(t+)(x),s\bigr),
        & s\in(0,1].
      \end{cases}
  \]
\end{definition}
In order to verify that the lift above is indeed a lift to the decorated path
space with values in $C(\mathbb{R}^d;\mathbb{R}^k)$, we record the following
lemma.
\begin{lemma}
  \label{Lemma-Marcus-lift}
  Let $u\in D([0,T];C(\mathbb{R}^d;\mathbb{R}^k))$ and
  $W\in D([0,T];\mathbb{R}^e)$. Assume that the vector field $h(t,x,\cdot)$ is
  complete, continuous in $(t,x)$, and uniformly Lipschitz continuous in the
  value variable. Then
  \[
    \kappa_{h,W}(u)
    \in \mathfrak{D}([0,T];C(\mathbb{R}^d;\mathbb{R}^k)).
  \]
\end{lemma}
\begin{proof}
  For fixed $x\in\mathbb{R}^d$, set $h^x_t\assign h(t,x,\cdot)$. Then
  $\kappa_{h^x,W}$ is precisely the Marcus embedding from
  Example~\ref{Marcus-embedding}. Hence, pointwise in $x$, the expression in
  Definition~\ref{Def-Marcus-lift} gives a decorated path. It remains to check
  that the decoration takes values in $C(\mathbb{R}^d;\mathbb{R}^k)$. For
  fixed $t\in[0,T]$ and $s\in(0,1]$, the map
  \[
    x\longmapsto
    \varphi\bigl(h(t,x,\cdot)\Delta W_t,u(t+)(x),s\bigr)
  \]
  is continuous by the continuous dependence of ODE solutions on parameters,
  using the continuity of $x\mapsto u(t+)(x)$ and of $h$ in $(t,x)$ together
  with the uniform Lipschitz bound in the value variable; see
  \cite[Theorem~III.13.II]{walter_ordinary_1998}. The case $s=0$ follows from
  $u(t)\in C(\mathbb{R}^d;\mathbb{R}^k)$. Thus the lift belongs to
  $\mathfrak{D}([0,T];C(\mathbb{R}^d;\mathbb{R}^k))$.
\end{proof}
We next specify the space of discontinuous driving signals $W$. In the spirit
of Wong--Zakai approximation, the natural path space should allow approximation
by smooth paths. To this end, we denote by
\label{smooth-path-closures}
$C^{0,p}(I;\mathbb{R}^e)$ the closure of the smooth paths in
$C^p(I;\mathbb{R}^e)$, the space of continuous paths of finite $p$-variation,
with respect to the $p$-variation metric. Similarly, we define its discontinuous
counterpart $D^{0,p}(I;\mathbb{R}^e)$ as the closure of the smooth paths in
$D^p(I;\mathbb{R}^e)$ with respect to the $p$-variation-type $M1$ metric
$\sigma^p_{M1}$ introduced in \cite{chevyrev_canonical_2019}; see also
p.~\pageref{embedding-i}. For every $\varepsilon>0$, these spaces satisfy the strict
inclusions
\[
  C^{0,p}(I;\mathbb{R}^e) \subsetneq C^p(I;\mathbb{R}^e)
  \subsetneq C^{0,p+\varepsilon}(I;\mathbb{R}^e),
  \qquad
  D^{0,p}(I;\mathbb{R}^e) \subsetneq D^p(I;\mathbb{R}^e)
  \subsetneq D^{0,p+\varepsilon}(I;\mathbb{R}^e).
\]
Thus, at the cost of an arbitrarily small loss of variation regularity, every
continuous (respectively, c\`agl\`ad) path of finite $p$-variation admits a
sequence of smooth approximations converging in the
$(p+\varepsilon)$-variation metric (respectively, the
$(p+\varepsilon)$-variation-type $M1$ metric).

We write $\mathrm{BUC}(\mathbb{R}^d;\mathbb{R}^k)$ for the space of bounded
uniformly continuous functions from $\mathbb{R}^d$ to $\mathbb{R}^k$,
equipped with the uniform norm.

We are now ready to define robust viscosity solutions. For readers familiar
with the continuous setting (see \cite{lions_fully_1998, caruana_partial_2009,
caruana_rough_2011, friz_rough_2014, diehl_backward_2017,
liang_multidimensional_2023}), we note that the following definition is
compatible with the classical formulation whenever
$W\in C^{0,q}([0,T];\mathbb{R}^e)$.
\begin{definition}\label{def-robust-viscosity-solution}
  Let $q<2$, $\xi\in \mathrm{BUC}(\mathbb{R}^d;\mathbb{R}^k)$, and
  $W\in D^{0,q}([0,T];\mathbb{R}^e)$. We call
  \[
    u^W=(u_i^W)_{i=1,\ldots,k}\assign\mathcal{U}(\xi,W)
  \]
  a
  \textbf{robust (viscosity) solution} to the RPDE
  \eqref{semilinear-RPDE} with terminal condition $\xi$ and driven by $W$
  if
  \[
    \mathcal{U}\colon \mathrm{BUC}(\mathbb{R}^d;\mathbb{R}^k)
    \times D^{0,q}([0,T];\mathbb{R}^e)
    \longrightarrow D([0,T];C(\mathbb{R}^d;\mathbb{R}^k))
  \]
  satisfies:
  \begin{enumerateroman}
    \item For every $W\in C^\infty([0,T];\mathbb{R}^e)$, we have that $u^W = \mathcal{U}(\xi,W)$ 
    is the unique viscosity
    solution to \eqref{semilinear-RPDE}, with $\diamond dW$ replaced by $\dot{W} dt$,
    in the classical sense of Definition~\ref{Def-viscosity}.
    \item For each
    $(\xi,W)\in \mathrm{BUC}(\mathbb{R}^d;\mathbb{R}^k)
    \times D^{0,q}([0,T];\mathbb{R}^e)$, set $\mathbf{u}^{W}
      \assign \kappa_{-h,W}(u^W)$.
    The map $(\xi,W)\mapsto \mathbf{u}^{W}$ is continuous from
    \[
      \mathrm{BUC}(\mathbb{R}^d;\mathbb{R}^k)
      \times
      \bigl(D^{0,q}([0,T];\mathbb{R}^e),\sigma^q_{M1}\bigr)
      \quad\text{to}\quad
      \bigl(\mathfrak{D}([0,T];C(\mathbb{R}^d;\mathbb{R}^k)),
      \alpha_\infty\bigr).
    \]
    
  \end{enumerateroman}
\end{definition}

\begin{remark}\label{remark-unique-continuous-extension}
  The restriction to $D^{0,q}([0,T];\mathbb R^e)$ makes uniqueness intrinsic
  to Definition~\ref{def-robust-viscosity-solution}. Indeed, if two solution
  maps satisfy (i) and (ii), then they agree on smooth drivers by (i), and hence on all of $D^{0,q}$, since smooth paths are dense in
$D^{0,q}$ and both maps are continuous by (ii). Smooth paths are not dense in the whole space $D^q$ in
  the $\sigma^q_{M1}$ topology, so the two conditions alone would not
  determine a solution map on $D^q\setminus D^{0,q}$.  
  The rough FBSDE representation established below nevertheless selects a canonical and
  unique solution for every $W\in D^q$. We do not include this additional
  generality in the definition, since its uniqueness would then rely on the
  subsequent rough FBSDE representation rather than following directly from the definition.
\end{remark}

The next proposition shows that the natural requirements (i) and (ii) yield a
Wong--Zakai-type approximation result of the kind illustrated in
Example~\ref{example} and, in turn, force the Marcus jump structure of the
RPDE. This also justifies the notation $\diamond\,dW$ used in
\eqref{semilinear-RPDE}.
\begin{proposition}\label{prop-robust-solution-marcus-jumps}
  Let $u^W=\mathcal{U}(\xi,W)$ be a robust viscosity solution to
  \eqref{semilinear-RPDE}. For any sequence
  $(W^n)_{n\in\mathbb{N}}\subset C^\infty([0,T];\mathbb{R}^e)$ such that
  \[
    \sigma^q_{M1}(W^n,W)\longrightarrow0,
    \qquad\text{equivalently}\qquad
    \alpha_q(\jmath W^n,\jmath W)\longrightarrow0,
  \]
  the viscosity solutions
  $u^n$ to
  \eqref{semilinear-smoothPDE-main} satisfy
  \[
    \alpha_\infty
    \bigl(\jmath u^n,\kappa_{-h,W}(u^W)\bigr)\longrightarrow 0
  \]
  in $\mathfrak{D}([0,T];C(\mathbb{R}^d;\mathbb{R}^k))$. In particular, for
  every $t\in[0,T]$ and every $x\in\mathbb{R}^d$,
  {
  \[
    \Delta u^W_t(x)
    \assign u^W(t+)(x)-u^W(t)(x)
    =
    u^W(t+)(x)
    -\varphi\bigl(-h(t,x,\cdot)\,\Delta W_t,u^W(t+)(x),0\bigr).
  \]
  }
\end{proposition}

\begin{proof}
  The asserted convergence follows immediately from (ii), while (i)
  identifies $\mathcal{U}(\xi,W^n)$ with the unique viscosity solution $u^n$
  of the corresponding smooth PDE. By
  Lemma~\ref{alternative-a-metric}, for every $\delta>0$ there exist
  reparametrizations $\lambda^{n,\delta}$ such that
  \[
    \bigl\|u^\delta-u^n\circ\lambda^{n,\delta}\bigr\|_\infty
    \longrightarrow0,
  \]
  where $u^\delta$ is the $\delta$-extension of
  $\kappa_{-h,W}(u^W)$. Since $u^n\circ\lambda^{n,\delta}$ is continuous,
  so is $u^\delta$. In particular, every inserted excursion is continuous.
  Thus, by continuity at its initial point and
  Definition~\ref{Def-Marcus-lift}, for every $t\in[0,T]$,
  \[
    u^W(t)(x)
    =
    \lim_{s\searrow0}
    \varphi\bigl(-h(t,x,\cdot)\,\Delta W_t,u^W(t+)(x),s\bigr)
    =
    \varphi\bigl(-h(t,x,\cdot)\,\Delta W_t,u^W(t+)(x),0\bigr),
  \]
  which is equivalent to the asserted Marcus jump relation.
\end{proof}

\begin{remark}\label{remark-marcus-lift-not-circular}
  The preceding proposition shows that, in the discontinuous setting, every robust
  viscosity solution necessarily follows the Marcus jump dynamics. At first
  sight, this conclusion may appear circular, since the Marcus lift of $u^W$
  already enters the continuity requirement in (ii). The role of this lift,
  however, is not to force the limiting solution to have Marcus jumps. Rather,
  the solution has Marcus jump structure, and the lift records the
  corresponding excursions so that the solution map is continuous. Indeed,
  the pointwise limit of the smooth approximations at the continuity points
  of $W$ does not depend on the Marcus lift. This limit coincides with the
  projection of the decorated Skorokhod limit from the space of decorated paths 
  onto the standard path space and
  therefore has the same Marcus jump structure.
\end{remark}

We now fix a filtered probability space $(\Omega,
\mathcal{F}, (\mathcal{F}^B_t)_{t \in [0, T]}, \mathbb{P})$, which supports a
$d$-dimensional Brownian motion $B$. The filtration $(\mathcal{F}^B_t)_{t
\in [0, T]}$ is given by the usual filtration of $B$.

Throughout this paper, we work under various combinations of the following assumptions. The assumptions marked with primes are stronger than their unprimed counterparts. In general, weaker assumptions on the SDE require stronger conditions on the RBSDE, and vice versa. \\

\noindent\hypertarget{A:W}{}\textbf{Assumption $(W)$:} 
\textit{The rough path $W$ belongs to ${D^{0,q}}([0,T]; \mathbb{R}^e)$ for some $q \in [1,2)$, and we fix $p>2$ such that $\frac{1}{p} + \frac{1}{q} > 1$. Throughout the paper, we work on the space of c\`agl\`ad paths, as the PDEs are formulated backward in time.}

\medskip

\noindent\hypertarget{A:b-sigma}{}\textbf{Assumption $(b,\sigma)$:} 
\textit{There exists a constant $C_{b,\sigma} > 0$ such that the functions 
$b : [0,T]\times \mathbb{R}^d \to \mathbb{R}^d$ and 
$\sigma : [0,T]\times \mathbb{R}^d \to \mathbb{R}^{d\times d}$ are jointly continuous in $(t,x)$ and, for all $t,x,x'$,}
\begin{align*}
|b(t,x) - b(t,x')| + |\sigma(t,x) - \sigma(t,x')|
\le C_{b,\sigma} |x-x'|.
\end{align*}

\medskip

\noindent\hypertarget{A:b-sigma'}{}\textbf{Assumption $(b',\sigma')$:} 
\textit{There exists a constant $C_{b,\sigma} > 0$ such that the functions 
$b, \sigma$ are jointly continuous in $(t,x)$ and, for all $t,x,x'$,}
\begin{align*}
|b(t,x)| + |\sigma(t,x)| &\le C_{b,\sigma}, \\
|b(t,x) - b(t,x')| + |\sigma(t,x) - \sigma(t,x')|
&\le C_{b,\sigma} |x-x'|.
\end{align*}

\medskip

\noindent\hypertarget{A:xi-f}{}\textbf{Assumption $(\xi,f)$:} 
\textit{The function
$\xi:\mathbb{R}^d\to\mathbb{R}^k$ is bounded and uniformly continuous, and
$f:[0,T]\times\mathbb{R}^d\times\mathbb{R}^k
\times\mathbb{R}^{{k}\times d}\to\mathbb{R}^k$
is continuous. Moreover, there exists a constant $C_{\xi,f}>0$ such that,
for all $t,x,x',y,y',z,z'$,}
\begin{align*}
|\xi(x)| + |f(t,x,0,0)| &\le C_{\xi,f}, \\
|f(t,x,y,z) - f(t,x',y',z')|
&\le C_{\xi,f} (|x-x'| + |y-y'| + |z-z'|).
\end{align*}
\textit{Moreover, for each $i=1,\dots,{k}$, the component $f_i(t,x,y,z)$ depends only on the $i$-th row of $z$.}

\medskip

\noindent\hypertarget{A:h}{}\textbf{Assumption $(h)$:} 
\textit{There exist a constant $C_h > 0$ and an exponent $\rho\in(2,p)$ such that for
$h : [0,T]\times \mathbb{R}^d \times \mathbb{R}^k \to \mathcal{L}(\mathbb{R}^e,\mathbb{R}^k)$,
the paths \(h_\cdot(x,y)\) and
\(\mathrm D_yh_\cdot(x,y)\) are continuous for every \((x,y)\), and}
\[
\sup_{t,x} |h_t(x,\cdot)|_{C^2_b}
+
\sup_{t,x,y}
\bigl(
|\mathrm{D}_xh_t(x,y)|
+|\mathrm{D}_x^2h_t(x,y)|
+|\mathrm{D}_x\mathrm{D}_yh_t(x,y)|
\bigr)
\le C_h,
\]
\[
\|h\|_{\rho;[0,T];C_b(\mathbb R^d\times\mathbb R^k)}
+\|\mathrm D_xh\|_{\rho;[0,T];C_b(\mathbb R^d\times\mathbb R^k)}
+\|\mathrm D_yh\|_{\rho;[0,T];C_b(\mathbb R^d\times\mathbb R^k)}
\le C_h.
\]
Here, $C_b(\mathbb R^d\times\mathbb R^k)$ denotes the relevant space of
bounded continuous operator-valued maps, equipped with the supremum norm
over $(x,y)$. For $\mathrm D_xh$ and $\mathrm D_yh$, the absolute value
denotes the corresponding operator norm.

\medskip

\noindent\hypertarget{A:h'}{}\textbf{Assumption $(h')$:} 
\textit{There exist a constant $C_h > 0$ and an exponent $\rho\in(2,p)$ such that for
$h : [0,T]\times \mathbb{R}^k \to \mathcal{L}(\mathbb{R}^e,\mathbb{R}^k)$
(independent of $x$), the paths \(h_\cdot(y)\) and
\(\mathrm D_yh_\cdot(y)\) are continuous for every \(y\), and}
\[
\sup_{t \in [0,T]} |h_t(\cdot)|_{C^2_b} \le C_h, \quad
\|h\|_{\rho;[0,T];C_b(\mathbb R^k)} \le C_h, \quad
\|\mathrm{D}_y h\|_{\rho;[0,T];C_b(\mathbb R^k)} \le C_h.
\]

\medskip

\noindent\hypertarget{A:hC}{}\textbf{Assumption $(h^C)$:} 
\textit{The paths $h_\cdot(x,y)$ and $\mathrm{D}h_\cdot(x,y)$ are uniformly continuous in $t$, uniformly in $(x,y)$.}

\medskip

When referring to these assumptions, we use the tuple notation: 
Assumptions $(W,b,\sigma,\xi,f)$ are satisfied if and only if
Assumptions $(W)$, $(b,\sigma)$, and $(\xi,f)$ are satisfied.

\begin{theorem}
  \label{main-result}Suppose that either Assumptions 
$(\hyperlink{A:W}{W}, \hyperlink{A:b-sigma}{b,\sigma}, 
\hyperlink{A:xi-f}{\xi,f}, \hyperlink{A:h'}{h'}, \hyperlink{A:hC}{h^C})$
or Assumptions \linebreak
$(\hyperlink{A:W}{W}, \hyperlink{A:b-sigma'}{b',\sigma'}, 
\hyperlink{A:xi-f}{\xi,f},
\hyperlink{A:h}{h}, \hyperlink{A:hC}{h^C})$ hold.
Then the RPDE
\eqref{semilinear-RPDE} admits a unique robust viscosity solution
$u \in D
  ([0, T] ; C (\mathbb{R}^d; \mathbb{R}^k))$, such that $u (t) \in \mathrm{BUC}(\mathbb{R}^d; \mathbb{R}^k)$ for all $t \in [0,
  T]$. Moreover, the solution satisfies the following
  properties:
  \begin{enumeratealpha}
    \item \textbf{Stochastic representation.} For all $(t,x) \in [0,T] \times \mathbb{R}^d$,
\[
u(t,x) = Y_t^{t,x},
\]
where $Y^{t,x}$ is the solution of the rough FBSDE 
\eqref{rough-FBSDE}.
    
    \item \textbf{Backward flow property.} To emphasize the dependence on the terminal
    condition $u (T, \cdot) = \xi(\cdot)$, we write $(t,x) \mapsto u (t ; T, \xi) (x)$ for the
    unique robust viscosity solution to RPDE \eqref{semilinear-RPDE}. Then,
    \[
      u (r ; t, \xi) (x) = u (r ; s, u (s ; t, \xi)) (x)
    \]
    for all $0 \leqslant r \leqslant s \leqslant t \leqslant T$ and $x \in
    \mathbb{R}^d$.
  \end{enumeratealpha}
\end{theorem}
We prove the above theorem in {Section~\ref{section-proof}}. But first, we revisit Example~\ref{example} from the introduction.

\subsection{Example \ref{example} revisited} \label{section-example-revisited}
Recall that in Example~\ref{example} we established that the correct candidate for the robust viscosity solution to \eqref{example-RPDE} is given by
\[
u(t,x) = e^{t-T} \sin(x)\, (\cos(W_t),\; -\sin(W_t))^{\top}.
\]
For any sequence of smooth approximations $(W^n)_{n \in \mathbb{N}}$ converging pointwise to the c\`agl\`ad path $W$ at the continuity points of $W$, the corresponding approximating PDE has solution
\[
u^n(t,x) = e^{t-T} \sin(x)\, (\cos(W^n_t),\; -\sin(W^n_t))^{\top},
\]
and $u^n$ converges pointwise to $u$ at the continuity points of $W$. However, such pointwise convergence is not satisfactory, especially when we later consider L\'evy processes, whose sample paths may have infinite activity and typically different jump times. This makes it necessary to work with a Skorokhod-type topology.

For simplicity, and to emphasize that the need for a Skorokhod-type metric on decorated paths is not due to a lack of path regularity, let us assume that $W = \frac{\pi}{2}\,\mathds{1}_{\{t\mathrel{\leq}0.5\}}$, and that $(W^n)_{n\in \mathbb{N}}$ are continuous linear approximations; see Figure~\ref{fig:example-rpde}. While the convergence of $W^n$ to $W$ can already be described by the classical Skorokhod $M_1$ metric, the convergence of $u^n$ to $u$ is no longer captured by $M_1$ and really requires the more general framework of Skorokhod-type metrics on decorated paths\footnote{We refer the reader to \cite[Example 1.4]{chevyrev_superdiffusive_2020} for a similar example in the case of ODEs.}.

This is visible in the plots. At $x=\pi/2$, even for large $n$, the graph of $u^n(t,\pi/2)$ near the jump time $t=0.5$ follows a curved path along the rotating field $\mathcal{R}$. This indicates that the limiting object in the space of decorated paths should exhibit a curved excursion at $t=0.5$, rather than a linear excursion, as would be implied by convergence in the classical Skorokhod $M_1$ topology.

\begin{figure}[H]
    \centering
    \begin{minipage}[t]{0.35\textwidth}
        \centering
        \includegraphics[width=\textwidth]{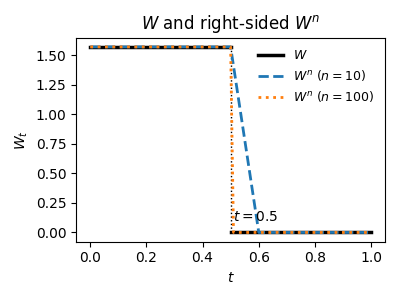}
    \end{minipage}
    \hfill
    \begin{minipage}[t]{0.64\textwidth}
        \centering
        \includegraphics[width=\textwidth]{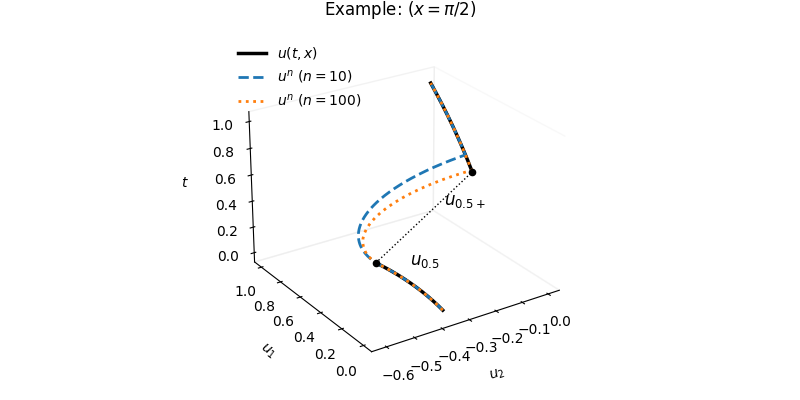}
    \end{minipage}
    \caption{c\`agl\`ad path $W$ and its continuous approximations $W^n$ (left). 
    Trajectories of $u^n(t,\pi/2)$ (for $n=10$ and $n=100$) and the limiting path $u(t,\pi/2)$ (right), illustrating the curved excursion near the jump time.}
    \label{fig:example-rpde}
\end{figure}

\section{Proof of Theorem \ref{main-result}} \label{section-proof}

The rough FBSDE is the central tool in the proof of
Theorem~\ref{main-result}. We therefore begin in
Section~\ref{Rough-FSBDEs} by establishing its well-posedness and stability,
as well as its Markovian representation and backward flow property. This
allows us to define the candidate solution directly by
\[
  u(t,x)\assign Y_t^{t,x},
\]
where $(Y^{t,x},Z^{t,x})$ solves the rough FBSDE
\eqref{rough-FBSDE}. Once properties~(i) and (ii) in the definition of a
robust viscosity solution have been verified, the stochastic representation
in Theorem~\ref{main-result} follows immediately from this construction,
while its backward flow property is already contained in
Proposition~\ref{Prop-RFBSDE-Markov-flow}.

In Section~\ref{Chap-ApproxPDE}, we verify property~(i): when $W$ is smooth,
the function $u(t,x)=Y_t^{t,x}$ is the unique viscosity solution of the RPDE
restricted to
the smooth driver $W$. For this purpose, we work within the framework of
multidimensional viscosity solutions introduced in
\cite{decreusefond_backward_1998} and complement the existing existence
results with a simple proof of uniqueness in
Theorem~\ref{multidimentional-visco}.

We turn to property~(ii) in Section~\ref{Skorokhod-limit}. Rather than proving
continuity of the solution map directly, we first show that the Marcus lift of
$u(t,x)=Y_t^{t,x}$ is the decorated Skorokhod limit of the solutions
corresponding to smooth drivers. By
Proposition~\ref{prop-robust-solution-marcus-jumps}, this implies in particular that
the $\delta$-extension of $u$ is continuous and that $u$ follows the Marcus
jump dynamics. This preliminary step is needed because the continuity
argument uses continuity of the $\delta$-extensions of the approximating
solutions, which is not available a priori when their drivers are
discontinuous. Section~\ref{Skorokhod-limit} provides precisely this input by
identifying each such extension as a limit of solutions corresponding to
smooth drivers. With this in hand, we reuse the same argument in
Section~\ref{continuity-solution-map} to prove property~(ii).

\subsection{Rough FBSDEs }\label{Rough-FSBDEs}

In this subsection, we develop a solution theory for the rough FBSDE
\eqref{rough-FBSDE}. Our main goal is to adapt the purely backward results of
\cite{becherer_rough_2026} to the forward--backward setting. We establish
well-posedness and stability of such rough FBSDEs in
Propositions~\ref{Prop-RFBSDE-wellposedness},
\ref{prop:uniform-pairwise-stability}, and
\ref{Prop-RFBSDE-stability}. Their proofs require estimates for the forward
SDE and for the random coefficients obtained by composing $f$ and $h$ with
its solution. For better readability, these estimates are collected in
Appendix~\ref{appendix-forward-SDE}. We also show that standard results from
the Markovian FBSDE theory continue to hold in the presence of a
discontinuous Young driver. In particular, the Markovian representation of
$Y$ and the backward flow property are collected in
Proposition~\ref{Prop-RFBSDE-Markov-flow}.

The stability analysis below is carried out in Skorokhod-type metrics on the
space of decorated paths. This requires inserting additional time into the rough
FBSDE. It is therefore convenient to work with a slightly more general
clocked rough FBSDE that contains both \eqref{rough-FBSDE} and its
time-extended version.

Following the setup of \cite[Section~2]{becherer_rough_2026}, fix $S>0$ and
a filtered probability space
$(\Omega,\mathcal F,(\mathcal F_t^M)_{t\in[0,S]},\mathbb P)$ supporting a
$d$-dimensional continuous martingale $M$. The filtration is the usual
filtration of $M$. We assume that $M=B\circ c$, where $B$ is a
$d$-dimensional Brownian motion and
$c:[0,S]\to[0,\infty)$ is deterministic, continuous and non-decreasing,
with $c_0=0$. In particular, $M$ has the martingale representation property:
for every square-integrable $(\mathcal F_t^M)$-martingale $L$, there exists a
predictable process $H\in L^2(dc\otimes\mathbb P)$ such that
\[
L_t=L_0+\int_0^t H_r\,dM_r.
\]
{For the rough FBSDE results, let
$W\in D^q([0,S];\mathbb R^e)$, $q\in[1,2)$,
and fix $p>2$ such that $\frac{1}{p}+\frac{1}{q}>1$. Notice that unlike in the 
rough PDE results later, we do not require $W$ to be in $D^{0,q}$.}
For $t\in[0,S]$ and
$\zeta\in L^2(\Omega,\mathcal F_t^M;\mathbb R^d)$, we consider the rough FBSDE
\begin{equation}
\label{extended-rough-FBSDE}
\begin{aligned}
X_s^{t,\zeta}
&=\zeta+\int_t^s b(r,X_r^{t,\zeta})\,dc_r
  +\int_t^s\sigma(r,X_r^{t,\zeta})\,dM_r,\\
Y_s^{t,\zeta}
&=\xi(X_S^{t,\zeta})
  +\int_s^S f(r,X_r^{t,\zeta},Y_r^{t,\zeta},Z_r^{t,\zeta})\,dc_r\\
&\quad
  +\int_s^S h(r,X_r^{t,\zeta},Y_r^{t,\zeta})\diamond dW_r
  -\int_s^S Z_r^{t,\zeta}\,dM_r,
\qquad s\in[t,S].
\end{aligned}
\end{equation}
The original rough FBSDE \eqref{rough-FBSDE} is recovered by taking $S=T$,
$c_t=t$, $M=B$, and $\zeta=x$. Its time-extended version is obtained
by taking $S=T+\delta$, $c=c^\delta$, and $M=B^\delta$, together with
the time-extended driver and coefficients. This construction is explained
after Proposition~\ref{prop:uniform-pairwise-stability}; see
\eqref{pairwise-time-stretched-FBSDE}.

We next recall the solution space for rough BSDEs in this setting.
\begin{definition}
  \label{YZ-statespace}For $p > 2$, define $\mathcal{B}_p ([0, S],
  \mathbb{R}^k)$ to be the space of adapted
  c{\`a}gl{\`a}d processes $Y : \Omega \times [0, S] \rightarrow \mathbb{R}^k$
  with
  \[ \|Y\|_{p, 2 ; [0, S]} \assign \esssup_{\omega \in \Omega} \sup_{t \in [0, S]} \mathbb{E}_t [\|Y\|^2_{p ; [t, S]}]^{1 / 2} <
     \infty . \]
  Denote by $\tmop{BMO} ([0, S], \mathbb{R}^{k \times d}) $ the space of all progressively measurable $Z : \Omega \times [0,
  S] \rightarrow \RR^{k \times d}$ with
  \[ \|Z\|_{\tmop{BMO} ; [0, S]} \assign \esssup_{\omega \in \Omega} \sup_{t \in [0, S]} \mathbb{E}_t \bigg[ \int_t^S |Z_r |^2
     \tmop{dc}_r \bigg]^{1 / 2} < \infty, \]
  where $| Z_r |$ denotes the Frobenius norm $| Z_r | \assign \sqrt{\tmop{tr} (Z_r^{\top} Z_r)}$.
\end{definition}

We can now state the well-posedness result for the rough FBSDE. We consider two sets of assumptions. Under stronger assumptions on the SDE, one obtains bounds for the classical solution $X$ in the $\|\cdot\|_{p,2}$ norm, which is essential when $h$ depends on $x$. Under weaker assumptions, such bounds are no longer available, and one therefore restricts to the case where $h$ is independent of $x$.

Although deterministic initial states are sufficient for the rough FBSDE
representation of the RPDE, we allow the forward SDE, and consequently the
rough FBSDE, to start at any random variable
$\zeta\in L^2(\Omega,\mathcal F_t^M;\mathbb R^d)$. This extension is needed
below for stability with respect to the initial state and to justify restarting
the FBSDE at the random state $X_v^{t,x}$, which in turn yields the Markovian
representation and the backward flow property. The corresponding forward-SDE
estimates are collected in Appendix~\ref{appendix-forward-SDE}.

\begin{proposition}
  \label{Prop-RFBSDE-wellposedness}Let $W\in D^q([0,S];\mathbb R^e)$ with $q\in[1,2)$.
Suppose that the coefficients on $[0,S]$
satisfy either Assumptions
$(\hyperlink{A:b-sigma}{b}, \hyperlink{A:b-sigma}{\sigma},
\hyperlink{A:xi-f}{\xi}, \hyperlink{A:xi-f}{f}, \hyperlink{A:h'}{h'})$
or Assumptions
$(\hyperlink{A:b-sigma'}{b'}, \hyperlink{A:b-sigma'}{\sigma'},
\hyperlink{A:xi-f}{\xi}, \hyperlink{A:xi-f}{f}, \hyperlink{A:h}{h})$
with the same constants, and let $p>2$ satisfy
$\frac{1}{p}+\frac{1}{q}>1$. Then, for
every $t\in[0,S]$ and every
$\zeta\in L^2(\Omega,\mathcal F_t^M;\mathbb R^d)$, the rough FBSDE
\eqref{extended-rough-FBSDE} on $[t,S]$
admits a unique solution
$(X^{t,\zeta},Y^{t,\zeta},Z^{t,\zeta})$. Moreover,
  \begin{equation}
    \sup_{\substack{t\in[0,S]\\
      \zeta\in L^2(\Omega,\mathcal F_t^M;\mathbb R^d)}}
    \left(
    \|Y^{t,\zeta}\|_{p,2;[t,S]}
    +\|Z^{t,\zeta}\|_{\mathrm{BMO};[t,S]}
    \right)<
    \infty . \label{bound-XY-p2}
  \end{equation}
  If Assumption $(b',\sigma')$ holds, then also
  \begin{equation}
    \sup_{\substack{t\in[0,S]\\
      \zeta\in L^2(\Omega,\mathcal F_t^M;\mathbb R^d)}}
    \|X^{t,\zeta}\|_{p,2;[t,S]}<
    \infty . \label{boundXp2}
  \end{equation}
\end{proposition}

\begin{proof}
  The forward equation has a unique strong solution by
  Proposition~\ref{Prop-random-initial-SDE-estimates}. We first prove the
  estimate \eqref{boundXp2} under Assumptions $(b',\sigma')$. For
  $s\in[t,S]$, boundedness of the coefficients gives
  \begin{eqnarray*}
    \mathbb{E}_s \left[ \left\| \int_s^{\cdot} b (v, X^{t,\zeta}_v) dc_v
    \right\|_{p ; [s, S]}^2 \right]^{1 / 2} & \leqslant & \mathbb{E}_s \left[
    \left( \int_s^S | b (v, X^{t,\zeta}_v) | dc_v \right)^2 \right]^{1 / 2}
    \leqslant C_{b, \sigma} c_S,\\
    \mathbb{E}_s \left[ \left\| \int_s^{\cdot} \sigma (v, X^{t,\zeta}_v)
    dM_v \right\|_{p ; [s, S]}^2 \right]^{1 / 2} & \leqslant &
    C\mathbb{E}_s \left[ \int_s^S | \sigma (v, X^{t,\zeta}_v) |^2 dc_v
    \right]^{1 / 2} \leqslant C_{S,p} C_{b, \sigma} c_S^{1 / 2},
  \end{eqnarray*}
  where the second line follows from the Burkholder--Davis--Gundy inequality
  in \cite[Theorem~14.12]{friz_multidimensional_2010}. The constants do not
  depend on $t$ or $\zeta$, and hence \eqref{boundXp2} follows.\\
  Define $\xi^X : \Omega \rightarrow \mathbb{R}^k$, $f^X : [t, S] \times
  \Omega \times \mathbb{R}^k \times \mathbb{R}^{{k} \times d} \rightarrow
  \mathbb{R}^k$, $h^X : [t, S] \times \Omega \times \mathbb{R}^k
  \rightarrow \mathcal L(\mathbb R^e,\mathbb R^k)$ as
  \[
  \xi^X(\omega)\assign\xi(X_S^{t,\zeta}(\omega)),\qquad
  f^X(s,\omega,\cdot)\assign f(s,X_s^{t,\zeta}(\omega),\cdot),\qquad
  h^X(s,\omega,\cdot)\assign h(s,X_s^{t,\zeta}(\omega),\cdot).
\]
  We verify Assumption~A of \cite{becherer_rough_2026} for these data.
  Since $X^{t,\zeta}$ is adapted and continuous, all three data are
  measurable in the required sense. Assumption $(\xi,f)$ implies that
  $\xi^X$ is uniformly bounded, that $f^X$ is uniformly Lipschitz in
  $(y,z)$, and that $f^X(\cdot,0,0)$ is uniformly bounded. It also preserves
  the componentwise dependence of $f$ on the rows of $z$.

  It remains to verify condition~A.d for $h^X$. Under Assumption $(h')$,
  $h$ is independent of $x$, so that $h^X(s,\omega,y)=h(s,y)$, and the
  condition follows directly from $(h')$. Under Assumptions
  $(b',\sigma',h)$, estimate \eqref{boundXp2} and
  Lemma~\ref{Lemma-hX-Assumption-Ad}, applied on $[t,S]$, give the required
  conditional $p$-variation bounds for $h^X$ and $\mathrm D_yh^X$, uniformly
  in $t$ and $\zeta$. Thus all parts of Assumption~A hold, and
  {\cite[Theorems~3.2 and~3.5]{becherer_rough_2026}} {yield} existence and uniqueness of
  the backward equation as well as \eqref{bound-XY-p2}. Together with
  uniqueness of the forward equation, this proves the result.
\end{proof}

The next proposition is one of the main technical lemmas of this section.
The general formulation allows the time changes used in the proof of
Theorem~\ref{theorem-Skorokhod}. For the present subsection,
including the stability argument used in the following proposition, only the
special case $r_n=r$ and $W^n=W$ is needed. In that case, the result reduces
to the standard stability of the rough FBSDE with respect to the initial
condition of the forward SDE.

\begin{proposition}\label{prop:uniform-pairwise-stability}
Let $S>0$, let $q\in[1,2)$ and $p>2$ satisfy
$\frac{1}{p}+\frac{1}{q}>1$, and, for every
$n\in\mathbb N$, let $r_n\in[0,S]$. Let
$c^n\colon[0,S]\to[0,\infty)$ be deterministic, continuous and non-decreasing
with $c_0^n=0$, and let $M^n=B\circ c^n$ be a time-changed
Brownian motion equipped with its usual filtration. Then $M^n$ has the
martingale representation property in this filtration. {Let
$W^n$ be a continuous path of finite $q$-variation.} Assume that
\begin{equation}\label{eq:uniform-controls-pairwise}
 \sup_{n\in\mathbb N}c^n_S<\infty,
 \qquad
 \sup_{n\in\mathbb N}\lVert W^n\rVert_{q;[0,S]}<\infty.
\end{equation}
Suppose that either Assumptions
$(\hyperlink{A:b-sigma}{b,\sigma},\hyperlink{A:xi-f}{\xi,f},
\hyperlink{A:h'}{h'})$ or Assumptions
$(\hyperlink{A:b-sigma'}{b',\sigma'},\hyperlink{A:xi-f}{\xi,f},
\hyperlink{A:h}{h})$ hold. For $i\in\{1,2\}$, let
$\zeta_n^i\in L^2(\Omega,\mathcal F^{M^n}_{r_n};\mathbb R^d)$, and let
$(X^{n,i},Y^{n,i},Z^{n,i})$ be the solution on $[r_n,S]$ of
\begin{align}
 X_t^{n,i}
 &=\zeta_n^i
   +\int_{r_n}^t b(s,X_s^{n,i})\,\mathrm dc_s^n
   +\int_{r_n}^t \sigma(s,X_s^{n,i})\,\mathrm dM_s^n,
 \nonumber\\
 Y_t^{n,i}
 &=\xi(X_S^{n,i})
   +\int_t^S f(s,X_s^{n,i},Y_s^{n,i},Z_s^{n,i})\,\mathrm dc_s^n
   +\int_t^S h(s,X_s^{n,i},Y_s^{n,i})\,\mathrm dW_s^n
   -\int_t^S Z_s^{n,i}\,\mathrm dM_s^n.
 \label{eq:pairwise-backward}
\end{align}
If $\mathbb E\left[|\zeta_n^1-\zeta_n^2|^2\right]\longrightarrow0$, then
it holds that
\begin{align}
 \lVert Y^{n,1}-Y^{n,2}\rVert_{p;[r_n,S]}
 +\lVert Y^{n,1}-Y^{n,2}\rVert_{\infty;[r_n,S]}
 &\longrightarrow0
 &&\text{in probability},
 \label{eq:pairwise-Y-convergence}\\
 \mathbb E\left[
   \int_{r_n}^S|Z_s^{n,1}-Z_s^{n,2}|^2\,\mathrm dc_s^n
 \right]
 &\longrightarrow0.
 \label{eq:pairwise-Z-convergence}
\end{align}
The same conclusions hold if $\zeta_n^1=\zeta_n^2$ and the terminal
condition $\xi(X_S^{n,i})$ in \eqref{eq:pairwise-backward} is replaced by
$\xi_n^i(X_S^{n,i})$, $i\in\{1,2\}$, where
$\xi_n^i\in\mathrm{BUC}(\mathbb R^d;\mathbb R^k)$ satisfy
\[
 \sup_{n\in\mathbb N}\sup_{i\in\{1,2\}}\|\xi_n^i\|_\infty<\infty,
 \qquad
 \|\xi_n^1-\xi_n^2\|_\infty\longrightarrow0.
\]
\end{proposition}

\medskip
\noindent\textit{Remark.}
We follow the double Picard iteration from
\cite[proof of Theorem~4.11]{becherer_rough_2026}. The same idea is used
later in Lemma~\ref{Lemma-excursion-Y-convergence} to overcome the
adaptedness problem caused by reparametrization when working with
Skorokhod-type metrics. No such reparametrization, and hence no corresponding
adaptedness issue, is present here. Nevertheless, the double Picard scheme is
preferable to a direct Gronwall-type argument, since the latter would
typically require stronger convergence in
\eqref{eq:terminal-pairwise-data}, \eqref{eq:f-pairwise-data}, and
\eqref{eq:g-pairwise-data}, for instance convergence in
$L^\infty$ rather than in $L^2$ or in probability. Such convergence
is not available because the solutions of the forward SDE are unbounded.
\medskip

\begin{proof}
We start by defining the data of the two rough BSDEs by
\begin{equation*}
 \xi^{n,i}:=\xi(X_S^{n,i}),\qquad
 f_s^{n,i}(y,z):=f(s,X_s^{n,i},y,z),\qquad
 g_s^{n,i}(y):=h(s,X_s^{n,i},y).
\end{equation*}
We introduce the notation $\Delta X^n:=X^{n,1}-X^{n,2}$ and use analogous notation for all
other pairwise differences.
For every $\rho>2$, Proposition~\ref{Prop-random-initial-SDE-stability}
gives
\begin{equation}\label{eq:quantitative-forward-pairwise}
 \mathbb E\left[
   \lVert\Delta X^n\rVert_{\infty;[r_n,S]}^2
   +\lVert\Delta X^n\rVert_{\rho;[r_n,S]}^2
 \right]
 \le C_\rho\,\mathbb E\left[|\zeta_n^1-\zeta_n^2|^2\right].
\end{equation}
The constant depends only on $\rho$, $S$, the Lipschitz constants of $b$ and
$\sigma$, and the bound in \eqref{eq:uniform-controls-pairwise}; in particular,
it is independent of $n$ and $r_n$. Consequently, the left-hand side of
\eqref{eq:quantitative-forward-pairwise} converges to zero.
Assumption~\hyperlink{A:xi-f}{\((\xi,f)\)} and
\eqref{eq:quantitative-forward-pairwise} imply
\begin{equation}\label{eq:terminal-pairwise-data}
 \mathbb E\left[|\xi^{n,1}-\xi^{n,2}|^2\right]
 \longrightarrow0,
 \qquad
 \sup_{n,i}\lVert\xi^{n,i}\rVert_{L^\infty}<\infty.
\end{equation}
Indeed, the terminal difference converges to zero in probability by uniform
continuity, and its uniform boundedness upgrades the convergence to $L^2$.
Moreover, the same assumption gives
\begin{equation}\label{eq:f-pairwise-data}
 \varepsilon_{f,n}
 :=\sup_{s\in[r_n,S]}\sup_{y,z}
   |f_s^{n,1}(y,z)-f_s^{n,2}(y,z)|
 \le C_{\xi,f}\lVert\Delta X^n\rVert_\infty
 \longrightarrow0
\end{equation}
in $L^2$, hence in probability.

Suppose first that Assumption~\hyperlink{A:h}{\((h)\)} holds. Let
$\rho\in(2,p)$ be the stronger variation exponent appearing in that
assumption.
Proposition~\ref{Prop-random-initial-SDE-estimates} yields
\begin{equation}\label{eq:uniform-individual-forward-rho}
\sup_{n\in\mathbb N}\sup_{i\in\{1,2\}}
 \lVert X^{n,i}\rVert_{\rho,2;[r_n,S]}<\infty.
\end{equation}
Since $\rho<p$, this implies the corresponding uniform $p,2$-bound.
Lemma~\ref{Lemma-hX-Assumption-Ad} therefore gives all individual
admissibility bounds for $g^{n,i}$ and $D_yg^{n,i}$, uniformly in $n,i$.
Lemma~\ref{Lemma-hX-Assumption-Bg} gives
\begin{align}
 \varepsilon_{g,n}
 &:=\sup_y\lVert g^{n,1}(y)-g^{n,2}(y)\rVert_{p;[r_n,S]}
 \longrightarrow0
 &&\text{in probability},
 \label{eq:g-pairwise-data}\\
 \varepsilon_{Dg,n}
 &:=\sup_{s\in[r_n,S]}
   |D_yg_s^{n,1}-D_yg_s^{n,2}|_\infty
 \longrightarrow0
 &&\text{in }L^2,
 \label{eq:Dg-pairwise-data}\\
 \varepsilon_{g,n}^{\infty}
 &:=\sup_{s\in[r_n,S]}\sup_y
   |g_s^{n,1}(y)-g_s^{n,2}(y)|
 \longrightarrow0
 &&\text{in }L^2.
 \label{eq:g-uniform-pairwise-data}
\end{align}
The use of the $p$-variation topology in \eqref{eq:g-pairwise-data} is
legitimate because \eqref{eq:quantitative-forward-pairwise} holds for every
exponent larger than $2$.
If Assumption~\hyperlink{A:h'}{\((h')\)} holds instead, the individual
admissibility bounds follow directly from that assumption, while
$g^{n,1}=g^{n,2}=h$. In this case, set
$\varepsilon_{g,n}=\varepsilon_{Dg,n}=\varepsilon_{g,n}^{\infty}=0$, so
\eqref{eq:g-pairwise-data}--\eqref{eq:g-uniform-pairwise-data} hold
trivially.

For $i\in\{1,2\}$, set $Y^{n,i,0}=Z^{n,i,0}=0$ and define recursively
\begin{align}
 Y_t^{n,i,m+1}
 ={}&\xi^{n,i}
 +\int_t^S f_s^{n,i}(Y_s^{n,i,m},Z_s^{n,i,m})\,\mathrm dc_s^n
 +\int_t^S g_s^{n,i}(Y_s^{n,i,m})\,\mathrm dW_s^n \nonumber \\
 &-\int_t^S Z_s^{n,i,m+1}\,\mathrm dM_s^n.
 \nonumber
\end{align}
For every $m\in\mathbb N_0$, we can write
\begin{align}
 \lVert\Delta Y^n\rVert_{p;[r_n,S]}
 \le
 \lVert Y^{n,1}-Y^{n,1,m}\rVert_{p;[r_n,S]}
 +\lVert\Delta Y^{n,m}\rVert_{p;[r_n,S]}
 +\lVert Y^{n,2,m}-Y^{n,2}\rVert_{p;[r_n,S]}.
 \label{eq:pairwise-picard-three-term}
\end{align}
We first show that the Picard approximations converge uniformly in $n$ and
$i$, and then prove convergence of $\Delta Y^{n,m}$ at every fixed Picard
level.

{For the uniform convergence of the Picard schemes, we first
bound the data.} Writing $C^{\mathrm{Lip}}_{f^{n,i}}$ for the Lipschitz constant of
$f^{n,i}$ in $(y,z)$, we have the following bound:
\begin{align}
 \sup_{n,i}\Bigg(&
 \lVert\xi^{n,i}\rVert_{L^\infty}
 +\lVert f^{n,i}(\cdot,0,0)\rVert_
   {L^\infty(\Omega\times[r_n,S])}
 +C^{\mathrm{Lip}}_{f^{n,i}}
 +\sup_{s\in[r_n,S]}
   \bigl\||g_s^{n,i}|_{C_b^2}\bigr\|_{L^\infty}
 \notag\\
 &+\sup_{a\in[r_n,S]}\mathop{\rm ess\,sup}_{\omega\in\Omega}
   \mathbb E_a\left[
     \lVert g^{n,i}\rVert_{p;[a,S];C_b(\mathbb R^k)}^2
   \right]^{1/2}
 \label{eq:uniform-picard-data-bounds}\\
 &+\sup_{a\in[r_n,S]}\mathop{\rm ess\,sup}_{\omega\in\Omega}
   \mathbb E_a\left[
     \lVert D_yg^{n,i}\rVert_{p;[a,S];C_b(\mathbb R^k)}^2
   \right]^{1/2}
 \Bigg)
 +\sup_n\left(c_S^n+\lVert W^n\rVert_{q;[0,S]}\right)
 <\infty.
 \notag
\end{align}
Here the bounds for $\xi^{n,i}$, $f^{n,i}$, $c^n$, and
$W^n$ follow from Assumption~\hyperlink{A:xi-f}{\((\xi,f)\)} and
\eqref{eq:uniform-controls-pairwise}. Under
Assumption~\hyperlink{A:h}{\((h)\)}, the bounds for $g^{n,i}$ and
$D_yg^{n,i}$ follow from Lemma~\ref{Lemma-hX-Assumption-Ad} and
\eqref{eq:uniform-individual-forward-rho}; under
Assumption~\hyperlink{A:h'}{\((h')\)}, they follow directly from that
assumption.
By the same argument as in the proof of
\cite[Theorem~4.11]{becherer_rough_2026}, more precisely the argument leading
to \cite[(4.16)]{becherer_rough_2026}, this common bound yields
\begin{equation}\label{eq:uniform-picard-bounds}
 \sup_{n,i,m}\left(
   \lVert Y^{n,i,m}\rVert_{p,2;[r_n,S]}
   +\lVert Z^{n,i,m}\rVert_{\mathrm{BMO};[r_n,S]}
 \right)<\infty
\end{equation}
and
\begin{equation}\label{eq:uniform-picard-convergence}
 \lim_{m\to\infty}\sup_{n,i}\left(
   \lVert Y^{n,i,m}-Y^{n,i}\rVert_{p,2;[r_n,S]}
   +\lVert Z^{n,i,m}-Z^{n,i}\rVert_{\mathrm{BMO};[r_n,S]}
 \right)=0.
\end{equation}

{It remains to show, for every fixed $m\in\mathbb N_0$, that}
\begin{equation}
 \lVert\Delta Y^{n,m}\rVert_{p;[r_n,S]}
 \longrightarrow0
 \qquad\text{in probability}.
 \label{eq:fixed-picard-Y}
\end{equation}
We establish this iteratively in $m$, together with
\begin{equation}
 \mathbb E\left[
   \int_{r_n}^S|\Delta Z_s^{n,m}|^2\,\mathrm dc_s^n
 \right]
 \longrightarrow0.
 \label{eq:fixed-picard-Z}
\end{equation}
Both assertions are immediate for $m=0$. Assume that they hold at level $m$.
Since $Y_S^{n,i,m}=\xi^{n,i}$ for $m\ge1$, while $Y^{n,i,0}=0$, we also have
\begin{equation}\label{eq:fixed-picard-Y-sup}
 \lVert\Delta Y^{n,m}\rVert_{\infty;[r_n,S]}
 \le |\Delta Y_S^{n,m}|+
      \lVert\Delta Y^{n,m}\rVert_{p;[r_n,S]}
 \longrightarrow0
\end{equation}
in probability. For easier notation, set
\begin{align*}
 F_s^{n,m}
 &:={f_s^{n,1}(Y_s^{n,1,m},Z_s^{n,1,m})
     -f_s^{n,2}(Y_s^{n,2,m},Z_s^{n,2,m})},\\
 G_s^{n,m}
 &:={g_s^{n,1}(Y_s^{n,1,m})
     -g_s^{n,2}(Y_s^{n,2,m})}.
\end{align*}

{We first control the finite-variation term.}
By the common Lipschitz constant of the generators, it holds
\begin{equation*}
 \int_{r_n}^S|F_s^{n,m}|\,\mathrm dc_s^n
 \le C\lVert\Delta Y^{n,m}\rVert_\infty
 +C\left(\int_{r_n}^S|\Delta Z_s^{n,m}|^2\,\mathrm dc_s^n\right)^{1/2}
 +C\varepsilon_{f,n}.
\end{equation*}
It follows from the induction hypothesis and
\eqref{eq:f-pairwise-data} that
\begin{equation}\label{eq:F-integral-convergence}
 \left\lVert\int_{\cdot}^S F_s^{n,m}\,\mathrm dc_s^n
 \right\rVert_{p;[r_n,S]}
 \le\int_{r_n}^S|F_s^{n,m}|\,\mathrm dc_s^n
 \longrightarrow0 \qquad \text{in probability.}
\end{equation}

{For the Young term, we use the composition-difference estimate}
\cite[Lemma~2.3]{becherer_rough_2026}, together with the uniform
bounds for $g^{n,i}$ in \eqref{eq:uniform-picard-data-bounds}, {to obtain}
\begin{align}
 \lVert G^{n,m}\rVert_{p;[r_n,S]}
 \le C\bigl(&\lVert\Delta Y^{n,m}\rVert_p
 +(\lVert Y^{n,1,m}\rVert_p+\lVert Y^{n,2,m}\rVert_p)
      \lVert\Delta Y^{n,m}\rVert_\infty \nonumber\\
 &+\varepsilon_{g,n}
 +\varepsilon_{Dg,n}\lVert Y^{n,2,m}\rVert_p\bigr).
 \nonumber
\end{align}
The $p$-variation norms on the right-hand side are bounded in
probability, uniformly in $n$, by \eqref{eq:uniform-picard-bounds}. Thus
\eqref{eq:fixed-picard-Y}, \eqref{eq:fixed-picard-Y-sup},
\eqref{eq:g-pairwise-data} and \eqref{eq:Dg-pairwise-data} show that
\begin{equation*}
 \lVert G^{n,m}\rVert_{p;[r_n,S]}\longrightarrow0
 \qquad\text{in probability}.
\end{equation*}
Moreover, by \eqref{eq:terminal-pairwise-data} and
\eqref{eq:g-uniform-pairwise-data},
\begin{equation*}
 |G_S^{n,m}|
 \le C|\Delta Y_S^{n,m}|+\varepsilon_{g,n}^{\infty}
 \longrightarrow0
 \qquad\text{in probability}.
\end{equation*}
The Young integral estimate
\cite[Proposition~A.7]{becherer_rough_2026} and
\eqref{eq:uniform-controls-pairwise} now imply
\begin{equation}\label{eq:G-young-integral-convergence}
 \begin{aligned}
 &\left\lVert\int_{\cdot}^S G_s^{n,m}\,\mathrm dW_s^n
 \right\rVert_{p;[r_n,S]}\\
 &\quad\le C_{p,q}\bigl(|G_S^{n,m}|+\lVert G^{n,m}\rVert_p\bigr)
                 \lVert W^n\rVert_q
 \longrightarrow0
 \qquad\text{in probability}.
 \end{aligned}
\end{equation}

{To control the martingale terms, define the terminal Picard input}
\begin{equation*}
 \Gamma^{n,i,m+1}
 :=\xi^{n,i}
   +\int_{r_n}^S f_s^{n,i}(Y_s^{n,i,m},Z_s^{n,i,m})\,\mathrm dc_s^n
   +\int_{r_n}^S g_s^{n,i}(Y_s^{n,i,m})\,\mathrm dW_s^n.
\end{equation*}
Equations \eqref{eq:terminal-pairwise-data},
\eqref{eq:F-integral-convergence} and
\eqref{eq:G-young-integral-convergence} yield
\begin{equation}\label{eq:Gamma-probability-convergence}
 \Gamma^{n,1,m+1}-\Gamma^{n,2,m+1}\longrightarrow0
 \qquad\text{in probability}.
\end{equation}
We next upgrade this to $L^2$-convergence. By construction,
\begin{equation}\label{eq:Gamma-martingale-decomposition}
 \Gamma^{n,i,m+1}
 =Y_{r_n}^{n,i,m+1}
  +\int_{r_n}^S Z_s^{n,i,m+1}\,\mathrm dM_s^n.
\end{equation}
The uniform terminal bound \eqref{eq:terminal-pairwise-data},
\eqref{eq:uniform-picard-bounds}, and
\cite[Lemma~2.2(a)]{becherer_rough_2026} give
\[
 \sup_{n,i,m}
 \lVert Y^{n,i,m}\rVert_{L^\infty(\Omega\times[r_n,S])}<\infty.
\]
The first inequality below follows from the
Burkholder--Davis--Gundy inequality, while the second follows from the
standard BMO energy estimate:
\begin{align*}
 \mathbb E\left[
   \left|\int_{r_n}^S Z_s^{n,i,m+1}\,\mathrm dM_s^n\right|^4
 \right]
 &\le C\mathbb E\left[
   \left(\int_{r_n}^S|Z_s^{n,i,m+1}|^2\,\mathrm dc_s^n\right)^2
 \right]\\
 &\le C\lVert Z^{n,i,m+1}\rVert_{\mathrm{BMO};[r_n,S]}^4.
\end{align*}
The uniform BMO bound in \eqref{eq:uniform-picard-bounds} therefore makes
the right-hand side uniform in $n$, $i$, and $m$. This is the same argument
as in the proof of \cite[(4.24)]{becherer_rough_2026}.
Consequently,
\begin{equation}\label{eq:Gamma-L4-bound}
 \sup_{n,i,m}\mathbb E\left[|\Gamma^{n,i,m+1}|^4\right]<\infty,
 \qquad
 \sup_n\mathbb E\left[
   |\Gamma^{n,1,m+1}-\Gamma^{n,2,m+1}|^4
 \right]<\infty.
\end{equation}
Combining \eqref{eq:Gamma-probability-convergence} and
\eqref{eq:Gamma-L4-bound}, Vitali's theorem yields
\begin{equation*}
\mathbb E\left[
  |\Gamma^{n,1,m+1}-\Gamma^{n,2,m+1}|^2
\right]\longrightarrow0.
\end{equation*}
By It\^o's isometry, \eqref{eq:Gamma-martingale-decomposition}, and the
orthogonality of the martingale increment to
$\mathcal F^{M^n}_{r_n}$, using that
$Y_{r_n}^{n,1,m+1}-Y_{r_n}^{n,2,m+1}$ is
$\mathcal F^{M^n}_{r_n}$-measurable, we obtain
\begin{align}
 \mathbb E\left[
   \int_{r_n}^S|\Delta Z_s^{n,m+1}|^2\,\mathrm dc_s^n
 \right]
 &=\mathbb E\left[
   \left|
     \int_{r_n}^S\Delta Z_s^{n,m+1}\,\mathrm dM_s^n
   \right|^2
 \right]
 \nonumber\\
 &\leq
 \mathbb E\left[
   |Y_{r_n}^{n,1,m+1}-Y_{r_n}^{n,2,m+1}|^2
   +\left|
     \int_{r_n}^S\Delta Z_s^{n,m+1}\,\mathrm dM_s^n
   \right|^2
 \right]
 \nonumber\\
 &=\mathbb E\left[
   |\Gamma^{n,1,m+1}-\Gamma^{n,2,m+1}|^2
 \right]
 \longrightarrow0.
 \label{eq:next-picard-Z}
\end{align}
This proves \eqref{eq:fixed-picard-Z} at level $m+1$.

Finally, subtracting the two equations at Picard level $m+1$ yields
\[
 \Delta Y_t^{n,m+1}
 =\Delta\xi^n
  +\int_t^S F_s^{n,m}\,\mathrm dc_s^n
  +\int_t^S G_s^{n,m}\,\mathrm dW_s^n
  -\int_t^S\Delta Z_s^{n,m+1}\,\mathrm dM_s^n.
\]
The first three terms converge in $p$-variation in probability by
\eqref{eq:terminal-pairwise-data}, \eqref{eq:F-integral-convergence} and
\eqref{eq:G-young-integral-convergence}. For the martingale term, the
$p$-variation BDG inequality
\cite[Theorem~14.12]{friz_multidimensional_2010}, valid because $p>2$, and
\eqref{eq:next-picard-Z} give
\begin{equation*}
 \mathbb E\left[
   \left\lVert
   \int_{r_n}^{\cdot}\Delta Z_s^{n,m+1}\,\mathrm dM_s^n
 \right\rVert_{p;[r_n,S]}^2
 \right]
 \le C_p\mathbb E\left[
   \int_{r_n}^S|\Delta Z_s^{n,m+1}|^2\,\mathrm dc_s^n
 \right]
 \longrightarrow0.
\end{equation*}
Thus \eqref{eq:fixed-picard-Y} also holds at level $m+1$, completing the
induction.

Returning to \eqref{eq:pairwise-picard-three-term}, we first let
$n\to\infty$ and use \eqref{eq:fixed-picard-Y}. We then let
$m\to\infty$ and use \eqref{eq:uniform-picard-convergence}. This proves
the $p$-variation convergence in \eqref{eq:pairwise-Y-convergence}. Similarly,
\begin{align*}
 \mathbb E\left[\int_{r_n}^S|\Delta Z_s^n|^2\,\mathrm dc_s^n\right]
 \le{}&3\mathbb E\left[
   \int_{r_n}^S|Z_s^{n,1}-Z_s^{n,1,m}|^2\,\mathrm dc_s^n
 \right]
 +3\mathbb E\left[
   \int_{r_n}^S|\Delta Z_s^{n,m}|^2\,\mathrm dc_s^n
 \right]\\
 &+3\mathbb E\left[
   \int_{r_n}^S|Z_s^{n,2,m}-Z_s^{n,2}|^2\,\mathrm dc_s^n
 \right].
\end{align*}
The same order of limits, using \eqref{eq:fixed-picard-Z} and the BMO part of
\eqref{eq:uniform-picard-convergence}, proves
\eqref{eq:pairwise-Z-convergence}. Finally,
\[
 \lVert\Delta Y^n\rVert_\infty
 \le |\xi^{n,1}-\xi^{n,2}|+\lVert\Delta Y^n\rVert_p,
\]
so \eqref{eq:terminal-pairwise-data} and
the $p$-variation convergence in \eqref{eq:pairwise-Y-convergence} imply its
uniform-convergence assertion.

For the final assertion, uniqueness of the forward SDE gives
$X^{n,1}=X^{n,2}$. Hence the composed coefficients $f^{n,i}$ and $g^{n,i}$
coincide. With $\xi^{n,i}:=\xi_n^i(X_S^{n,i})$, the terminal data satisfy
\[
 \left|
 \xi_n^1(X_S^{n,1})-\xi_n^2(X_S^{n,2})
 \right|
 \leq \|\xi_n^1-\xi_n^2\|_\infty\longrightarrow0.
\]
Thus \eqref{eq:terminal-pairwise-data} holds and all coefficient-difference
terms in \eqref{eq:f-pairwise-data}--\eqref{eq:g-uniform-pairwise-data}
vanish. The double Picard argument above applies without change and proves
the claim.
\end{proof}

\phantomsection\label{pairwise-Marcus-extension}
Although Proposition~\ref{prop:uniform-pairwise-stability} is stated for
continuous drivers $W^n$, it also applies to rough FBSDEs with
discontinuous Young drivers integrated in the Marcus sense, by passing to
their time-extended versions. To explain this, let
$W^n\in D^q([0,T];\mathbb R^e)$, let $r_n\in[0,T]$, and, for
$i\in\{1,2\}$, consider the solution $(X^{n,i},Y^{n,i},Z^{n,i})$ of
\begin{equation}\label{pairwise-Marcus-FBSDE}
\begin{aligned}
 X_t^{n,i}
 &=\zeta_n^i+\int_{r_n}^t b(s,X_s^{n,i})\,ds
   +\int_{r_n}^t\sigma(s,X_s^{n,i})\,dB_s,\\
 Y_t^{n,i}
 &=\xi(X_T^{n,i})
   +\int_t^T f(s,X_s^{n,i},Y_s^{n,i},Z_s^{n,i})\,ds\\
 &\quad+\int_t^T h(s,X_s^{n,i},Y_s^{n,i})\diamond dW_s^n
   -\int_t^T Z_s^{n,i}\,dB_s,
 \qquad t\in[r_n,T],
\end{aligned}
\end{equation}
where $\zeta_n^i\in L^2(\Omega,\mathcal F^B_{r_n};\mathbb R^d)$.
Fix $\delta>0$ and a common countable set $\Pi$ containing all jump times
of all $W^n$. Let $\tau^\delta$ and $c^\delta$ be defined as in Section~\ref{section-decorated-path}. Set
\[
 W^{\delta,n}=(\jmath W^n)^\delta,\qquad
 B^\delta=B\circ c^\delta,\qquad
 r_n^\delta=\tau^\delta(r_n),
\]
and define the time-extended coefficients by
\[
\begin{aligned}
 b^\delta(s,x)&=b(c_s^\delta,x),&
 \sigma^\delta(s,x)&=\sigma(c_s^\delta,x),\\
 f^\delta(s,x,y,z)&=f(c_s^\delta,x,y,z),&
 h^\delta(s,x,y)&=h(c_s^\delta,x,y).
\end{aligned}
\]
We then consider the solution
$(X^{\delta,n,i},Y^{\delta,n,i},Z^{\delta,n,i})$ of the time-stretched
rough FBSDE
\begin{equation}\label{pairwise-time-stretched-FBSDE}
\begin{aligned}
 X_t^{\delta,n,i}
 &=\zeta_n^i+\int_{r_n^\delta}^t
     b^\delta(s,X_s^{\delta,n,i})\,dc_s^\delta
   +\int_{r_n^\delta}^t
     \sigma^\delta(s,X_s^{\delta,n,i})\,dB_s^\delta,\\
 Y_t^{\delta,n,i}
 &=\xi(X_{T+\delta}^{\delta,n,i})
   +\int_t^{T+\delta}
     f^\delta(s,X_s^{\delta,n,i},Y_s^{\delta,n,i},Z_s^{\delta,n,i})
     \,dc_s^\delta\\
 &\quad+\int_t^{T+\delta}
     h^\delta(s,X_s^{\delta,n,i},Y_s^{\delta,n,i})\,dW_s^{\delta,n}
   -\int_t^{T+\delta}Z_s^{\delta,n,i}\,dB_s^\delta,
 \qquad t\in[r_n^\delta,T+\delta].
\end{aligned}
\end{equation}
Notice that $W^{\delta,n}$ is continuous: its jumps have been replaced by
linear excursions. Although $B^\delta$ is generally no longer a Brownian
motion in the extended time variable, it is a continuous martingale in
its usual filtration $\mathcal F_t^\delta=\mathcal F^B_{c_t^\delta}$,
with quadratic covariation
$\langle B^\delta\rangle_t=c_t^\delta I_d$. Thus
\eqref{pairwise-time-stretched-FBSDE} is covered by our general setting
$M=B\circ c$, with horizon $T+\delta$, clock $c^\delta$, and
martingale $M=B^\delta$.

By construction, $X_s^{\delta,n,i}=X_{c_s^\delta}^{n,i}$.
And by {\cite[Theorem~4.7]{becherer_rough_2026}}, we can see that
\eqref{pairwise-time-stretched-FBSDE} is indeed the time-stretched version of
\eqref{pairwise-Marcus-FBSDE}, with
\begin{equation}\label{delta-extension-Y}
 Y_s^{\delta,n,i}=
 \begin{cases}
   Y_r^{n,i},
     & s=\tau^\delta(r),\quad r\in[r_n,T],\\[1ex]
   \displaystyle\varphi\!\left(
     -h(r,X_r^{n,i},\cdot)\Delta W_r^n,
     Y_{r+}^{n,i},
     \frac{s-\tau^\delta(r)}{\tau^\delta(r+)-\tau^\delta(r)}
   \right),
     & \begin{gathered}
         s\in(\tau^\delta(r),\tau^\delta(r+)],\\
         r\in\Pi\cap[r_n,T),
       \end{gathered}
 \end{cases}
\end{equation}
and $Z_s^{\delta,n,i}=Z_{c_s^\delta}^{n,i}$ holds $dc_s^\delta\otimes\mathbb P$-almost everywhere.
Conversely, the original process is recovered simply by
$Y^{\delta,n,i}\circ\tau^\delta=Y^{n,i}$ on $[r_n,T]$.

Now we can apply Proposition~\ref{prop:uniform-pairwise-stability} to the
time-stretched equations.
The coefficient assumptions are preserved under the common
non-decreasing time change, and the linear extensions $W^{\delta,n}$
have uniformly bounded $q$-variation whenever the original drivers do.
Proposition~\ref{prop:uniform-pairwise-stability} consequently applies to
\eqref{pairwise-time-stretched-FBSDE}. Composition via $\tau^\delta$
cannot increase either the uniform norm or the $p$-variation seminorm of
the difference of the two $Y$-solutions, while change of variables gives
\[
 \int_{r_n}^T |Z_s^{n,1}-Z_s^{n,2}|^2\,ds
 =\int_{r_n^\delta}^{T+\delta}
   |Z_s^{\delta,n,1}-Z_s^{\delta,n,2}|^2\,dc_s^\delta.
\]
Thus the stability conclusions of the proposition hold for
\eqref{pairwise-Marcus-FBSDE}, including stability with respect to the
terminal function.

The same argument applies if the original equations already have clocks
$c^n$ and martingales $M^n=B\circ c^n$, as in
\eqref{eq:pairwise-backward}. Their extended clocks and martingales are
$c^n\circ c^\delta$ and $M^n\circ c^\delta=B\circ(c^n\circ c^\delta)$,
respectively. The terminal clock values are unchanged, and the preceding
change-of-variables identity holds with $dc_s^n$ on the original interval
and $d(c^n\circ c^\delta)_s$ on the extended interval.
\medskip

When we wish to emphasize the terminal time and terminal function of the
backward equation, for $0\leq r\leq s\leq S$ and
$\zeta\in L^2(\Omega,\mathcal F_r^M;\mathbb R^d)$ we write
$(X^{r,\zeta},Y^{r,\zeta;s,\xi},Z^{r,\zeta;s,\xi})$ for the solution on
$[r,s]$ of the rough FBSDE \eqref{extended-rough-FBSDE}, with initial
condition $X_r^{r,\zeta}=\zeta$ and terminal condition
$Y_s^{r,\zeta;s,\xi}=\xi(X_s^{r,\zeta})$.

\begin{proposition}
  \label{Prop-RFBSDE-Markov-flow}Let $W\in D^q([0,S];\mathbb R^e)$ with $q\in[1,2)$.
Suppose that the coefficients on $[0,S]$
satisfy either Assumptions
$(\hyperlink{A:b-sigma}{b}, \hyperlink{A:b-sigma}{\sigma},
\hyperlink{A:xi-f}{\xi}, \hyperlink{A:xi-f}{f}, \hyperlink{A:h'}{h'})$
or Assumptions
$(\hyperlink{A:b-sigma'}{b'}, \hyperlink{A:b-sigma'}{\sigma'},
\hyperlink{A:xi-f}{\xi}, \hyperlink{A:xi-f}{f}, \hyperlink{A:h}{h})$.
Define $u(r;s,\xi)(x)\assign Y_r^{r,x;s,\xi}$, and for $s=S$ simply write
$u(r,x)\assign u(r;S,\xi)(x)$.
Then:
\begin{enumerateroman}
  \item $u(r;s,\xi)(x)$ is deterministic;
  \item for $0\leq r\leq v\leq s\leq S$,
  \[
    Y_v^{r,x;s,\xi}
    =u(v;s,\xi)\bigl(X_v^{r,x}\bigr);
  \]
  \item for $0\leq r\leq v\leq s\leq t\leq S$,
  \begin{equation}
    Y_v^{r,x;t,\xi}
    =
    Y_v^{r,x;s,Y_s^{s,\cdot;t,\xi}}
    \label{flow-RBSDE}
  \end{equation}
  and, in particular,
  \[
    u(r;t,\xi)(x)
    =
    u\bigl(r;s,u(s;t,\xi)\bigr)(x).
  \]
\end{enumerateroman}
\end{proposition}

\begin{proof}
Fix $r$, $x$, and $s$, and let
\[
  \mathcal{F}^{r,M}_v
  \assign
  \sigma\{M_q-M_r:r\leq q\leq v\}\vee\mathcal{N},
  \qquad v\in[r,s],
\]
where $\mathcal{N}$ denotes the null sets. Since $W$, $\xi$, and all
coefficients are deterministic, the FBSDE started from $(r,x)$ can be solved
in the shifted filtration $(\mathcal{F}^{r,M}_v)_{v\in[r,s]}$. Uniqueness
implies that this solution agrees with the one constructed in the original
filtration. Hence $Y_r^{r,x;s,\xi}$ is measurable with respect to the trivial
$\sigma$-field $\mathcal{F}^{r,M}_r$ and is therefore deterministic. This
proves (i). \\
We next establish the representation for a random initial state. If
$\zeta=\sum_{i=1}^m x_i\mathbf{1}_{A_i}$ is a simple
$\mathcal{F}_v^M$-measurable random variable, then uniqueness of solutions and (i)
give
\[
  Y_v^{v,\zeta;s,\xi}
  =
  \sum_{i=1}^m\mathbf{1}_{A_i}Y_v^{v,x_i;s,\xi}
  =
  \sum_{i=1}^m\mathbf{1}_{A_i}u(v;s,\xi)(x_i)
  =
  u(v;s,\xi)(\zeta).
\]
For general $\zeta\in L^2(\Omega,\mathcal F_v^M;\mathbb R^d)$, choose simple
$\zeta_n\to\zeta$ in $L^2$. Proposition~\ref{prop:uniform-pairwise-stability},
using the Marcus time-extension argument on
p.~\pageref{pairwise-Marcus-extension} when $W$ is discontinuous,
applied on $[v,s]$ with the constant sequences $c^n=c$, $M^n=M$,
$W^n=W$, and $r_n=v$, yields
\[
  Y_v^{v,\zeta_n;s,\xi}
  \longrightarrow
  Y_v^{v,\zeta;s,\xi}
  \qquad\text{in probability}.
\]
Applied with deterministic initial states $x_n$ and $x$ satisfying
$x_n\to x$, the same proposition and (i) show that
$u(v;s,\xi)(x_n)\to u(v;s,\xi)(x)$, so $u(v;s,\xi)$ is continuous. Hence
$u(v;s,\xi)(\zeta_n)\to u(v;s,\xi)(\zeta)$ in probability, and passing to
the limit in the simple-state identity gives
\[
  Y_v^{v,\zeta;s,\xi}=u(v;s,\xi)(\zeta).
\]

Taking $\zeta=X_v^{r,x}$, the forward flow property yields
\[
  X_q^{v,X_v^{r,x}}=X_q^{r,x},
  \qquad q\in[v,s].
\]
Consequently, the restriction of $(Y^{r,x;s,\xi},Z^{r,x;s,\xi})$ to
$[v,s]$ and $(Y^{v,X_v^{r,x};s,\xi},Z^{v,X_v^{r,x};s,\xi})$ solve the same
rough BSDE. Uniqueness therefore gives
\[
  Y_v^{r,x;s,\xi}
  =
  Y_v^{v,X_v^{r,x};s,\xi}
  =
  u(v;s,\xi)\bigl(X_v^{r,x}\bigr),
\]
which proves (ii). Finally, applying (ii) with terminal time $t$ gives
\[
  Y_s^{r,x;t,\xi}
  =
  u(s;t,\xi)\bigl(X_s^{r,x}\bigr).
\]
Thus the restriction of $(Y^{r,x;t,\xi},Z^{r,x;t,\xi})$ to $[r,s]$
solves the rough BSDE with terminal function $u(s;t,\xi)$. Uniqueness on
$[r,s]$ proves \eqref{flow-RBSDE}, and evaluating it at $v=r$ gives the
asserted flow identity for $u$.
\end{proof}

To formulate stability in the Skorokhod topology, we lift the rough BSDE
solutions to the space of decorated paths. Let
$W^n,W\in D^q([0,T];\mathbb R^e)$ and fix $(t,x)\in[0,T]\times\mathbb R^d$.
Write $(X^{t,x},Y^{n,t,x},Z^{n,t,x})$ for the solution of
\eqref{pairwise-Marcus-FBSDE} with $r_n=t$ and $\zeta_n^i=x$,
suppressing the index $i$, and write $(X^{t,x},Y^{t,x},Z^{t,x})$ for the
solution driven by $W$. Use a common countable set containing the jump
times of $W$ and all $W^n$ for their time extensions. Following the
construction after Proposition~\ref{prop:uniform-pairwise-stability},
denote the corresponding extended solutions by
$(X^{\delta,\tau^\delta(t),x},Y^{\delta,n,\tau^\delta(t),x},
Z^{\delta,n,\tau^\delta(t),x})$ and
$(X^{\delta,\tau^\delta(t),x},Y^{\delta,\tau^\delta(t),x},
Z^{\delta,\tau^\delta(t),x})$, respectively. By
{\cite[Theorem~4.7]{becherer_rough_2026}}, these are the unique solutions of
\eqref{pairwise-time-stretched-FBSDE} with drivers $W^{\delta,n}$ and
$W^\delta=(\jmath W)^\delta$.

For each $s\in[t,T)$, define $\mathbf Y^{n,t,x}(s)$ as the linear
reparametrization onto $[0,1]$ of the restriction of
$Y^{\delta,n,\tau^\delta(t),x}$ to
$[\tau^\delta(s),\tau^\delta(s+)]$.
When no time is inserted, including at $T$, take the constant path with value
$Y_s^{n,t,x}$. By \eqref{delta-extension-Y}, this is precisely the Marcus
lift $\kappa_{-h,W^n}(Y^{n,t,x})$, with $h$ evaluated along the forward
process $X^{t,x}$, and its $\delta$-extension is
$Y^{\delta,n,\tau^\delta(t),x}$. Define $\mathbf Y^{t,x}$ in the same way
using the driver $W$.

We are now in a position to state the following stability result for the rough
FBSDE \eqref{rough-FBSDE} on the fixed interval $[t,T]$.
\begin{proposition}
 \label{Prop-RFBSDE-stability}Let $W\in D^q([0,T];\mathbb R^e)$ with $q\in[1,2)$.
Suppose that either Assumptions
$(\hyperlink{A:b-sigma}{b}, \hyperlink{A:b-sigma}{\sigma}, 
\hyperlink{A:xi-f}{\xi}, \hyperlink{A:xi-f}{f},
\hyperlink{A:h'}{h'}, \hyperlink{A:hC}{h^C})$ 
or Assumptions 
$(\hyperlink{A:b-sigma'}{b'}, \hyperlink{A:b-sigma'}{\sigma'}, 
\hyperlink{A:xi-f}{\xi}, \hyperlink{A:xi-f}{f},
\hyperlink{A:h}{h}, \hyperlink{A:hC}{h^C})$ are satisfied. Fix
  \(t\in[0,T]\) and \(x\in\mathbb R^d\), and assume that $\lim_{n
  \rightarrow \infty} \alpha_{q;[t,T]} (\jmath W^n, \jmath W) = 0$ for some sequence of
  $(W^n)_{n \in \mathbb{N}} \subset D^q  ([0, T]; \mathbb{R}^e)$, and let
  $(X^{t,x},Y^{n,t,x},Z^{n,t,x})$ denote the solution to the rough FBSDE
  \eqref{pairwise-Marcus-FBSDE}. Then
  \[ \lim_{n \rightarrow \infty} \mathbb{P} (\alpha_{p ; [t, T]}
     (\mathbf{Y}^{n,t,x}, \mathbf{Y}^{t,x}) >
     \varepsilon) = 0 \quad \text{and} \quad \lim_{n \rightarrow \infty}
     \mathbb{E} \left[ \int_t^T (Z^{n,t,x}_r-Z^{t,x}_r)^2\,dr
     \right] = 0. \]
\end{proposition}

\begin{proof}
Define
\begin{align*}
  \xi^X&\assign\xi(X_T^{t,x}),&
  f^X(s,\cdot)&\assign f(s,X_s^{t,x},\cdot),&
  h^X(s,\cdot)&\assign h(s,X_s^{t,x},\cdot).
\end{align*}
The forward process, and hence the terminal condition and coefficients
\(\xi^X,f^X,h^X\), are independent of \(n\). Their uniform bounds and
regularity follow directly from Assumption \((h')\) in the
forward-independent case, and from
Proposition~\ref{Prop-RFBSDE-wellposedness} and
Lemma~\ref{Lemma-hX-Assumption-Ad} under Assumption \((h)\). Their
stability requirements are immediate because the coefficient sequence
is constant. The result therefore follows from
\cite[Theorem~4.11]{becherer_rough_2026}.
\end{proof}

\subsection{Smooth drivers}\label{Chap-ApproxPDE}

We start by recalling the notion of viscosity solutions introduced in
\cite{decreusefond_backward_1998} for the following system of PDEs
\begin{equation}
  \frac{\partial u_i}{\partial t} +\mathcal{L}_t u_i + f_i (t, x, u,
  \sigma^{\top} \nabla_x u_i) = 0, \quad u (T, x) = \xi (x), \quad 1 \leqslant i
  \leqslant {k}. \label{system-PDE}
\end{equation}
\begin{definition}
  \label{Def-viscosity}The function $u \in C ([0, T] \times \mathbb{R}^d ;
  \mathbb{R}^k)$ satisfying the polynomial growth condition
  \begin{equation}
    |u (t, x)| \leqslant C (1 + | x |^\ell),
    \qquad \text{for all } t\in[0,T],\, x\in \RR^d,
    \label{polynomial-growth}
  \end{equation}
  for some $\ell,C>0$, is called a \textbf{viscosity subsolution} of
  \eqref{system-PDE} if $u_i (T, x) \leqslant \xi_i (x), x \in \mathbb{R}^d, {1} \leqslant
  i \leqslant {k}$, and for any $1 \leqslant i \leqslant {k},$ $\varphi \in C^{1, 2} 
  ([0, T] \times \mathbb{R}^d)$ and $(t, x) \in [0, T) \times \mathbb{R}^d$
  at which $u_i - \varphi$ attains a local maximum, it holds that
  \[ \frac{\partial \varphi}{\partial t} (t, x) +\mathcal{L}_t \varphi (t, x)
     + f_i (t, x, u (t, x), (\sigma^{\top} \nabla_x \varphi) (t, x)) \geq 0.
  \]
  Similarly, such a function $u$ is called a
  \textbf{viscosity supersolution} of \eqref{system-PDE} if $u_i (T, x) \geq
  \xi_i (x), x \in \mathbb{R}^d, {1} \leqslant i \leqslant {k}$, and for any $1 \leqslant
  i \leqslant {k},$ $\varphi \in C^{1, 2}  ([0, T] \times \mathbb{R}^d)$ and $(t, x)
  \in [0, T) \times \mathbb{R}^d$ at which $u_i - \varphi$ attains a local minimum,
  it holds that
  \[ \frac{\partial \varphi}{\partial t} (t, x) +\mathcal{L}_t \varphi (t, x)
     + f_i (t, x, u (t, x), (\sigma^{\top} \nabla_x \varphi) (t, x)) \leqslant 0.
  \]
  Such a function $u$ is called a
  \textbf{viscosity solution} of \eqref{system-PDE} if it is both a
  viscosity subsolution and a viscosity supersolution.
\end{definition}

\begin{remark}
  \label{Rmk-growth-condition}A growth condition is essential to ensure the uniqueness of solutions 
to \eqref{system-PDE}, although it is not always stated explicitly 
in the literature. See, for example, the discussion after  {\cite[Theorem~3.5]{barles_backward_1997}}. In particular, the classical counterexample of Tychonoff shows that even the heat equation may fail to admit a unique solution without an appropriate growth restriction.
  
\end{remark}

\begin{remark}
  Viscosity solutions are most commonly formulated in the $1$-dimensional case. In the multidimensional setting, it is standard to impose the structural condition that the $i$-th component of 
$f(t,x,y,z)$ depends only on the $i$-th row of the matrix $z$ 
(see \cite{decreusefond_backward_1998, barles_backward_1997}). This restriction is required for the PDE theory, 
but it is not necessary at the level of the associated FBSDE.
\end{remark}
Under Assumptions 
$(\hyperlink{A:b-sigma}{b,\sigma}, \hyperlink{A:xi-f}{\xi,f})$, the existence of a viscosity
solution to \eqref{system-PDE} was shown in
{\cite{decreusefond_backward_1998}}. In terms of uniqueness, the existing results appear to be limited to the
time-homogeneous case for $b$ and $\sigma$ 
(see \cite{barles_backward_1997, pardoux_probabilistic_1997}). For completeness, we therefore provide a short proof of uniqueness 
in the time-inhomogeneous setting using the BSDE theory.
Consider the classical FBSDE
\begin{align}
  X^{t,x}_s
  ={}&x+\int_t^s b(r,X^{t,x}_r)\,dr
      +\int_t^s \sigma(r,X^{t,x}_r)\,dB_r,\nonumber\\
  Y^{t,x}_s
  ={}&\xi(X^{t,x}_T)
      +\int_s^T f(r,X^{t,x}_r,Y^{t,x}_r,Z^{t,x}_r)\,dr
      -\int_s^T Z^{t,x}_r\,dB_r.
  \label{classical-FBSDE}
\end{align}
\begin{theorem}
  \label{multidimentional-visco}Under Assumptions 
$(\hyperlink{A:b-sigma}{b,\sigma}, \hyperlink{A:xi-f}{\xi,f})$,
  the function given by $u(t,x):=Y_t^{t,x}$, where $(X^{t,x},Y^{t,x},Z^{t,x})$
  denotes the solution of FBSDE \eqref{classical-FBSDE}, is the unique viscosity
  solution to \eqref{system-PDE}.
\end{theorem}

\begin{proof}
  Existence of a viscosity solution follows from {\cite[Theorem~2.4]{decreusefond_backward_1998}}. For uniqueness, let $u, \tilde{u}$ be two viscosity solutions of \eqref{system-PDE} and define
  \[ f_i^u (t, x, z_i) \assign f_i (t, x, u (t, x), z_i) \quad \tmop{and}
     \quad f_i^{\tilde{u}} (t, x, z_i) \assign f_i (t, x, \tilde{u} (t, x),
     z_i) . \]
 Consider the FBSDE
  \begin{eqnarray}
    X^{t, x}_s & = & x + \int_t^s b (r, X^{t, x}_r) dr + \int_t^s
    \sigma (r, X^{t, x}_r) dB_r, \nonumber\\
    \mathcal{Y}^{u, t, x}_s & = & \xi_i (X^{t, x}_T) + \int^T_s f^u_i (r,
    X^{t, x}_r, \mathcal{Z}^{u, t, x}_r) dr - \int_s^T \mathcal{Z}^{u,
    t, x}_r dB_r. \nonumber
  \end{eqnarray}
  {
  The frozen generator $f_i^u$ is continuous, uniformly Lipschitz in $z_i$,
  and independent of the scalar backward unknown. By
  \eqref{polynomial-growth}, it satisfies
  $|f_i^u(t,x,z_i)|\le C(1+|x|^\ell+|z_i|)$.
  Hence \cite[Theorem~2.4]{decreusefond_backward_1998} shows that
  $\omega(t,x):=\mathcal Y_t^{u,t,x}$ is a continuous viscosity solution
  of polynomial growth to the scalar PDE
  \[ \frac{\partial \omega}{\partial t} +\mathcal{L}_t \omega + f^u_i (t, x,
     \sigma^{\top} \nabla_x \omega) = 0, \quad \omega (T, x) = \xi_i (x) . \]
  By construction, $u_i$ solves the same scalar PDE. On each compact
  cylinder $[0,T]\times\overline B_R$, $u$ is uniformly continuous.
  Together with the estimate
  $|f_i^u(t,x,z)-f_i^u(t,x',z)|
  \le L(|x-x'|+|u(t,x)-u(t,x')|)$
  and the spatial Lipschitz bounds on $b,\sigma$, this verifies the
  structure condition for parabolic comparison in
  \cite[Theorem~8.2]{crandall_users_1992}, after reversing time.
  We may therefore compare sub- and supersolutions on
  $[0,T]\times B_R$ whenever they are ordered at time $T$ and on
  $[0,T]\times\partial B_R$.

  To prove $u_i\le\omega$, we first arrange this boundary ordering by
  adding a positive term to $\omega$. Choose an integer $m$ such
  that $2m$ exceeds the growth exponents of $u_i$ and $\omega$, and put
  $\psi(t,x)=e^{A(T-t)}(1+|x|^2)^m$.
  Since $b,\sigma$ have at most linear growth, for sufficiently large $A$,
  \[
  \partial_t\psi+\mathcal L_t\psi
  +L|\sigma^{\top}\nabla_x\psi|\le0.
  \]
  Thus $\omega+\varepsilon\psi$ remains a supersolution of the scalar
  PDE for every $\varepsilon>0$.
  Fix $\varepsilon>0$. Since $\psi$ grows faster than $u_i-\omega$,
  there exists $R_\varepsilon$ such that
  $u_i\le\omega+\varepsilon\psi$ on $[0,T]\times\partial B_R$
  for every $R\ge R_\varepsilon$.
  The same inequality holds at time $T$, since $u_i(T,\cdot)=\omega(T,\cdot)$.
  Comparison on $[0,T]\times B_R$ therefore gives
  $u_i\le\omega+\varepsilon\psi$ inside each such cylinder.
  Every fixed point $(t,x)$ belongs to one of these cylinders by taking
  $R>\max\{R_\varepsilon,|x|\}$, so the inequality holds throughout
  $[0,T]\times\mathbb R^d$.
  Now let $\varepsilon\downarrow0$ at each fixed $(t,x)$ to obtain
  $u_i\le\omega$. Interchanging $u_i$ and $\omega$ gives the reverse
  inequality, hence $u_i(t,x)=\mathcal Y_t^{u,t,x}$.
  The same argument gives
  $\tilde u_i(t,x)=\mathcal Y_t^{\tilde u,t,x}$.

  Applying \cite[Theorem~1.5]{decreusefond_backward_1998} and the
  Lipschitz continuity of $f$ in the full vector $y$, we obtain
  \[
  |u_i(t,x)-\tilde u_i(t,x)|^2
  \le C\mathbb E\!\left[
  \int_t^T |u(r,X_r^{t,x})-\tilde u(r,X_r^{t,x})|^2\,dr
  \right].
  \]
  To control the difference of the full vectors $u$ and $\tilde u$
  uniformly in space despite their possible polynomial growth, we use
  a weighted supremum. Choose an integer $m$ larger than the growth
  exponents of $u$ and $\tilde u$, and define
  \[
  D(t):=\sup_{x\in\mathbb R^d}
  \frac{|u(t,x)-\tilde u(t,x)|^2}{(1+|x|^2)^m}.
  \]
  The weight makes $D$ finite and uniformly bounded in time. By definition,
  \[
  |u(r,X_r^{t,x})-\tilde u(r,X_r^{t,x})|^2
  \le D(r)(1+|X_r^{t,x}|^2)^m.
  \]
  Summing the component estimates over $i=1,\dots,k$ and inserting
  this bound gives
  \[
  |u(t,x)-\tilde u(t,x)|^2
  \le C\int_t^T D(r)\,
  \mathbb E[(1+|X_r^{t,x}|^2)^m]\,dr.
  \]
  By \cite[Theorem~1.3.15]{pham_continuous-time_2009} with $p=2m$, the
  forward moment estimate
  $\mathbb E[(1+|X_r^{t,x}|^2)^m]\le C(1+|x|^2)^m$
  is uniform in $t\le r\le T$. Dividing by $(1+|x|^2)^m$ therefore yields
  \[
  \frac{|u(t,x)-\tilde u(t,x)|^2}{(1+|x|^2)^m}
  \le C\int_t^T D(r)\,dr.
  \]
  Since the right-hand side is independent of $x$, taking the
  supremum over $x$ gives $D(t)\le C\int_t^T D(r)\,dr$.
  Backward Gronwall's inequality yields $D(t)=0$ for all $t\in[0,T]$, hence
  $u\equiv\tilde u$.
  }
\end{proof}

For a smooth path $W^n$, consider the PDE
\begin{equation}
  \partial_t u^n_i+\mathcal{L}_t u^n_i
  +{f_i}(t,x,u^n,\sigma^{\top}\nabla_xu^n_i)
  +{h_i}(t,x,u^n)\,\dot W^n=0,
  \qquad u^n(T,x)=\xi(x),\quad 1\leq i\leq k,
  \label{semilinear-smoothPDE}
\end{equation}
and the associated FBSDE
\begin{align}
  X^{t,x}_s
  ={}&x+\int_t^s b(r,X^{t,x}_r)\,dr
      +\int_t^s \sigma(r,X^{t,x}_r)\,dB_r,\nonumber\\
  Y^{n,t,x}_s
  ={}&\xi(X^{t,x}_T)
      +\int_s^T f(r,X^{t,x}_r,Y^{n,t,x}_r,Z^{n,t,x}_r)\,dr
      \nonumber\\
      &+\int_s^T h(r,X^{t,x}_r,Y^{n,t,x}_r)\,dW^n_r
      -\int_s^T Z^{n,t,x}_r\,dB_r.
  \nonumber
\end{align}

\begin{corollary}
  Under Assumptions 
$(\hyperlink{A:b-sigma}{b,\sigma}, \hyperlink{A:xi-f}{\xi,f})$
{and either \hyperlink{A:h}{$(h)$} or
\hyperlink{A:h'}{$(h')$}}, the PDE
\eqref{semilinear-smoothPDE} admits a unique viscosity solution 
given by
\[
u^n(t,x) = Y_t^{n,t,x},
\qquad (t,x) \in [0,T] \times \mathbb R^d.
\]
In particular, for each $(t,x)$, the random variable 
$Y_t^{n,t,x}$ is almost surely deterministic.
\end{corollary}

\begin{proof}
  Observe that the PDE \eqref{semilinear-smoothPDE} can be rewritten as
  \begin{equation*}
    \frac{\partial u^n}{\partial t} +\mathcal{L}u^n + f^n(t,x,u^n,
    \sigma^{\top} \nabla_x u^n) = 0, \qquad u^n (T, x) = \xi (x),
  \end{equation*}
  where $f^n(t,x,y,z)\assign f(t,x,y,z)+h(t,x,y)\dot W^n_t$.
  {Under either assumption on $h$, the modified generator
  $f^n$ is continuous in $(t,x,y,z)$, globally Lipschitz in $(x,y,z)$,
  and bounded at $(y,z)=(0,0)$, since $\dot W^n$ is continuous and
  bounded. The added term is independent of $z$, so the row-wise
  dependence required in Assumption $(\xi,f)$ is preserved.}
  The rest follows by Theorem \ref{multidimentional-visco}.
\end{proof}

\subsection{Skorokhod Limit in the Space of Decorated Paths}\label{Skorokhod-limit}
Assume that $\lim_{n \rightarrow \infty} \alpha_q (\jmath W^n, \jmath W) = 0$
for some sequence $(W^n)_{n \in \mathbb{N}} \subset D^q([0,T];\mathbb R^e)$.
For each $n$, define
\[
  u^n(t,x)\assign Y_t^{n,t,x},
  \qquad
  u(t,x)\assign Y_t^{t,x},
\]
where $Y^{n,t,x}$ denotes the solution to the rough FBSDE
\eqref{pairwise-Marcus-FBSDE} driven by $W^n$, and $Y^{t,x}$ denotes the
solution to \eqref{rough-FBSDE} driven by $W$. By
Proposition~\ref{Prop-RFBSDE-Markov-flow}, these random variables are
deterministic. We start by showing the easy pointwise convergence of $u^n$ to
$u$.
\begin{proposition} \label{Prop-pointwise-conv}
Under Assumptions 
$(\hyperlink{A:b-sigma}{b}, \hyperlink{A:b-sigma}{\sigma}, 
\hyperlink{A:xi-f}{\xi}, \hyperlink{A:xi-f}{f},
\hyperlink{A:h'}{h'}, \hyperlink{A:hC}{h^C})$ 
or Assumptions 
$(\hyperlink{A:b-sigma'}{b'}, \hyperlink{A:b-sigma'}{\sigma'}, 
\hyperlink{A:xi-f}{\xi}, \hyperlink{A:xi-f}{f},
\hyperlink{A:h}{h}, \hyperlink{A:hC}{h^C})$, assume that $\lim_{n
  \rightarrow \infty} \alpha_q (\jmath W^n, \jmath W) = 0$ for some sequence
  $(W^n)_{n \in \mathbb{N}} \subset D^q([0,T];\mathbb R^e)$. Then
  $u^n(t,x)$ converges to $u(t,x)=Y^{t,x}_t$ pointwise for all
  $t\in\mathcal C(W)$ and $x\in\mathbb R^d$.
\end{proposition}

\begin{proof}
Fix $t\in\mathcal C(W)$ and $x\in\mathbb R^d$. By
Lemma~\ref{restriction-alpha-metric}, we have
\[
\lim_{n\to\infty}\alpha_{q;[t,T]}(\jmath W^n,\jmath W)=0.
\]
Proposition~\ref{Prop-RFBSDE-stability} therefore yields
\[
\lim_{n \to \infty}\alpha_{p;[t,T]}(\mathbf Y^{n,t,x},\mathbf Y^{t,x})
= 0
\qquad\text{in probability}.
\]
{The initial-value term in the definition of
$\alpha_{p;[t,T]}$ gives}
\[
|Y_t^{n,t,x}-Y_t^{t,x}|
\leq
\alpha_{p;[t,T]}(\mathbf Y^{n,t,x},\mathbf Y^{t,x}).
\]
Since the left-hand side is
deterministic by Proposition~\ref{Prop-RFBSDE-Markov-flow}, we get for all
$t\in\mathcal C(W)$ and $x\in\mathbb R^d$ that
\begin{equation}
  \lim_{n \rightarrow \infty} | u^n (t, x) - u (t, x) | = \lim_{n \to \infty}
  | Y^{n, t, x}_t - Y^{t, x}_t | = 0. \label{eq-pointwise-limit}
\end{equation}
\end{proof}

In order to have the convergence in the Skorokhod metric on decorated paths,
we first need to consider the corresponding lifts. Set
$\mathbf{u}\assign\kappa_{-h,W}(u)$ and
$\mathbf{u}^n\assign\kappa_{-h,W^n}(u^n)$; see
Definition~\ref{Def-Marcus-lift}. The following theorem is first proved under
the additional assumption that the $\delta$-extensions of the approximating
lifts are continuous. The corollary below then shows, by smooth
approximation, that this assumption is automatic.

\begin{theorem}
  \label{theorem-Skorokhod}{Let $q\in[1,2)$ and $p>2$
  satisfy $1/p+1/q>1$, and let $W\in D^q([0,T];\mathbb R^e)$.}
  {Suppose that} either Assumptions 
$(\hyperlink{A:b-sigma}{b}, \hyperlink{A:b-sigma}{\sigma}, 
\hyperlink{A:xi-f}{\xi}, \hyperlink{A:xi-f}{f},
\hyperlink{A:h'}{h'}, \hyperlink{A:hC}{h^C})$ 
or Assumptions \linebreak
$(\hyperlink{A:b-sigma'}{b'}, \hyperlink{A:b-sigma'}{\sigma'}, 
\hyperlink{A:xi-f}{\xi}, \hyperlink{A:xi-f}{f},
\hyperlink{A:h}{h}, \hyperlink{A:hC}{h^C})$ hold.
  {Let $(W^n)_{n\in\mathbb N}\subset D^q([0,T];\mathbb R^e)$ satisfy}
  \[
    \lim_{n\to\infty}\alpha_q(\jmath W^n,\jmath W)=0.
  \]
  Suppose, in addition, that, for every $\delta>0$ and $n\in\mathbb N$, the
  $\delta$-extension $u^{\delta,n}$ of $\mathbf u^n$ is continuous. Then the
  Marcus lift $\mathbf{u}:= \kappa_{-h,W}(u) \in \mathfrak{D} ([0, T] ; C (\mathbb{R}^d;\mathbb{R}^k))$ of $u$ is the
  limit of $\mathbf{u}^n$ in the metric $\alpha_{\infty}$. Moreover, its
  $\delta$-extension $u^\delta$ is continuous and
  \begin{equation}
    u^{\delta}(t,x)=Y_t^{\delta,t,x},\qquad
    u^{\delta,n}(t,x)=Y_t^{\delta,n,t,x},
    \label{representation-delta-extension-u}
  \end{equation}
  where the processes on the right-hand side solve
  \eqref{pairwise-time-stretched-FBSDE} with drivers $W^\delta$ and
  $W^{\delta,n}$, respectively, starting from $(t,x)$ in extended time. Furthermore,
  $u(t)\in\mathrm{BUC}(\mathbb{R}^d;\mathbb{R}^k)$ for every $t\in[0,T]$.
\end{theorem}

\begin{proof}
 We first verify the second identity in
 \eqref{representation-delta-extension-u}. The Markov representation gives,
 for $s>t$,
 \[
   Y_s^{n,t,x}=u^n(s,X_s^{t,x}).
 \]
 Since $X_s^{t,x}\to x$ as $s\searrow t$, continuity of $u^{\delta,n}$
 implies, in probability,
 \[
   Y_{t+}^{n,t,x}=u^n(t+,x).
 \]
 {Now consider an inserted interval
 $[\tau^\delta(t),\tau^\delta(t+)]$ and let $s$ belong to this interval.}
 By the definition of the
 Marcus lift,
 \[
   u^{\delta,n}(s,x)
   =\varphi\left(
     {-h(t,x,\cdot)\Delta W_t^n},u^n(t+,x),
     \frac{s-\tau^\delta(t)}{\tau^\delta(t+)-\tau^\delta(t)}
   \right).
 \]
 while \eqref{delta-extension-Y}, applied to the FBSDE started from $(t,x)$,
 gives the same expression with $Y_{t+}^{n,t,x}$ in place of $u^n(t+,x)$.
 Hence
 \[
   u^{\delta,n}(s,x)=Y_s^{\delta,n,\tau^\delta(t),x}.
 \]
 Since the time-changed forward process is constant on the inserted interval,
 the backward-flow property for the time-stretched FBSDE yields
 $Y_s^{\delta,n,\tau^\delta(t),x}=Y_s^{\delta,n,s,x}$.
 Thus $u^{\delta,n}(s,x)=Y_s^{\delta,n,s,x}$ on all of $[0,T+\delta]$.

 Fix a countable set $\Pi \subset [0,T]$ containing all non-stationary points of $W, W^1, W^2, \ldots$.
 {By Lemma~\ref{alternative-a-metric}, for every $\delta>0$
 there exists a sequence $(\lambda^n_\delta)_{n\in\mathbb N}
 \subset\Lambda_{[0,T+\delta]}$ such that}
  \begin{align}
      \lim_{n \rightarrow \infty}
      {|W_0-W_0^n|}
      +\| W^{\delta} - W^{n, \delta} \circ
  \lambda^n_{\delta} \|_q + \| c^{\delta, \Pi} - c^{\delta, \Pi} \circ
  \lambda^n_{\delta} \|_{\infty} = 0. \label{property-lambda-n}
  \end{align}
  Due to the isomorphism $C ([0, T +
  \delta] \times \mathbb{R}^d) \simeq C ([0, T + \delta] ; C (\mathbb{R}^d))$
  under the compact-open topology, our aim is to prove
  \[
    \lim_{n \rightarrow \infty}
    \sup_{(t,x)\in[0,T+\delta]\times K}
    \left|u^{\delta}(t,x)
    -u^{n,\delta}(\lambda^n_{\delta}(t),x)\right|=0
  \]
  for every compact set $K \subset \mathbb{R}^d$ under Assumptions
  $(W,b,\sigma,\xi,f,h')$. Fix an arbitrary compact set
  $K\subset\mathbb R^d$. We first obtain a continuous subsequential limit
  $\bar u^\delta$ on $[0,T+\delta]\times K$ and then identify it with
  $u^\delta$ there; in particular, the continuity of $u^\delta$ is a
  conclusion of the argument.
We proceed as follows.  \\
In \textit{part~1)}, we apply the Arzelà--Ascoli theorem to extract a
subsequence converging in $C([0,T+\delta]\times K)$ to a limit denoted
by $\bar u^\delta$.\\
In \textit{part~2)}, we prove that $\bar u^\delta=u^\delta$ on
$[0,T+\delta]\times K$ and use uniqueness of the cluster point to obtain
convergence of the full sequence on this set. \\
In \textit{part~3)}, we show that $u(t) \in \mathrm{BUC} (\mathbb{R}^d)$ for all $t \in [0,T]$. \\

\noindent \textbf{Part 1:}  \\
We first record the uniform bound that will also be used in the
equicontinuity argument. By
{\cite[Theorem~3.2]{becherer_rough_2026}}, together with
Proposition~\ref{Prop-RFBSDE-wellposedness}, there exists a constant
$C_\delta>0$ such that
\begin{equation}
\sup_{n\in\mathbb N}
\sup_{a\in[0,T+\delta]}
\sup_{x\in\mathbb R^d}
\left(
\|Y^{n,\delta,a,x}\|_{p,2;[a,T+\delta]}
+
\|Z^{n,\delta,a,x}\|_{\mathrm{BMO};[a,T+\delta]}
\right)
\leq C_\delta.
\label{uniform-YZ-bound}
\end{equation}
Indeed, under Assumption $(h')$, the Young coefficient is independent
of the forward process, whereas under Assumptions $(b',\sigma',h)$ the
required uniform forward estimate follows from
Proposition~\ref{Prop-random-initial-SDE-estimates}.
Moreover, for every $(t,x)\in[0,T+\delta]\times\mathbb R^d$, the
stochastic representation gives
\[
\left|u^{n,\delta}(\lambda_\delta^n(t),x)\right|
=
\left|
Y_{\lambda_\delta^n(t)}^{
n,\delta,\lambda_\delta^n(t),x}
\right|
\leq
\left\|
Y^{n,\delta,\lambda_\delta^n(t),x}
\right\|_{p,2;[\lambda_\delta^n(t),T+\delta]}
+\|\xi\|_\infty.
\]
Hence the preceding estimate implies
\begin{equation}
M_\delta
:=
\sup_{n\in\mathbb N}
\sup_{(t,x)\in[0,T+\delta]\times\mathbb R^d}
|u^{n,\delta}(\lambda_\delta^n(t),x)|
<\infty.
\label{uniform-bound-vn}
\end{equation}
To show equicontinuity of
$\bigl(u^{n,\delta}(\cdot,x)\circ\lambda^n_\delta\bigr)_{n\in\mathbb N}$,
we first prove spatial equicontinuity, uniformly in time. For $\rho>0$,
define
\[
\omega^x(\rho)
:=
\sup_{n\in\mathbb N}
\sup_{t\in[0,T+\delta]}
\sup_{\substack{x,y\in\mathbb R^d\\|x-y|\leq\rho}}
\left|
u^{n,\delta}(\lambda_\delta^n(t),x)
-
u^{n,\delta}(\lambda_\delta^n(t),y)
\right|.
\]
We claim that
\begin{equation}
\lim_{\rho\downarrow0}\omega^x(\rho)=0.
\label{spatial-modulus-vn}
\end{equation}
Suppose otherwise. Then there exist $\varepsilon>0$ and sequences
$n_k\in\mathbb N$, $t_k\in[0,T+\delta]$, and
$x_k,y_k\in\mathbb R^d$ such that
$|x_k-y_k|\to0$, but
\[
\left|
u^{n_k,\delta}(\lambda_\delta^{n_k}(t_k),x_k)
-
u^{n_k,\delta}(\lambda_\delta^{n_k}(t_k),y_k)
\right|
\geq\varepsilon.
\]
By the stochastic representation and
Proposition~\ref{prop:uniform-pairwise-stability}, however,
\[
\begin{aligned}
\left|
u^{n_k,\delta}(\lambda_\delta^{n_k}(t_k),x_k)
-
u^{n_k,\delta}(\lambda_\delta^{n_k}(t_k),y_k)
\right| =
\left|
Y_{\lambda_\delta^{n_k}(t_k)}^{
n_k,\delta,\lambda_\delta^{n_k}(t_k),x_k}
-
Y_{\lambda_\delta^{n_k}(t_k)}^{
n_k,\delta,\lambda_\delta^{n_k}(t_k),y_k}
\right|
\longrightarrow0,
\end{aligned}
\]
which is a contradiction and proves \eqref{spatial-modulus-vn}.

We now prove time equicontinuity (uniformly in space), that is,
\begin{equation}
\lim_{\eta\downarrow0}
\sup_{n\in\mathbb N}
\sup_{x\in K}
\sup_{\substack{s,t\in[0,T+\delta]\\|t-s|\leq\eta}}
\left|
u^{n,\delta}(\lambda_\delta^n(t),x)
-u^{n,\delta}(\lambda_\delta^n(s),x)
\right|
=0.
\label{time-modulus-vn}
\end{equation}
Fix $0\leq s\leq t\leq T+\delta$ and $x\in K$.
Using the Markovian representation in
Proposition~\ref{Prop-RFBSDE-Markov-flow} at times
$\lambda_\delta^n(s)$ and $\lambda_\delta^n(t)$, we obtain
\begin{align*}
&|u^{n,\delta}(\lambda_\delta^n(t),x)-u^{n,\delta}(\lambda_\delta^n(s),x)| \\
&\quad\leq
\left|
\mathbb E\left[
u^{n,\delta}(\lambda_\delta^n(t),x)
-
u^{n,\delta}\bigl(\lambda_\delta^n(t),
X_{\lambda_\delta^n(t)}^{\delta,\lambda_\delta^n(s),x}\bigr)
\right]
\right|\\
&\qquad+
\left|
\mathbb E\left[
Y_{\lambda_\delta^n(t)}^{n,\delta,\lambda_\delta^n(s),x}
-Y_{\lambda_\delta^n(s)}^{n,\delta,\lambda_\delta^n(s),x}
\right]
\right|\quad
=:\mathrm I+\mathrm{II}.
\end{align*}
For term~$\mathrm I$, Proposition~\ref{Prop-random-initial-SDE-estimates} yields
\[
\mathbb E\left[
\left|
X_{\lambda_\delta^n(t)}^{\delta,\lambda_\delta^n(s),x}-x
\right|^2
\right]
\leq
C_K\bigl(
c^\delta(\lambda_\delta^n(t))
-c^\delta(\lambda_\delta^n(s))
\bigr),
\]
uniformly in $n$, $s,t$, and $x\in K$. Under Assumptions
$(b',\sigma')$, the constant may be chosen independently of $K$. For every $\rho\in(0,1]$, splitting according to
$\{|X_{\lambda_\delta^n(t)}^{\delta,\lambda_\delta^n(s),x}-x|\leq\rho\}$ and using
\eqref{uniform-bound-vn}, \eqref{spatial-modulus-vn}, and Markov's
inequality, we obtain
\begin{align}
\mathrm I
&\leq
\omega^x(\rho)
+
2M_\delta
\mathbb P\left(
\left|X_{\lambda_\delta^n(t)}^{\delta,\lambda_\delta^n(s),x}-x\right|>\rho
\right)
\nonumber\\
&\leq
\omega^x(\rho)
+
\frac{2M_\delta C_K}{\rho^2}
\bigl(
c^\delta(\lambda_\delta^n(t))
-c^\delta(\lambda_\delta^n(s))
\bigr).
\label{term-I-estimate}
\end{align}
For term~$\mathrm{II}$, from the BSDE integral equation we have
\begin{align*}
\mathrm{II}
\leq{} &
\mathbb E\left[
\int_{\lambda_\delta^n(s)}^{\lambda_\delta^n(t)}
\left|f^\delta\left(
r,X_r^{\delta,\lambda_\delta^n(s),x},
Y_r^{n,\delta,\lambda_\delta^n(s),x},Z_r^{n,\delta,\lambda_\delta^n(s),x}
\right)\right|\,dc_r^\delta
\right] \\ &+
\mathbb E\left[
\left|
\int_{\lambda_\delta^n(s)}^{\lambda_\delta^n(t)}
\widehat h^\delta\left(
r,X_r^{\delta,\lambda_\delta^n(s),x},Y_r^{n,\delta,\lambda_\delta^n(s),x}
\right)\,dW_r^{n,\delta}
\right|
\right].
\end{align*}
Since $f^\delta(\cdot,\cdot,0,0)$ is bounded and $f^\delta$ is
uniformly Lipschitz in $(y,z)$,
\[
\begin{aligned}
&\left|f^\delta\left(
r,X_r^{\delta,\lambda_\delta^n(s),x},
Y_r^{n,\delta,\lambda_\delta^n(s),x},Z_r^{n,\delta,\lambda_\delta^n(s),x}
\right)\right|
\\
&\quad\leq
C\left(1+
\left|Y_r^{n,\delta,\lambda_\delta^n(s),x}\right|
+\left|Z_r^{n,\delta,\lambda_\delta^n(s),x}\right|
\right).
\end{aligned}
\]
The uniform bounds in \eqref{uniform-YZ-bound} give
\begin{align}
&\mathbb E\left[
\int_{\lambda_\delta^n(s)}^{\lambda_\delta^n(t)}
\left|f^\delta\left(
r,X_r^{\delta,\lambda_\delta^n(s),x},
Y_r^{n,\delta,\lambda_\delta^n(s),x},Z_r^{n,\delta,\lambda_\delta^n(s),x}
\right)\right|\,dc_r^\delta
\right]
\nonumber\\
&\quad\leq
C\left(
c^\delta(\lambda_\delta^n(t))
-c^\delta(\lambda_\delta^n(s))
\right)\nonumber\\
&\qquad+
C\left(
c^\delta(\lambda_\delta^n(t))
-c^\delta(\lambda_\delta^n(s))
\right)^{1/2}
\left(
\mathbb E\left[
\int_{\lambda_\delta^n(s)}^{\lambda_\delta^n(t)}
\left|Z_r^{n,\delta,\lambda_\delta^n(s),x}\right|^2\,dc_r^\delta
\right]
\right)^{1/2}
\nonumber\\
&\quad\leq
C\left(
c^\delta(\lambda_\delta^n(t))
-c^\delta(\lambda_\delta^n(s))
+
\left(
c^\delta(\lambda_\delta^n(t))
-c^\delta(\lambda_\delta^n(s))
\right)^{1/2}
\right).
\label{f-time-estimate}
\end{align}
For the Young integral, put
\[
H_r^{n,s,x}
:=
\widehat h^\delta\left(
r,X_r^{\delta,\lambda_\delta^n(s),x},Y_r^{n,\delta,\lambda_\delta^n(s),x}
\right).
\]
The Young estimate {\cite[Proposition~A.1]{becherer_rough_2026}} gives
\[
\begin{aligned}
&\left|
\int_{\lambda_\delta^n(s)}^{\lambda_\delta^n(t)}H_r^{n,s,x}\,dW_r^{n,\delta}
\right|
\\
&\quad\leq
C_{p,q}
\left(
|H_{\lambda_\delta^n(s)}^{n,s,x}|
+
\|H^{n,s,x}\|_{p\text{-var};[\lambda_\delta^n(s),\lambda_\delta^n(t)]}
\right)
\|W^{n,\delta}\|_{q\text{-var};[\lambda_\delta^n(s),\lambda_\delta^n(t)]}.
\end{aligned}
\]
The uniform bounds in \eqref{uniform-YZ-bound}, together
with Assumption $(h')$ in the state-independent case and
Lemma~\ref{Lemma-hX-Assumption-Ad} and
Proposition~\ref{Prop-random-initial-SDE-estimates} under Assumptions
$(b',\sigma',h)$, imply
\[
\sup_{n,s,x}
\left(
\mathbb E\left[
|H_{\lambda_\delta^n(s)}^{n,s,x}|^2
\right]
+
\mathbb E\left[
\|H^{n,s,x}\|_{p\text{-var};[\lambda_\delta^n(s),T+\delta]}^2
\right]
\right)^{1/2}
<\infty.
\]
Taking expectations in the Young estimate therefore gives
\begin{equation}
\mathbb E\left[
\left|
\int_{\lambda_\delta^n(s)}^{\lambda_\delta^n(t)}H_r^{n,s,x}\,dW_r^{n,\delta}
\right|
\right]
\leq
C
\|W^{n,\delta}\|_{q\text{-var};[\lambda_\delta^n(s),\lambda_\delta^n(t)]}.
\label{h-time-estimate}
\end{equation}
Combining \eqref{f-time-estimate} and \eqref{h-time-estimate}, we obtain
\[
\mathrm{II}
\leq
C\left(
c^\delta(\lambda_\delta^n(t))
-c^\delta(\lambda_\delta^n(s))
+
\left(
c^\delta(\lambda_\delta^n(t))
-c^\delta(\lambda_\delta^n(s))
\right)^{1/2}
+
\|W^{n,\delta}\|_{q\text{-var};[\lambda_\delta^n(s),\lambda_\delta^n(t)]}
\right).
\]
Together with \eqref{term-I-estimate}, this yields, for every
$\eta>0$ and $\rho\in(0,1]$,
\begin{align}
&\sup_{n\in\mathbb N}
\sup_{x\in K}
\sup_{\substack{0\leq s\leq t\leq T+\delta\\t-s\leq\eta}}
|u^{n,\delta}(\lambda_\delta^n(t),x)-u^{n,\delta}(\lambda_\delta^n(s),x)|
\nonumber\\
&\quad\leq
\omega^x(\rho)
+
\frac{2M_\delta C_K}{\rho^2}\alpha_\delta(\eta)
+
C\left(
\alpha_\delta(\eta)
+
\alpha_\delta(\eta)^{1/2}
+
\beta_\delta(\eta)
\right), \label{final-time-equicontinuity-estimate}
\end{align}
where we used the notation
\begin{align*}
\alpha_\delta(\eta)
&:=
\sup_{n\in\mathbb N}
\sup_{\substack{0\leq s\leq t\leq T+\delta\\t-s\leq\eta}}
\left[
c^\delta\bigl(\lambda_\delta^n(t)\bigr)
-
c^\delta\bigl(\lambda_\delta^n(s)\bigr)
\right],\\
\beta_\delta(\eta)
&:=
\sup_{n\in\mathbb N}
\sup_{\substack{0\leq s\leq t\leq T+\delta\\t-s\leq\eta}}
\|W^{n,\delta}\|_{q\text{-var};
[\lambda_\delta^n(s),\lambda_\delta^n(t)]}.
\end{align*}
Lemma~\ref{equicontinuity-p-var} gives the required convergence
uniformly for all sufficiently large $n$. For each of the finitely many
remaining indices, the continuity of $c^\delta\circ\lambda_\delta^n$
and the continuity of the $q$-variation control of
$W^{n,\delta}\circ\lambda_\delta^n$ give the same conclusion.
Consequently,
\begin{equation}
\alpha_\delta(\eta)+\beta_\delta(\eta)
\longrightarrow0
\qquad\text{as }\eta\downarrow0.
\label{clock-driver-moduli}
\end{equation}
Letting first $\eta\downarrow0$ and then $\rho\downarrow0$ in \eqref{final-time-equicontinuity-estimate}, using
\eqref{spatial-modulus-vn} and \eqref{clock-driver-moduli}, proves
\eqref{time-modulus-vn}.

Finally, the following elementary inequality
\begin{align*}
|u^{n,\delta}(\lambda_\delta^n(t),x)-u^{n,\delta}(\lambda_\delta^n(s),y)|
\leq{}&
|u^{n,\delta}(\lambda_\delta^n(t),x)-u^{n,\delta}(\lambda_\delta^n(t),y)|\\
&+
|u^{n,\delta}(\lambda_\delta^n(t),y)-u^{n,\delta}(\lambda_\delta^n(s),y)|,
\end{align*}
together with the spatial and time equicontinuity, implies joint equicontinuity on
$[0,T+\delta]\times K$. Together with the uniform bound in
\eqref{uniform-bound-vn}, the general Arzelà--Ascoli theorem
\cite[p.~236]{kelley_general_1975} shows that
$\bigl(u^{n,\delta}(\lambda_\delta^n(\cdot),\cdot)\bigr)_n$ is relatively compact in
$C([0,T+\delta]\times K)$. Consequently, there exists a function $\bar u^\delta
\in
C([0,T+\delta]\times K)$
such that (possibly after passing to a subsequence)
\begin{equation}
\lim_{k\to\infty}
\sup_{(t,x)\in[0,T+\delta]\times K}
\left|
\bar u^\delta(t,x)
-
u^{n,\delta}
\bigl(\lambda_\delta^{n}(t),x\bigr)
\right|
=0.
\label{arzela-ascoli-result}
\end{equation}

  \noindent \textbf{Part 2:} 
  We show that $\bar u^\delta=u^\delta$ on
  $[0,T+\delta]\times K$. \\
  Define $\bar{u} \in D ([0, T] ; C(K))$ by $\bar{u} (t, x)
  \assign \bar{u}^{\delta} (\tau^{\delta} (t), x)$. By
  Lemma~\ref{Lemma-localuniform} and \eqref{arzela-ascoli-result}, for every
  $t\in\mathcal C(W)\subset\mathcal C(\bar u)$, it holds that
  \[
  \sup_{x\in K}|u^n(t,x)-\bar u(t,x)|\longrightarrow0.
  \]
  Combining this with \eqref{eq-pointwise-limit}, we derive that $\bar{u} (t,
  x) = u (t, x) = Y^{t, x}_t$ for every $t \in \mathcal{C} (W)$ and
  $x\in K$. For a c{\`a}gl{\`a}d path, it is known that for any $t \in [0, T]$, there exists a sequence $(t^+_m)_{m \in
  \mathbb{N}} \subset \mathcal{C} (W)$ such that $t^+_m \searrow t$. Hence,
  \begin{equation*}
    \bar{u} (t +, x) = \lim_{m \rightarrow \infty} \bar{u} (t^+_m, x) =
    \lim_{m \rightarrow \infty} u (t^+_m, x) = u (t +, x) .
  \end{equation*}
  Similarly, one can find a sequence $(t^-_m)_{m \in \mathbb{N}} \subset
  \mathcal{C} (W)$ such that $t^-_m \nearrow t$, which yields
  \[ \bar{u} (t, x) = \lim_{m \rightarrow \infty} \bar{u} (t^-_m, x) = \lim_{m
     \rightarrow \infty} u (t^-_m, x) = u (t, x) \]
  for any $t \in [0,T]$ and $x\in K$. It therefore remains to prove equality on every
  inserted interval $[\tau^\delta(t),\tau^\delta(t+)]$. \\
  Fix \(t\notin\mathcal C(W)\),
  \(\bar t\in[\tau^\delta(t),\tau^\delta(t+)]\), and
  \(x\in K\). By \eqref{arzela-ascoli-result}, we have
  \begin{equation}
  u^{n,\delta}\bigl(\lambda^{n,\delta}(\bar t);T,\xi\bigr)(x)
  \longrightarrow \bar u^\delta(\bar t;T,\xi)(x).
  \label{excursion-arzela-convergence}
  \end{equation}
  At the right endpoint of the excursion interval, the same convergence
  and the identity
  \(\bar u^\delta(\tau^\delta(t+),x)=u(t+,x)\) give
  \begin{equation}
  u^{n,\delta}\bigl(
  \lambda^{n,\delta}(\tau^\delta(t+));T,\xi
  \bigr)(x)
  \longrightarrow u(t+,x),
  \label{un-convergence-tplus}
  \end{equation}
  uniformly for \(x\in K\). Moreover,
  Proposition~\ref{Prop-random-initial-SDE-estimates} gives
  \[
  \mathbb E\left[
  \left|
  X_{\lambda^{n,\delta}(\tau^\delta(t+))}^{
  \delta,\lambda^{n,\delta}(\bar t),x}-x
  \right|^2
  \right]^{1/2}
  \leq
  C(1+|x|)
  \left(
  c^\delta(\lambda^{n,\delta}(\tau^\delta(t+)))
  -c^\delta(\lambda^{n,\delta}(\bar t))
  \right)^{1/2}
  \xrightarrow{n\to\infty}0.
  \]
  Set
  \[
  r_n:=\lambda^{n,\delta}(\tau^\delta(t+)),\qquad
  \zeta_n:=
  X_{r_n}^{\delta,\lambda^{n,\delta}(\bar t),x}.
  \]
  Apply Proposition~\ref{prop:uniform-pairwise-stability} on
  $[r_n,T+\delta]$ to the two time-stretched FBSDEs starting from
  $\zeta_n$ and $x$. The preceding estimate gives
  $\mathbb E\left[|\zeta_n-x|^2\right]\to0$, and hence
  \[
  \left|
  u^{n,\delta}(r_n;T,\xi)(\zeta_n)
  -u^{n,\delta}(r_n;T,\xi)(x)
  \right|\longrightarrow0
  \qquad\text{in probability}.
  \]
  The two terms are uniformly bounded, so this convergence also holds in
  $L^2$. Combining it with \eqref{un-convergence-tplus} yields
  \begin{equation}
  \mathbb E\left[\left|
  u^{n,\delta}(r_n;T,\xi)(\zeta_n)-u(t+,x)
  \right|^2\right]^{1/2}
  \xrightarrow{n\to\infty}0.
  \label{excursion-terminal-L2-convergence}
  \end{equation}
  We postpone the remaining technical step
  \begin{equation*}
  \left\|
  Y_{\lambda^{n,\delta}(\cdot)}^{
  n,\delta,\lambda^{n,\delta}(\bar t),x}
  -Y_\cdot^{\delta,\bar t,x}
  \right\|_{p;[\bar t,\tau^\delta(t+)]}
  \xrightarrow{n\to\infty}0
  \qquad\text{in probability}
  \end{equation*}
  to Lemma~\ref{Lemma-excursion-Y-convergence}. Combining this result
  with \eqref{excursion-terminal-L2-convergence}, we obtain
  \begin{align*}
  &\left|
  Y_{\lambda^{n,\delta}(\bar t)}^{
  n,\delta,\lambda^{n,\delta}(\bar t),x}
  -Y_{\bar t}^{\delta,\bar t,x}
  \right|\\
  &\quad\leq
  \left|
  Y_{\lambda^{n,\delta}(\tau^\delta(t+))}^{
  n,\delta,\lambda^{n,\delta}(\bar t),x}
  -Y_{\tau^\delta(t+)}^{\delta,\bar t,x}
  \right|
  +
  \left\|
  Y_{\lambda^{n,\delta}(\cdot)}^{
  n,\delta,\lambda^{n,\delta}(\bar t),x}
  -Y_\cdot^{\delta,\bar t,x}
  \right\|_{p;[\bar t,\tau^\delta(t+)]}
  \xrightarrow{n\to\infty}0
  \end{align*}
  in probability. Since
  \[
  Y_{\lambda^{n,\delta}(\bar t)}^{
  n,\delta,\lambda^{n,\delta}(\bar t),x}
  =
  u^{n,\delta}\bigl(\lambda^{n,\delta}(\bar t);T,\xi\bigr)(x),
  \]
  comparison with \eqref{excursion-arzela-convergence} gives
  \begin{equation}
  \bar u^\delta(\bar t;T,\xi)(x)
  =
  Y_{\bar t}^{\delta,\bar t,x}.
  \label{excursion-identification}
  \end{equation}
  Since \(\bar t\) was arbitrary, \eqref{excursion-identification} holds
  throughout \([\tau^\delta(t),\tau^\delta(t+)]\). On this interval,
  \(c^\delta\) and \(B^\delta\) are constant and the forward process
  remains equal to \(x\). By the backward flow property, the values on
  the right-hand side of \eqref{excursion-identification} form the
  unique solution \(V\) of
  \[
  V_s
  =
  u(t+,x)
  +
  \int_s^{\tau^\delta(t+)}
  h^\delta(r,x,V_r)\,dW_r^\delta,
  \qquad
  s\in[\tau^\delta(t),\tau^\delta(t+)].
  \]
  Since \(W^\delta\) is linear and
  \(h^\delta(r,\cdot,\cdot)=h(t,\cdot,\cdot)\) on the inserted interval,
  \(V\) is precisely the Marcus excursion associated with the jump at
  \(t\). Consequently,
  \[
  \bar u^\delta(s,x)=u^\delta(s,x),
  \qquad s\in[\tau^\delta(t),\tau^\delta(t+)].
  \]
  Together with \(\bar u(t,x)=u(t,x)\) for every physical time \(t\), this
  proves
  \[
  \bar u^\delta(s,x)=u^\delta(s,x),
  \qquad (s,x)\in[0,T+\delta]\times K.
  \]
  Since the preceding compactness and identification argument applies to
  every subsequence, every subsequence of
  $\bigl(u^{n,\delta}(\lambda_\delta^n(\cdot),\cdot)\bigr)_n$ has a further
  subsequence converging uniformly on $[0,T+\delta]\times K$ to
  $u^\delta$. Consequently,
  \[
  \lim_{n\to\infty}
  \sup_{(s,x)\in[0,T+\delta]\times K}
  \left|
  u^{n,\delta}(\lambda_\delta^n(s),x)-u^\delta(s,x)
  \right|
  =0.
  \]
  Since \(K\) was arbitrary, this proves the desired locally uniform
  convergence and the continuity of \(u^\delta\). Together with
  \eqref{property-lambda-n} and the definition of $\alpha_\infty$, it also
  proves that $\mathbf u^n\to\mathbf u$ in $\alpha_\infty$. Moreover,
  \(\eqref{excursion-identification}\), together with the
  identity $u(t,x)=Y_t^{t,x}$ at physical times, gives the first identity in
  \eqref{representation-delta-extension-u}.

  \noindent \textbf{Part 3:} By Proposition~\ref{Prop-RFBSDE-wellposedness}, we obtain
\[
\sup_{x} |u(t,x)|
= \sup_{x} |Y^{t,x}_t|
\le \sup_{x} |\xi(X^{t,x}_T)|
   + \sup_{x} \|Y^{t,x}\|_{p,2}
< \infty.
\]
To prove uniform continuity, fix $t\in[0,T]$ and let
$x_n,y_n\in\mathbb R^d$ satisfy $|x_n-y_n|\to0$. Apply
Proposition~\ref{prop:uniform-pairwise-stability} with the constant driver
$W$, the fixed initial time $t$, and the deterministic initial states
$x_n,y_n$. It follows that
\[
|u(t,x_n)-u(t,y_n)|\longrightarrow0
\qquad\text{in probability}.
\]
The quantity on the left is deterministic, and hence converges to zero.
The sequential characterization of uniform continuity therefore shows that
$u(t)\in\mathrm{BUC}(\mathbb R^d;\mathbb R^k)$.
\end{proof}

\begin{corollary}
  \label{Cor-Skorokhod-general-drivers}The continuity assumption on the
  extensions $u^{\delta,n}$ in Theorem~\ref{theorem-Skorokhod} is automatic.
  Consequently, the conclusion of that theorem holds for every sequence
  $(W^n)_{n\in\mathbb N}\subset D^q([0,T];\mathbb R^e)$ satisfying
  \[
    \alpha_q(\jmath W^n,\jmath W)\longrightarrow0.
  \]
\end{corollary}

\begin{proof}
  If the drivers $W^n$ are smooth, then the corresponding extensions
  $u^{\delta,n}$ are continuous, and
  Theorem~\ref{theorem-Skorokhod} applies. Now fix
  $\widetilde W\in D^q([0,T];\mathbb R^e)$. Choose $q'>q$ sufficiently close
  to $q$ that $1/p+1/q'>1$. By the inclusions on
  p.~\pageref{smooth-path-closures}, there are smooth paths
  $\widetilde W^m$ converging to $\widetilde W$ in $\alpha_{q'}$. Applying
  the smooth case with limit $\widetilde W$ shows that the extension
  associated with $\widetilde W$ is continuous and satisfies
  \eqref{representation-delta-extension-u}. In particular, this holds for
  every $W^n$. Theorem~\ref{theorem-Skorokhod} can therefore be applied once
  more to the original sequence $(W^n)$, which proves the assertion.
\end{proof}

We next state the technical lemma needed in the preceding proof. Its
conclusion is analogous to the stability result
\cite[(4.14)]{becherer_rough_2026}. A direct application of that result
is, however, not possible because the RBSDEs underlying
\(Y^{n,\delta,\lambda^{n,\delta}(\bar t),x}\) are posed on the varying
intervals \([\lambda^{n,\delta}(\bar t),T+\delta]\). There is a
canonical way to extend these equations to the whole interval
\([0,T+\delta]\) by modifying \(f\) and \(W\), but the resulting
generators no longer satisfy the assumptions required by the stability
theorem in \cite{becherer_rough_2026}. We therefore follow the same
double Picard argument that underlies \cite[(4.14)]{becherer_rough_2026}
and refer to the corresponding arguments there whenever the necessary
adaptation is immediate.

\begin{lemma}
\label{Lemma-excursion-Y-convergence}
{Let $q\in[1,2)$ and $p>2$ satisfy $1/p+1/q>1$, and let
$W,W^n\in D^q([0,T];\mathbb R^e)$ for $n\in\mathbb N$.}
Suppose that Assumptions
$(\hyperlink{A:b-sigma'}{b'},
\hyperlink{A:b-sigma'}{\sigma'},\hyperlink{A:xi-f}{\xi},
\hyperlink{A:xi-f}{f},\hyperlink{A:h}{h},
\hyperlink{A:hC}{h^C})$ {or Assumptions
$(\hyperlink{A:b-sigma}{b},
\hyperlink{A:b-sigma}{\sigma},\hyperlink{A:xi-f}{\xi},
\hyperlink{A:xi-f}{f},\hyperlink{A:h'}{h'},
\hyperlink{A:hC}{h^C})$} hold. Fix \(\delta>0\), and let
\((\lambda^{n,\delta})_{n\in\mathbb N}\) be the reparametrizations
satisfying \eqref{property-lambda-n}. Then, for every jump time
\(t\notin\mathcal C(W)\), every
\(\bar t\in[\tau^\delta(t),\tau^\delta(t+)]\), and every
\(x\in\mathbb R^d\),
\[
\lim_{n\to\infty}\left\|
Y_{\lambda^{n,\delta}(\cdot)}^{n,\delta,
\lambda^{n,\delta}(\bar t),x}
-Y_\cdot^{\delta,\bar t,x}
\right\|_{p;[\bar t,\tau^\delta(t+)]}
=0
\qquad\text{in probability}.
\]
\end{lemma}

\begin{proof}
Fix \(t\notin\mathcal C(W)\),
\(\bar t\in[\tau^\delta(t),\tau^\delta(t+)]\), and
\(x\in\mathbb R^d\), and set
\[
I_{\bar t}:=[\bar t,\tau^\delta(t+)],
\qquad
I_k:=\bigl[\lambda^{k,\delta}(\bar t),
\lambda^{k,\delta}(\tau^\delta(t+))\bigr],\quad k\in\mathbb N.
\]
For \(k\in\mathbb N\), abbreviate
\[
X_r^{k,\delta}
:=
X_r^{\delta,\lambda^{k,\delta}(\bar t),x},
\qquad
(Y_r^{k,\delta},Z_r^{k,\delta})
:=
(Y_r^{k,\delta,\lambda^{k,\delta}(\bar t),x},
Z_r^{k,\delta,\lambda^{k,\delta}(\bar t),x}).
\]
Also write
\[
f_r^{k,\delta}(y,z):=f^\delta(r,X_r^{k,\delta},y,z),
\qquad
g_r^{k,\delta}(y):=h^\delta(r,X_r^{k,\delta},y),
\]
\[
\zeta_{\bar t}^{k,\delta}
:=
Y_{\lambda^{k,\delta}(\tau^\delta(t+))}^{k,\delta}
=
u^{k,\delta}\bigl(
\lambda^{k,\delta}(\tau^\delta(t+));T,\xi
\bigr)
\bigl(X_{\lambda^{k,\delta}(\tau^\delta(t+))}^{k,\delta}\bigr).
\]
By the backward flow property, for \(r\in I_k\), we have
\begin{align}
Y_r^{k,\delta}
={}&
\zeta_{\bar t}^{k,\delta}
+\int_r^{\lambda^{k,\delta}(\tau^\delta(t+))}
f_q^{k,\delta}(Y_q^{k,\delta},Z_q^{k,\delta})\,dc_q^\delta
\notag\\
&+\int_r^{\lambda^{k,\delta}(\tau^\delta(t+))}
g_q^{k,\delta}(Y_q^{k,\delta})\,dW_q^{k,\delta}
-\int_r^{\lambda^{k,\delta}(\tau^\delta(t+))}
Z_q^{k,\delta}\,dB_q^\delta.
\label{excursion-local-approximating-BSDE}
\end{align}
On \(I_{\bar t}\), \(c^\delta\) and \(B^\delta\) are constant and the
limiting forward process equals \(x\). Thus
\begin{equation}
Y_s^{\delta,\bar t,x}
=
u(t+,x)
+\int_s^{\tau^\delta(t+)}
h^\delta(r,x,Y_r^{\delta,\bar t,x})\,dW_r^\delta,
\qquad s\in I_{\bar t}.
\label{excursion-local-limiting-ODE}
\end{equation}
This is a deterministic ODE because \(W^\delta\) is linear on the
inserted interval.

Write \(Y^{\infty,\delta}:=Y^{\delta,\bar t,x}\), and let
\((Y^{k,m,\delta},Z^{k,m,\delta})_{m\in\mathbb N_0}\) and
\((Y^{\infty,m,\delta})_{m\in\mathbb N_0}\) be the Picard
iterations of \eqref{excursion-local-approximating-BSDE} and
\eqref{excursion-local-limiting-ODE}, respectively. They are initialized
by
\[
Y^{k,0,\delta}=Z^{k,0,\delta}=Y^{\infty,0,\delta}=0.
\]
For \(m\geq0\), their defining equations are
\begin{align}
Y_r^{k,m+1,\delta}
={}&
\zeta_{\bar t}^{k,\delta}
+\int_r^{\lambda^{k,\delta}(\tau^\delta(t+))}
f_q^{k,\delta}(Y_q^{k,m,\delta},Z_q^{k,m,\delta})\,dc_q^\delta
\notag\\
&+\int_r^{\lambda^{k,\delta}(\tau^\delta(t+))}
g_q^{k,\delta}(Y_q^{k,m,\delta})\,dW_q^{k,\delta}\notag\\
&-\int_r^{\lambda^{k,\delta}(\tau^\delta(t+))}
Z_q^{k,m+1,\delta}\,dB_q^\delta,
\label{excursion-local-approximating-Picard}
\end{align}
for \(r\in I_k\), and
\begin{equation}
Y_s^{\infty,m+1,\delta}
=
u(t+,x)
+\int_s^{\tau^\delta(t+)}
h^\delta(r,x,Y_r^{\infty,m,\delta})\,dW_r^\delta, \quad s\in I_{\bar t}.
\label{excursion-local-limiting-Picard}
\end{equation}
For every \(m\in\mathbb N_0\), invariance of \(p\)-variation under
increasing reparametrizations gives
\begin{align}
\left\|Y_{\lambda^{k,\delta}(\cdot)}^{k,\delta}
-Y_\cdot^{\infty,\delta}\right\|_{p;I_{\bar t}}
&\leq
\|Y^{k,\delta}-Y^{k,m,\delta}\|_{p;I_k}
\notag\\
&\quad+
\left\|Y_{\lambda^{k,\delta}(\cdot)}^{k,m,\delta}
-Y_\cdot^{\infty,m,\delta}\right\|_{p;I_{\bar t}}
+\|Y^{\infty,m,\delta}-Y^{\infty,\delta}\|_{p;I_{\bar t}}.
\label{excursion-Picard-three-term}
\end{align}

{We first establish uniform convergence of the Picard schemes.}
The third term in \eqref{excursion-Picard-three-term} vanishes as
\(m\to\infty\) by the usual Picard convergence for the limiting ODE.
For the first term, the terminal conditions are uniformly bounded by
Proposition~\ref{Prop-RFBSDE-wellposedness}, the drivers by
\eqref{property-lambda-n}, and the coefficients \(f^{k,\delta}\) by
Assumption~\hyperlink{A:xi-f}{\((\xi,f)\)}. For \(g^{k,\delta}\), the
extension estimate \cite[Lemma~2.16]{chevyrev_canonical_2019},
Proposition~\ref{Prop-RFBSDE-wellposedness}, and
Lemma~\ref{Lemma-hX-Assumption-Ad} give the required coefficient bounds
uniformly in \(k\) {under Assumption~\hyperlink{A:h}{\((h)\)}}.
{Under Assumption~\hyperlink{A:h'}{\((h')\)}, these bounds
follow directly from \((h')\), since \(h\) is independent of \(X\).}
Hence the argument of
\cite[(4.16)]{becherer_rough_2026} yields, for every \(\varepsilon>0\),
\begin{align}
\sup_{m\in\mathbb N_0}\sup_{k\in\mathbb N}\left(
\|Y^{k,m,\delta}\|_{p,2;I_k}
+\|Z^{k,m,\delta}\|_{\mathrm{BMO};I_k}
\right)&<\infty,
\label{excursion-uniform-Picard-bounds}
\\
\lim_{m\to\infty}\sup_{k\in\mathbb N}
\mathbb P\left(
\|Y^{k,m,\delta}-Y^{k,\delta}\|_{p;
I_k}>\varepsilon
\right)&=0.
\label{excursion-uniform-Picard-convergence}
\end{align}

{It remains to prove, for every fixed \(m\in\mathbb N_0\),}
\begin{equation}
\left\|Y_{\lambda^{k,\delta}(\cdot)}^{k,m,\delta}
-Y_\cdot^{\infty,m,\delta}\right\|_{p;I_{\bar t}}
\xrightarrow{k\to\infty}0
\qquad\text{in probability}.
\label{excursion-fixed-Picard-convergence}
\end{equation}
We prove this by induction on \(m\), simultaneously with
\begin{equation}
\mathbb E\left[
\int_{I_k}|Z_r^{k,m,\delta}|^2\,dc_r^\delta
\right]
\xrightarrow{k\to\infty}0.
\label{excursion-fixed-Picard-Z-convergence}
\end{equation}
Both induction assertions are immediate for \(m=0\). The terminal
convergence is precisely \eqref{excursion-terminal-L2-convergence} by
the definition of \(\zeta_{\bar t}^{k,\delta}\).
Suppose that \eqref{excursion-fixed-Picard-convergence} and
\eqref{excursion-fixed-Picard-Z-convergence} hold at level \(m\). For
\(s\in I_{\bar t}\), set
\begin{equation*}
F_s^{k,m}
:=
\int_{\lambda^{k,\delta}(s)}^{
\lambda^{k,\delta}(\tau^\delta(t+))}
f_r^{k,\delta}(Y_r^{k,m,\delta},Z_r^{k,m,\delta})\,dc_r^\delta,
\end{equation*}
\begin{equation*}
H_s^{k,m}
:=
\int_s^{\tau^\delta(t+)}
g_{\lambda^{k,\delta}(r)}^{k,\delta}
\bigl(Y_{\lambda^{k,\delta}(r)}^{k,m,\delta}\bigr)
\,d(W^{k,\delta}\circ\lambda^{k,\delta})_r,
\end{equation*}
\begin{equation*}
H_s^{\infty,m}
:=
\int_s^{\tau^\delta(t+)}
h^\delta(r,x,Y_r^{\infty,m,\delta})\,dW_r^\delta,
\end{equation*}
\begin{equation*}
M_s^{k,m+1}
:=
\int_{\lambda^{k,\delta}(s)}^{
\lambda^{k,\delta}(\tau^\delta(t+))}
Z_r^{k,m+1,\delta}\,dB_r^\delta.
\end{equation*}
Subtracting \eqref{excursion-local-limiting-Picard} from the time-changed
version of \eqref{excursion-local-approximating-Picard} gives
\begin{align}
&Y_{\lambda^{k,\delta}(s)}^{k,m+1,\delta}
-Y_s^{\infty,m+1,\delta}
\notag\\
&\quad=
\zeta_{\bar t}^{k,\delta}-u(t+,x)
+F_s^{k,m}
+H_s^{k,m}-H_s^{\infty,m}
-M_s^{k,m+1}.
\label{excursion-local-Picard-error-decomposition}
\end{align}
The terminal difference is constant in \(s\) and hence has zero
\(p\)-variation. It remains to control the finite-variation, Young, and
martingale terms and to prove
\eqref{excursion-fixed-Picard-Z-convergence} at level \(m+1\).

{To control the finite-variation term, put}
\[
\Delta c_k:=c^\delta(\lambda^{k,\delta}(\tau^\delta(t+)))
-c^\delta(\lambda^{k,\delta}(\bar t)),
\]
which converges to zero by \eqref{property-lambda-n}. The growth bound
on \(f\), the finite-variation estimate, and Cauchy--Schwarz give
\[
\|F^{k,m}\|_{p;I_{\bar t}}
\leq
C_f\Delta c_k\left(1+|\zeta_{\bar t}^{k,\delta}|
+\|Y^{k,m,\delta}\|_{p;I_k}
\right)
+
C_f\Delta c_k^{1/2}
\left(
\int_{I_k}|Z_r^{k,m,\delta}|^2\,dc_r^\delta
\right)^{1/2}.
\]
The bound \eqref{excursion-uniform-Picard-bounds} and the induction
hypothesis therefore give
\begin{equation}
\|F^{k,m}\|_{p;I_{\bar t}}
\xrightarrow{k\to\infty}0
\qquad\text{in probability}.
\label{excursion-local-F-convergence}
\end{equation}

{We next turn to the Young term and show}
\begin{equation}
\|H^{k,m}-H^{\infty,m}\|_{p;I_{\bar t}}
\xrightarrow{k\to\infty}0
\qquad\text{in probability}.
\label{excursion-local-H-convergence}
\end{equation}
For \(r\in I_{\bar t}\), we set
\[
G_r^{k,m}:=g_{\lambda^{k,\delta}(r)}^{k,\delta}
\bigl(Y_{\lambda^{k,\delta}(r)}^{k,m,\delta}\bigr), \quad
G_r^{\infty,m}:=h^\delta(r,x,Y_r^{\infty,m,\delta}).
\]
The Young integral estimate
\cite[Proposition~A.7]{becherer_rough_2026}, applied on
\(I_{\bar t}\), yields
\begin{align}
&\|H^{k,m}-H^{\infty,m}\|_{p;I_{\bar t}}
\notag\\
&\quad\leq
C\left(
\|G^{k,m}-G^{\infty,m}\|_{p;I_{\bar t}}
+|G_{\tau^\delta(t+)}^{k,m}-G_{\tau^\delta(t+)}^{\infty,m}|
\right)
\|W^{k,\delta}\circ\lambda^{k,\delta}\|_{q;I_{\bar t}}
\notag\\
&\qquad+
C\left(
\|G^{\infty,m}\|_{p;I_{\bar t}}
+|G_{\tau^\delta(t+)}^{\infty,m}|
\right)
\|W^{k,\delta}\circ\lambda^{k,\delta}-W^\delta
\|_{q;I_{\bar t}}.
\label{excursion-local-H-Young-estimate}
\end{align}
By \eqref{property-lambda-n}, we have
\begin{equation}
\|W^{k,\delta}\circ\lambda^{k,\delta}-W^\delta
\|_{q;I_{\bar t}}
\xrightarrow{k\to\infty}0,
\qquad
\sup_k
\|W^{k,\delta}\circ\lambda^{k,\delta}\|_{q;I_{\bar t}}
<\infty.
\label{excursion-local-driver-convergence}
\end{equation}
By construction of the time-stretched forward process, we have
\[
X_r^{k,\delta}
=X_{c^\delta(r)}^{c^\delta(\lambda^{k,\delta}(\bar t)),x},
\qquad r\geq\lambda^{k,\delta}(\bar t).
\]
Hence, invariance of \(p\)-variation under increasing
reparametrizations and Proposition~\ref{Prop-random-initial-SDE-estimates}
give
\begin{equation}
\mathbb E\!\left[\left\|X_{\lambda^{k,\delta}(\cdot)}^{k,\delta}-x
\right\|_{p;I_{\bar t}}^2\right]^{1/2}
\leq C(1+|x|)\Delta c_k^{1/2}\xrightarrow{k\to\infty}0.
\label{excursion-local-X-convergence}
\end{equation}
Together with Assumption~\hyperlink{A:hC}{\((h^C)\)},
\eqref{property-lambda-n}, and Lemma~\ref{Lemma-hX-Assumption-Bg}, this
implies
\begin{align}
&\sup_y\left\|g_{\lambda^{k,\delta}(\cdot)}^{k,\delta}(y)
-h^\delta(\cdot,x,y)\right\|_{p;I_{\bar t}}
+\sup_{r\in I_{\bar t}}\left|
\mathrm D_yg_{\lambda^{k,\delta}(r)}^{k,\delta}
-\mathrm D_yh^\delta(r,x,\cdot)
\right|_\infty
\notag\\
&\quad+
\sup_y\left|
g_{\lambda^{k,\delta}(\tau^\delta(t+))}^{k,\delta}(y)
-h^\delta(\tau^\delta(t+),x,y)
\right|
\xrightarrow{k\to\infty}0
\notag
\end{align}
in probability. More specifically, the time difference is controlled by
Assumption~\hyperlink{A:hC}{\((h^C)\)}, \eqref{property-lambda-n}, and
\cite[Lemma~4.9]{becherer_rough_2026}, while the state difference is
controlled by \eqref{excursion-local-X-convergence} and
Lemma~\ref{Lemma-hX-Assumption-Bg}.
Under Assumption~\hyperlink{A:h'}{\((h')\)}, \(h\) is independent
of \(x\), so the same argument applies without a state-difference term.
Combining this coefficient
convergence with the induction hypothesis
\eqref{excursion-fixed-Picard-convergence} and the terminal convergence
\eqref{excursion-terminal-L2-convergence}, we may apply
\cite[Lemma~2.3]{becherer_rough_2026} to obtain
\begin{equation}
\|G^{k,m}-G^{\infty,m}\|_{p;I_{\bar t}}
+|G_{\tau^\delta(t+)}^{k,m}-G_{\tau^\delta(t+)}^{\infty,m}|
\xrightarrow{k\to\infty}0
\label{excursion-local-G-convergence}
\end{equation}
in probability. For a detailed computation of this composition
estimate, we refer the reader to
\cite[(4.19)--(4.22)]{becherer_rough_2026}. Moreover,
\[
\|G^{\infty,m}\|_{p;I_{\bar t}}
+|G_{\tau^\delta(t+)}^{\infty,m}|<\infty.
\]
Combining
\eqref{excursion-local-G-convergence} and
\eqref{excursion-local-driver-convergence} in
\eqref{excursion-local-H-Young-estimate} proves
\eqref{excursion-local-H-convergence}.

{It remains to prove convergence of the martingale terms.} Set
\[
\Gamma_k^{m+1}
:=
\zeta_{\bar t}^{k,\delta}
+F_{\bar t}^{k,m}
+H_{\bar t}^{k,m}
=Y_{\lambda^{k,\delta}(\bar t)}^{k,m+1,\delta}
+\int_{I_k}Z_r^{k,m+1,\delta}\,dB_r^\delta,
\]
\[
\Gamma_\infty^{m+1}
:=
u(t+,x)+H_{\bar t}^{\infty,m}
=Y_{\bar t}^{\infty,m+1,\delta}.
\]
The limiting random variable \(\Gamma_\infty^{m+1}\) is deterministic.
By It\^o's isometry and the fact that
\(Y_{\lambda^{k,\delta}(\bar t)}^{k,m+1,\delta}
-Y_{\bar t}^{\infty,m+1,\delta}\) is
\(\mathcal F_{\lambda^{k,\delta}(\bar t)}^\delta\)-measurable, we get
\begin{align}
\mathbb E\left[
\int_{I_k}|Z_r^{k,m+1,\delta}|^2\,dc_r^\delta
\right]
&=
\mathbb E\left[
\left|
\int_{I_k}Z_r^{k,m+1,\delta}\,dB_r^\delta
\right|^2
\right]
\notag\\
&\leq
\mathbb E\left[
\left|Y_{\lambda^{k,\delta}(\bar t)}^{k,m+1,\delta}
-Y_{\bar t}^{\infty,m+1,\delta}\right|^2
+\left|\int_{I_k}Z_r^{k,m+1,\delta}\,dB_r^\delta\right|^2
\right]
\notag\\
&=
\mathbb E\left[
\left|\Gamma_k^{m+1}-\Gamma_\infty^{m+1}\right|^2
\right].
\label{excursion-local-Z-estimate}
\end{align}
The convergences \eqref{excursion-terminal-L2-convergence},
\eqref{excursion-local-F-convergence}, and
\eqref{excursion-local-H-convergence} imply
\[
\Gamma_k^{m+1}\xrightarrow{k\to\infty}\Gamma_\infty^{m+1}
\qquad\text{in probability}.
\]
The uniform terminal bound, \eqref{excursion-uniform-Picard-bounds}, and
\cite[Lemma~2.2(a)]{becherer_rough_2026} give
\[
\sup_{k\in\mathbb N}\sup_{m\in\mathbb N_0}
\|Y^{k,m,\delta}\|_{L^\infty(\Omega\times I_k)}<\infty.
\]
Together with the BMO bound
\eqref{excursion-uniform-Picard-bounds}, the BMO energy estimate and the
Burkholder--Davis--Gundy inequality therefore give, exactly as in the proof of
\cite[(4.24)]{becherer_rough_2026},
\[
\sup_k\mathbb E\left[|\Gamma_k^{m+1}|^4\right]<\infty.
\]
Thus \((|\Gamma_k^{m+1}|^2)_k\) is uniformly integrable, and Vitali's
convergence theorem gives
\[
\mathbb E\left[
|\Gamma_k^{m+1}-\Gamma_\infty^{m+1}|^2
\right]
\xrightarrow{k\to\infty}0.
\]
It then follows from \eqref{excursion-local-Z-estimate} that
\eqref{excursion-fixed-Picard-Z-convergence} holds at level \(m+1\).
Finally, the \(p\)-variation BDG inequality
\cite[Theorem~14.12]{friz_multidimensional_2010} yields
\[
\|M^{k,m+1}\|_{p;I_{\bar t}}
\xrightarrow{k\to\infty}0
\qquad\text{in probability}.
\]
Together with \eqref{excursion-local-Picard-error-decomposition},
\eqref{excursion-local-F-convergence}, and
\eqref{excursion-local-H-convergence}, this proves
\eqref{excursion-fixed-Picard-convergence} at level \(m+1\) and completes
the induction.

Returning to \eqref{excursion-Picard-three-term}, we first let
\(k\to\infty\) and use
\eqref{excursion-fixed-Picard-convergence}. We then let \(m\to\infty\)
and use \eqref{excursion-uniform-Picard-convergence}, together with the
convergence of the deterministic limiting Picard scheme. This yields
\[
\left\|
Y_{\lambda^{k,\delta}(\cdot)}^{k,\delta}
-Y_\cdot^{\infty,\delta}
\right\|_{p;I_{\bar t}}
\xrightarrow{k\to\infty}0
\qquad\text{in probability},
\]
which is the desired assertion.
\end{proof}

\subsection{Continuity of the solution map and backward flow property}
\label{continuity-solution-map}

We now extend the fixed-terminal stability established in
Theorem~\ref{theorem-Skorokhod} and
Corollary~\ref{Cor-Skorokhod-general-drivers} to simultaneous perturbations
of the rough driver and the terminal function in
Proposition~\ref{Prop-Continuity}. We then use the rough-FBSDE representation
to establish the backward flow property of the solution in
Corollary~\ref{Cor-backward-flow-property}.

\begin{proposition}
  \label{Prop-Continuity}{Let $q\in[1,2)$ and $p>2$
  satisfy $1/p+1/q>1$. We consider the FBSDE solution map for
  $W\in D^q([0,T];\mathbb R^e)$.}
  Suppose that either Assumptions 
$(\hyperlink{A:b-sigma}{b}, \hyperlink{A:b-sigma}{\sigma}, 
\hyperlink{A:xi-f}{\xi}, \hyperlink{A:xi-f}{f},
\hyperlink{A:h'}{h'}, \hyperlink{A:hC}{h^C})$ 
or Assumptions 
$(\hyperlink{A:b-sigma'}{b'}, \hyperlink{A:b-sigma'}{\sigma'}, 
\hyperlink{A:xi-f}{\xi}, \hyperlink{A:xi-f}{f},
\hyperlink{A:h}{h}, \hyperlink{A:hC}{h^C})$ hold. Then the map
  \begin{eqnarray*}
      \left( \mathfrak{D}^{q \text{-var}} ([0, T];\mathbb{R}^e), \alpha_q \right) \times
      \mathrm{BUC}(\mathbb{R}^d; \mathbb{R}^k) & \rightarrow & (\mathfrak{D} ([0, T] ; C (\mathbb{R}^d; \mathbb{R}^k)), \alpha_\infty),\\
      (\jmath W, \xi) & \mapsto & \mathbf{u},
    \end{eqnarray*}
  is continuous.
\end{proposition}

\begin{proof}
It is enough to prove sequential continuity. Let
\[
 \alpha_q(\jmath W^n,\jmath W)\longrightarrow0,
 \qquad
 \|\xi^n-\xi\|_\infty\longrightarrow0,
\]
and let $\mathbf u^n$ and $\mathbf u$ be the corresponding Marcus lifts.
Introduce the intermediate solution $\widehat{\mathbf u}^{\,n}$ driven by
$W^n$ but with terminal function $\xi$. The triangle inequality gives
\begin{equation}
 \alpha_\infty(\mathbf u^n,\mathbf u)
 \leq
 \alpha_\infty(\mathbf u^n,\widehat{\mathbf u}^{\,n})
 +\alpha_\infty(\widehat{\mathbf u}^{\,n},\mathbf u).
 \label{solution-map-triangle}
\end{equation}
By Corollary~\ref{Cor-Skorokhod-general-drivers}, applied with the fixed
terminal function $\xi$,
\begin{equation}
 \alpha_\infty(\widehat{\mathbf u}^{\,n},\mathbf u)
 \longrightarrow0.
 \label{fixed-terminal-driver-convergence}
\end{equation}

We claim that the first term in \eqref{solution-map-triangle} also converges
to zero. Fix $\delta>0$. By
Corollary~\ref{Cor-Skorokhod-general-drivers}, the $\delta$-extensions
$u^{n,\delta}$ and $\widehat u^{n,\delta}$ are continuous and satisfy the
rough-FBSDE representation \eqref{representation-delta-extension-u}. We show
that
\begin{equation}
 \sup_{(t,x)\in[0,T+\delta]\times\mathbb R^d}
 \left|u^{n,\delta}(t,x)-\widehat u^{n,\delta}(t,x)\right|
 \longrightarrow0.
 \label{uniform-terminal-stability-solution-map}
\end{equation}
Suppose otherwise. Then, along a subsequence, there are
$t_n\in[0,T+\delta]$ and $x_n\in\mathbb R^d$ for which the difference in
\eqref{uniform-terminal-stability-solution-map} is bounded away from zero.
Apply the terminal-stability assertion of
Proposition~\ref{prop:uniform-pairwise-stability} to the two time-extended
rough FBSDEs starting from $(t_n,x_n)$, driven by $W^{n,\delta}$, and with
terminal functions $\xi^n$ and $\xi$. Its assumptions follow from
$\|\xi^n-\xi\|_\infty\to0$ and the uniform bounds on the time-extended
drivers. We obtain
\[
 \left|u^{n,\delta}(t_n,x_n)
 -\widehat u^{n,\delta}(t_n,x_n)\right|
 \longrightarrow0
 \qquad\text{in probability}.
\]
Both quantities are deterministic by
Proposition~\ref{Prop-RFBSDE-Markov-flow}, so the convergence is
deterministic. This contradiction proves
\eqref{uniform-terminal-stability-solution-map}.

Taking the identity
reparametrization in Lemma~\ref{alternative-a-metric} and using
\eqref{uniform-terminal-stability-solution-map} gives
\[
 \alpha_\infty(\mathbf u^n,\widehat{\mathbf u}^{\,n})
 \longrightarrow0.
\]
Combining this with \eqref{solution-map-triangle} and
\eqref{fixed-terminal-driver-convergence} proves the assertion.
\end{proof}

\begin{corollary}[Backward flow property]
\label{Cor-backward-flow-property}
{Under the coefficient assumptions of Theorem~\ref{main-result},
let $W\in D^q([0,T];\mathbb R^e)$ with $q\in[1,2)$ and
$p>2$ satisfying $1/p+1/q>1$, and define $u$ by the FBSDE
representation. Then, for every}
$0\leqslant r\leqslant s\leqslant t\leqslant T$, $x\in\mathbb R^d$, and
$\xi\in\mathrm{BUC}(\mathbb R^d;\mathbb R^k)$, it holds
\[
  u(r;t,\xi)(x)
  =
  u\bigl(r;s,u(s;t,\xi)\bigr)(x).
\]
\end{corollary}

\begin{proof}
By the definition of $u$ and
Proposition~\ref{Prop-RFBSDE-Markov-flow}(iii),
\[
  u(r;t,\xi)(x)
  =Y_r^{r,x;t,\xi}
  =Y_r^{r,x;s,Y_s^{s,\cdot;t,\xi}}
  =Y_r^{r,x;s,u(s;t,\xi)}
  =u\bigl(r;s,u(s;t,\xi)\bigr)(x).
\]
\end{proof}

\begin{remark}
One could alternatively try to establish the flow property first for the
approximating solutions $u^n$ and then pass to the limit. In the presence of
jumps, however, the reparametrizations $\lambda^{n,\delta}$ used in the proof
of Theorem~\ref{theorem-Skorokhod} are not explicit, which makes this passage
to the limit less direct. The rough FBSDE representation yields the flow
property immediately from uniqueness and thereby avoids this difficulty.
\end{remark}

\section{Solutions to SPDEs as Markov Processes} \label{section-markov-process}

{
Throughout this section, we equip
$E=\mathrm{BUC}(\mathbb R^d;\mathbb R^k)$ with the Borel $\sigma$-algebra $\mathcal E$ generated by the compact-open
topology on $C(\mathbb R^d;\mathbb R^k)$.
By \cite[Lemma~A5.1]{kallenberg_foundations_2021} and continuity of
the functions in $E$, we have
\[
 \mathcal E=\sigma\{\eta\mapsto\eta(x):x\in\mathbb Q^d\}.
\]
All measurability
statements for $E$-valued maps in this section refer to $\mathcal E$ and
we write $B_b(E,\mathcal E)$ for the bounded $\mathcal E$-measurable
real-valued functions. We choose this topology, since the measurability of 
the solution map $u$ is easier to establish than for the uniform topology on $E$.
}

In this section, we introduce a completed probability space
$(\Omega, \mathcal{F}, \mathbb{P})$ supporting an \linebreak $e$-dimensional stochastic
process $L$ of finite $q$-variation, with $q < 2$. We denote by
$(\mathcal{F}^L_{t, T})_{t \in [0, T]}$ the backward filtration
generated by $L$, where the $\sigma$-algebra is defined by \linebreak $\mathcal{F}^L_{s, t}
\assign \sigma \left( L_u - L_v, \, s \leqslant u \leqslant v \leqslant t
\right)$ for all $0 \leqslant s \leqslant t \leqslant T$. Let $D ([s, t] ;
\mathbb{R}^e)$ denote the space of c{\`a}gl{\`a}d paths, and let
$\mathcal{D}_{s, t}$ be the $\sigma$-algebra generated by the coordinate
projections $\left\{ \pi_{t_1, \ldots, t_k} \, \mid \, t_1, \ldots, t_k \in
{[s, t]}, k \in \mathbb{N} \right\}$. In \cite[Theorem~11.5.2]{whitt_stochastic-process_2002} it is shown that the Borel-$\sigma$-algebra
generated by the classical Skorokhod $J1$ or $M1$ topology coincides with
$\mathcal{D}_{s, t}$. 

Stochastic processes with finite $q$-variation for $q<2$ form a rich class. Typical examples include fractional Brownian motion with Hurst parameter $H>1/2$ in the continuous case, as well as a large class of L\'evy processes. Besides L\'evy processes of bounded variation such as compound Poisson or Gamma processes, a sufficient condition for a L\'evy process without Gaussian part, with characteristic triplet $(0,\gamma,\nu)$, to have finite $q$-variation is stated in \cite{monroe_gamma-variation_1972}. We briefly recall it here.

\begin{theorem}\cite[Theorem 2]{monroe_gamma-variation_1972}
Given a L\'evy process $L$ with characteristic triplet $(0,\gamma,\nu)$, we define the so-called Blumenthal--Getoor index $\beta$ as
\begin{align*}
\beta := \inf \left\{ \alpha>0 \,\middle|\, \int_{|x| < 1} |x|^{\alpha}\,\nu(dx) < \infty \right\}.
\end{align*}
Then $L$ has finite $q$-variation for every $q>\beta$.
\end{theorem}

The above theorem implies, in particular, that for $\alpha<2$, $\alpha$-stable L\'evy processes (see \cite[Chapter 3.7]{cont_financial_2004}) have Blumenthal--Getoor index $\beta=\alpha$ and therefore finite $q$-variation for every $q>\alpha$. More generally, (generalized) tempered stable L\'evy processes, which often appear in financial modeling and in physics (see \cite[Chapter 3.7 and references therein]{cont_financial_2004}), also belong to this class. Their L\'evy measure admits a density of the form
\[
\nu(x) = \frac{c_+}{x^{1+\alpha_+}} e^{-\lambda_+ x} \mathbf{1}_{\{x>0\}} 
+ \frac{c_-}{|x|^{1+\alpha_-}} e^{-\lambda_- |x|} \mathbf{1}_{\{x<0\}},
\]
for some $\alpha_\pm \in (0,2)$, $c_\pm >0$ and $\lambda_\pm >0$. In particular, the Blumenthal--Getoor index is $\beta= \max \{\alpha_+, \alpha_- \} <2$. We also note that any independent sum of the above-mentioned L\'evy processes or fractional Brownian motions again has finite $q$-variation for $q<2$.

In Section~\ref{Section-measurability}, we show how one can rigorously randomize the RPDE solution to obtain an adapted stochastic process, which we interpret as the solution to the associated SPDE. In Section~\ref{Section-SPDE-Markov}, we then prove that, in the special case where $L$ is a L\'evy process of finite $q$-variation, the SPDE solution is a Markov process on $\mathrm{BUC}(\mathbb{R}^d; \mathbb{R}^k)$.

\subsection{A Notion of SPDE Solution}\label{Section-measurability}

Let $\eta : \Omega \rightarrow \mathrm{BUC}(\mathbb{R}^d; \mathbb{R}^k)$ be
$\mathcal{F}^L_{t, T}$-measurable. We consider a semilinear backward SPDE of the form
\begin{equation}
  \frac{\partial u}{\partial s} + \frac{1}{2} \mathrm{Tr} (\sigma \sigma^T
  \nabla_x^2 u) + b (s, x) \cdot \nabla_x u + f (s, x, u, \sigma^T \nabla_x u)
  + h (s, x, u) \diamond \tmop{dL} = 0, \label{SPDE}
\end{equation}
with terminal condition $u (t) = \eta$. We interpret \eqref{SPDE} pathwise as an RPDE. More precisely, for fixed $\omega\in\Omega$, we consider

\begin{align*}
  \frac{\partial u^{L (\omega)}}{\partial s} & + \frac{1}{2} \mathrm{Tr}
  (\sigma \sigma^T \nabla_x^2 u^{L (\omega)}) + b (s, x) \cdot \nabla_x u^{L
  (\omega)} \nonumber\\
  & + f (s, x, u^{L (\omega)}, \sigma^T \nabla_x u^{L (\omega)}) + h (s, x,
  u^{L (\omega)}) \diamond \tmop{dL} (\omega) = 0,
\end{align*}
with terminal condition $u^{L(\omega)}(t)=\eta(\omega)\in \mathrm{BUC}(\mathbb{R}^d;\mathbb{R}^k)$. We write $u^L(s; t,\eta)$ for the robust viscosity solution to emphasize its dependence on the terminal condition. For almost every $\omega$, there exists a solution $u^{L (\omega)} (s ; t, \eta
(\omega)) (x)$ (or simply $u^{L (\omega)}$ when there is no need to emphasize
the dependency on the terminal condition) to the above RPDE by Theorem
\ref{main-result}.

We then define
\[ u^L (s ; t, \eta) : \, \omega \mapsto u^{L (\omega)} (s ; {t}, \eta (\omega))
\]
and refer to $u^L$ as the solution to the semilinear SPDE \eqref{SPDE}. 

{
To justify this construction as an $E$-valued stochastic process, we
establish joint measurability of the solution in the driver and terminal
function. We use the rough-FBSDE representation and the
measurable selection result in \cite[Theorem~5.4]{becherer_rough_2026}.
}

\begin{proposition}
  \label{Prop-randomise}Suppose that either Assumptions 
$(\hyperlink{A:W}{W}, \hyperlink{A:b-sigma}{b}, \hyperlink{A:b-sigma}{\sigma},
\hyperlink{A:xi-f}{\xi}, \hyperlink{A:xi-f}{f},
\hyperlink{A:h'}{h'}, \hyperlink{A:hC}{h^C})$ 
or Assumptions \linebreak
$(\hyperlink{A:W}{W}, \hyperlink{A:b-sigma'}{b'}, \hyperlink{A:b-sigma'}{\sigma'},
\hyperlink{A:xi-f}{\xi}, \hyperlink{A:xi-f}{f},
\hyperlink{A:h}{h}, \hyperlink{A:hC}{h^C})$ hold{.
  For every $0\le s\le t\le T$, the map
  \[
   (W,\zeta)\longmapsto u^W(s;t,\zeta)
  \]
  from $(D^q([s,t];\mathbb R^e)\times E,
  \mathcal D_{s,t}\otimes\mathcal E)$ to $(E,\mathcal E)$ is
  measurable, where $\mathcal D_{s,t}$ is restricted to $D^q$ and
  $u^W$ is defined by the rough-FBSDE representation in
  Proposition~\ref{Prop-RFBSDE-Markov-flow}, with driver $W$ and
  terminal function $\zeta$ at time $t$.

  Consequently, let $t\in[0,T]$ and let $\eta:\Omega\to E$ be
  $\mathcal F^L_{t,T}$-measurable. Then the process
  $(u^L(s;t,\eta))_{s\in[0,t]}$ is adapted to the backward filtration
  $(\mathcal F^L_{s,T})_{s\in[0,t]}$. If $\eta\in E$ is deterministic,
  this process is adapted to $(\mathcal F^L_{s,t})_{s\in[0,t]}$.}
\end{proposition}

\begin{proof}
{
Fix $0\le s<t\le T$. Since $\mathcal E$ is generated by spatial
evaluations, it suffices to show that
$(W,\zeta)\mapsto u^W(s;t,\zeta)(x)$ is jointly measurable for each
fixed $x\in\mathbb Q^d$. We prove this using the rough-FBSDE representation in 
Proposition~\ref{Prop-RFBSDE-Markov-flow} and a measurable selection argument 
for this rough-FBSDE in
\cite[Theorem~5.4]{becherer_rough_2026}, where we treat $(W,\zeta)$ as a parameter in
$D^q([s,t];\mathbb R^e)\times E$ equipped with
$\mathcal D_{s,t}\otimes\mathcal E$.

We now verify the assumptions of that theorem. The forward process does
not depend on $(W,\zeta)$. Since the evaluation
map $(\zeta,y)\mapsto\zeta(y)$ is jointly continuous for the
compact-open topology, the terminal value $\zeta(X_t^{s,x}(\omega^B))$ is jointly
measurable in $(\omega^B,W,\zeta)$, where $\omega^B \in \Omega^B$ is from the
underlying probability space in the FBSDE construction. The generators
$f(r,X_r^{s,x},\cdot,\cdot)$ and $h(r,X_r^{s,x},\cdot)$ have the
required measurability, and the remaining assumptions were verified
in the proof of Proposition~\ref{Prop-RFBSDE-wellposedness}.
The theorem therefore gives a version
$\widetilde Y^{W,\zeta}$ jointly measurable in $(\omega^B,W,\zeta)$.

With the Brownian filtration started at $s$, the initial value is
almost surely equal to the deterministic quantity $u^W(s;t,\zeta)(x)$
by Proposition~\ref{Prop-RFBSDE-Markov-flow}. Hence
\[
 u^W(s;t,\zeta)(x)
 =\mathbb E_B[\widetilde Y_s^{W,\zeta}]
\]
is jointly measurable in $(W,\zeta)$, where $\mathbb E_B$ denotes
expectation over the Brownian space of the FBSDE. This proves the
first assertion for $s<t$; for $s=t$, the map is simply
$(W,\zeta)\mapsto\zeta$.

To obtain adaptedness, we now substitute the random driver and terminal
function. Since the equation depends only on driver increments, its
solution is unchanged when we subtract $L_s$ from the driver. The
resulting path $(L_r-L_s)_{r\in[s,t]}$ is
$\mathcal F^L_{s,t}$-measurable. Together with the
$\mathcal F^L_{t,T}$-measurable terminal function $\eta$, this gives an
$\mathcal F^L_{s,T}$-measurable pair of inputs. Joint measurability of
the solution map therefore shows that
$\omega\mapsto u^L(s;t,\eta)(\omega)
=u^{L(\omega)}(s;t,\eta(\omega))$ is
$\mathcal F^L_{s,T}$-measurable. If $\eta$ is deterministic, the pair
of inputs, and hence the solution at $s$, are already
$\mathcal F^L_{s,t}$-measurable.
}
\end{proof}

\subsection{SPDE Solution as Markov Process} \label{Section-SPDE-Markov}
In this section, we restrict our attention to the case where $L$ is a L\'evy process of finite $q$-variation for $q<2$, since we will rely on the fact that $L$ has independent increments. We also define the time-reversed process $\tilde{L}_s \assign L_T - L_{T - s}$. One can easily verify that $(\tilde{L}_s)_{s \in [0, T]}$ is a L\'evy process adapted to the filtration $(\mathcal{F}^{\tilde{L}}_{0, s})_{s \in [0,T]}$, where $\mathcal{F}^{\tilde{L}}_{s, t} \assign \sigma \big( \tilde{L}_u - \tilde{L}_v, \, s \leqslant u \leqslant v \leqslant t \big)$, $0 \leqslant s \leqslant t \leqslant T$, denotes the $\sigma$-algebra generated by the increments of $\tilde{L}$ on $[s,t]$. By construction, $\mathcal{F}^{\tilde{L}}_{s, t} = \mathcal{F}^L_{T - t, T - s}$. 

Next, define the forward-in-time process $v^L (t ; s, \eta) \assign u^L (T - t
; T - s, \eta)$. Then by Proposition~\ref{Prop-randomise}, if
$\eta$ is
$\mathcal{F}^{\tilde{L}}_{0,s}$-measurable, the process
\(
\big(v^{L}(t;s,\eta)\big)_{t\in[s,T]}
\)
is adapted to the (forward) filtration
$
(\mathcal{F}^{\tilde{L}}_{0,t})_{t\in[s,T]}$. If, moreover, $\eta$ is deterministic, then
\(
\big(v^{L}(t;s,\eta)\big)_{t\in[s,T]}
\)
is adapted to the filtration
$
(\mathcal{F}^{\tilde{L}}_{s,t})_{t\in[s,T]}.
$
In this section we prove that the process $v^L (s ; 0, {\xi})$, $0 \leqslant s \leqslant T$,
is a Markov process on ${(E,\mathcal E)}$ with respect to the
filtration $(\mathcal{F}^{\tilde{L}}_{0, s})_{s \in [0, T]}$. For $\varphi \in
{B_b(E,\mathcal E)}$, $\xi \in \mathrm{BUC}(\mathbb{R}^d; \mathbb{R}^k)$
and $0 \leqslant s \leqslant t \leqslant T$ we define the transition operators
\[ P_{s, t} \varphi (\xi) \assign \mathbb{E} [\varphi (v^L (t ; s, \xi))]. \]
To establish the Markov property, it suffices to show that
\[ \mathbb{E} [\varphi (v^L (t ; 0, \xi))  | \nobracket
   \mathcal{F}^{\tilde{L}}_{0, s}] = P_{s, t} \varphi (v^L (s ; 0, \xi)) \]
for all $\varphi \in {B_b(E,\mathcal E)}$ and $0 \leqslant
s \leqslant t \leqslant T$.

\begin{theorem}\label{Theorem-Markov}
  Suppose that either Assumptions 
$(\hyperlink{A:W}{W}, \hyperlink{A:b-sigma}{b}, \hyperlink{A:b-sigma}{\sigma}, 
\hyperlink{A:xi-f}{\xi}, \hyperlink{A:xi-f}{f},
\hyperlink{A:h'}{h'}, \hyperlink{A:hC}{h^C})$ 
or Assumptions \linebreak
$(\hyperlink{A:W}{W}, \hyperlink{A:b-sigma'}{b'}, \hyperlink{A:b-sigma'}{\sigma'}, 
\hyperlink{A:xi-f}{\xi}, \hyperlink{A:xi-f}{f},
\hyperlink{A:h}{h}, \hyperlink{A:hC}{h^C})$ hold. Then $v^L (s ; 0, \xi)$, $0 \leqslant s
  \leqslant T$, is a Markov process on ${(E,\mathcal E)}$ with
  respect to the filtration $(\mathcal{F}^{\tilde{L}}_{0, s})_{s \in [0, T]}$
  with transition operator $P_{s, t}$, $0 \leqslant s \leqslant t \leqslant
  T$.
\end{theorem}

\begin{remark}
  By an analogous proof it is also possible to see $u^L (s ; T, \xi)$, $0
  \leqslant s \leqslant T$, as a backward Markov process (see \cite[Chapter 1.7]{kunita_stochastic_2019}) on
  ${(E,\mathcal E)}$ with respect to the backward filtration
  $(\mathcal{F}^L_{s, T})_{s \in [0, T]}$ with backward transition operator
  $Q_{s, t}$, $0 \leqslant s \leqslant t \leqslant T$ now defined as
  \[ \mathbb{E} [\varphi (u^L (s ; T, \xi))  | \nobracket \mathcal{F}^L_{t,
     T}] = Q_{s, t} \varphi (u^L (t ; T, \xi)), \quad \forall \varphi \in {B_b(E,\mathcal E)}, \, 0 \leqslant s \leqslant t \leqslant
     T. \]
\end{remark}

\begin{proof}
{
Fix $0\le s<t\le T$ and $\varphi\in B_b(E,\mathcal E)$.
By the backward flow property in
Corollary~\ref{Cor-backward-flow-property}, we have pathwise
\[
 v^L(t;0,\xi)=v^L\bigl(t;s,v^L(s;0,\xi)\bigr).
\]
Since $v^L(s;0,\xi)$ is $\mathcal F^{\tilde L}_{0,s}$-measurable by
Proposition~\ref{Prop-randomise}, it suffices to show that
\begin{equation}\label{step-markov-property}
 \mathbb E[\varphi(v^L(t;s,\eta))\mid\mathcal F^{\tilde L}_{0,s}]
 =P_{s,t}\varphi(\eta)
\end{equation}
for every $\mathcal F^{\tilde L}_{0,s}$-measurable $E$-valued random
variable $\eta$.

By Proposition~\ref{Prop-randomise} and the fact that the equation depends
only on driver increments, the map
\[
 (\omega,\zeta)\longmapsto\varphi\bigl(v^L(t;s,\zeta)(\omega)\bigr)
\]
is $\mathcal F^L_{T-t,T-s}\otimes\mathcal E$-measurable. Taking
expectation in $\omega$ therefore shows that
$P_{s,t}\varphi\in B_b(E,\mathcal E)$.

We apply the conditioning formula in
\cite[Theorem~8.5(ii)]{kallenberg_foundations_2021}.
In our setting, the kernel is the conditional law of the driver segment
$(L_r-L_{T-t})_{r\in[T-t,T-s]}$ given
$\mathcal F^{\tilde L}_{0,s}=\mathcal F^L_{T-s,T}$, which equals its
unconditional law by independence of increments. Together with the
joint measurability established above, this gives
\eqref{step-markov-property}. Taking $\eta=v^L(s;0,\xi)$ proves
the Markov property. The case $s=t$ is immediate, since
$v^L(s;s,\zeta)=\zeta$.
}
\end{proof}

\appendix

\section{Stability of the time-extended forward SDE}
\label{appendix-forward-SDE}

This appendix provides the moment bounds and stability estimates for the
forward SDE needed in Section~\ref{section-proof}; see
Propositions~\ref{Prop-random-initial-SDE-estimates} and
\ref{Prop-random-initial-SDE-stability}. We also derive bounds for the random
functions obtained by composing $h$ with the forward process $X$ in
Lemmas~\ref{Lemma-hX-Assumption-Ad} and
\ref{Lemma-hX-Assumption-Bg}.

Let $S$, $c$, and $M$ be as in Subsection~\ref{Rough-FSBDEs}. For
$t\in[0,S]$ and
$\zeta\in L^2(\Omega,\mathcal F_t^M;\mathbb R^d)$, consider the forward SDE
\begin{equation}
 X_s^{t,\zeta}
 =\zeta+\int_t^s b(r,X_r^{t,\zeta})\,dc_r
 +\int_t^s\sigma(r,X_r^{t,\zeta})\,dM_r,
 \qquad s\in[t,S].
 \label{extended-forward-SDE}
\end{equation}

\begin{proposition}\label{Prop-random-initial-SDE-estimates}
Suppose that Assumptions
$(\hyperlink{A:b-sigma}{b},\hyperlink{A:b-sigma}{\sigma})$ hold on $[0,S]$.
Then, for every $t\in[0,S]$,
$\zeta,\zeta'\in L^2(\Omega,\mathcal F_t^M;\mathbb R^d)$, and
$s\in[t,S]$, the corresponding solutions of \eqref{extended-forward-SDE}
satisfy
\begin{align}
  \mathbb{E}_t\left[
    \sup_{r\in[t,s]}|X^{t,\zeta}_r|^2
  \right]
  &\leqslant C_{S,c_S,b,\sigma}(1+|\zeta|^2),
  \label{random-X-moment}\\
  \mathbb{E}_t\left[
    \sup_{r\in[t,s]}|X^{t,\zeta}_r-\zeta|^2
  \right]
  &\leqslant
  C_{S,c_S,b,\sigma}(1+|\zeta|^2)(c_s-c_t),
  \label{random-X-increment}\\
  \mathbb{E}_t\left[
    \sup_{r\in[t,s]}
    |X^{t,\zeta}_r-X^{t,\zeta'}_r|^2
  \right]
  &\leqslant C_{S,c_S,b,\sigma}|\zeta-\zeta'|^2.
  \label{random-X-stability-sup}
\end{align}
For $p>2$, one also has
\begin{equation}
  \mathbb{E}_t\left[
    \|X^{t,\zeta}\|_{p;[t,s]}^2
  \right]^{1/2}
  \leqslant
  C_{S,c_S,p,b,\sigma}(1+|\zeta|)
  \bigl((c_s-c_t)+(c_s-c_t)^{1/2}\bigr).
  \label{random-X-p-variation}
\end{equation}
\end{proposition}

\begin{proof}
The Lipschitz and local boundedness assumptions imply the linear-growth
bound
\[
  |b(r,x)|+|\sigma(r,x)|
  \leqslant C_{b,\sigma}(1+|x|).
\]
By the conditional Burkholder--Davis--Gundy inequality, for $u\in[t,s]$,
\begin{align*}
  \mathbb{E}_t\left[
    \sup_{r\in[t,u]}|X^{t,\zeta}_r|^2
  \right]
  \leqslant{}&
  3|\zeta|^2
  +3(c_u-c_t)\int_t^u
    \mathbb{E}_t[|b(r,X^{t,\zeta}_r)|^2]\,dc_r\\
  &+C\int_t^u
    \mathbb{E}_t[|\sigma(r,X^{t,\zeta}_r)|^2]\,dc_r\\
  \leqslant{}&
  C(1+|\zeta|^2)
  +C\int_t^u
    \mathbb{E}_t\left[
      \sup_{v\in[t,r]}|X^{t,\zeta}_v|^2
    \right]\,dc_r.
\end{align*}
Applying the Stieltjes version of Gronwall's inequality pointwise almost
surely proves
\eqref{random-X-moment}. Applying
the same argument to
\[
  X^{t,\zeta}_r-\zeta
  =
  \int_t^r b(v,X^{t,\zeta}_v)\,dc_v
  +\int_t^r\sigma(v,X^{t,\zeta}_v)\,dM_v
\]
and using \eqref{random-X-moment} gives
\begin{align*}
  \mathbb{E}_t\left[
    \sup_{r\in[t,s]}|X^{t,\zeta}_r-\zeta|^2
  \right]
  &\leqslant
  C(c_s-c_t)\int_t^s
    \mathbb{E}_t[|b(r,X^{t,\zeta}_r)|^2]\,dc_r +C\int_t^s
    \mathbb{E}_t[|\sigma(r,X^{t,\zeta}_r)|^2]\,dc_r\\
  &\leqslant
  C_{S,c_S,b,\sigma}(1+|\zeta|^2)(c_s-c_t),
\end{align*}
which is \eqref{random-X-increment}. Finally, Lipschitz continuity yields
\begin{align*}
  \mathbb{E}_t\left[
    \sup_{r\in[t,u]}
    |X^{t,\zeta}_r-X^{t,\zeta'}_r|^2
  \right]
  \leqslant
  C|\zeta-\zeta'|^2
  +C_{b,\sigma}\int_t^u
    \mathbb{E}_t\left[
      \sup_{v\in[t,r]}
      |X^{t,\zeta}_v-X^{t,\zeta'}_v|^2
    \right]\,dc_r.
\end{align*}
Another pointwise application of Gronwall's inequality proves
\eqref{random-X-stability-sup}.

Finally,
\begin{align*}
  \|X^{t,\zeta}\|_{p;[t,s]}
  \leqslant{}&
  \left\|\int_t^\cdot b(r,X^{t,\zeta}_r)\,dc_r
  \right\|_{p;[t,s]}
  +
  \left\|\int_t^\cdot\sigma(r,X^{t,\zeta}_r)\,dM_r
  \right\|_{p;[t,s]}.
\end{align*}
The finite-variation bound and the $p$-variation version of the
Burkholder--Davis--Gundy inequality
\cite[Theorem~14.12]{friz_multidimensional_2010}, together with \eqref{random-X-moment}, imply
\begin{align*}
  \mathbb{E}_t\left[
    \|X^{t,\zeta}\|_{p;[t,s]}^2
  \right]^{1/2}
  &\leqslant
  \mathbb{E}_t\left[
    \left(\int_t^s|b(r,X^{t,\zeta}_r)|\,dc_r\right)^2
  \right]^{1/2}
  +C_{S,p}
  \mathbb{E}_t\left[
    \int_t^s|\sigma(r,X^{t,\zeta}_r)|^2\,dc_r
  \right]^{1/2} \\
  &\leqslant 
  C_{S,c_S,p,b,\sigma}(1+|\zeta|)
  \bigl((c_s-c_t)+(c_s-c_t)^{1/2}\bigr),
\end{align*}
which proves \eqref{random-X-p-variation}.
\end{proof}

\begin{proposition}\label{Prop-random-initial-SDE-stability}
Suppose that Assumptions
$(\hyperlink{A:b-sigma}{b},\hyperlink{A:b-sigma}{\sigma})$ hold on $[0,S]$.
Then, for every $t\in[0,S]$,
$\zeta,\zeta'\in L^2(\Omega,\mathcal F_t^M;\mathbb R^d)$, the corresponding
solutions of \eqref{extended-forward-SDE} satisfy, for every $p>2$,
\begin{align}
  \mathbb{E}\left[
    \|X^{t,\zeta}-X^{t,\zeta'}\|_{p;[t,S]}^2
  \right]^{1/2}
  &+
  \mathbb{E}\left[
    \|X^{t,\zeta}-X^{t,\zeta'}\|_{\infty;[t,S]}^2
  \right]^{1/2}
  \leq
  C_{S,c_S,p,b,\sigma}\|\zeta-\zeta'\|_{L^2}.
  \label{random-initial-X-quantitative}
\end{align}
The constant is independent of $t$. In particular, it is uniform over
families of clocks for which $c_S$ is uniformly bounded.
\end{proposition}

\begin{proof}
Set $D\assign X^{t,\zeta}-X^{t,\zeta'}$. For $s\in[t,S]$,
\[
  D_s
  =
  \zeta-\zeta'
  +\int_t^s
    \bigl(b(r,X^{t,\zeta}_r)-b(r,X^{t,\zeta'}_r)\bigr)\,dc_r
  +\int_t^s
    \bigl(\sigma(r,X^{t,\zeta}_r)
          -\sigma(r,X^{t,\zeta'}_r)\bigr)\,dM_r.
\]
For every $u\in[t,S]$, the finite-variation estimate and the
Lipschitz continuity of $b$ give
\begin{align*}
  \mathbb{E}\left[
    \left\|
      \int_t^\cdot
        \bigl(b(r,X^{t,\zeta}_r)-b(r,X^{t,\zeta'}_r)\bigr)\,dc_r
    \right\|_{p;[t,u]}^2
  \right]^{1/2}\leqslant
  C_{b,\sigma,S,p}(c_u-c_t)
  \mathbb{E}\left[
    \|D\|_{\infty;[t,u]}^2
  \right]^{1/2}.
\end{align*}
Similarly, the Burkholder--Davis--Gundy inequality and the Lipschitz
continuity of $\sigma$ yield
\begin{align*}
  \mathbb{E}\left[
    \left\|
      \int_t^\cdot
        \bigl(\sigma(r,X^{t,\zeta}_r)
              -\sigma(r,X^{t,\zeta'}_r)\bigr)\,dM_r
    \right\|_{p;[t,u]}^2
  \right]^{1/2} \leqslant
  C_{b,\sigma,S,p}(c_u-c_t)^{1/2}
  \mathbb{E}\left[
    \|D\|_{\infty;[t,u]}^2
  \right]^{1/2}.
\end{align*}
Consequently,
\begin{align}
  \mathbb{E}\left[
    \|D\|_{p;[t,u]}^2
  \right]^{1/2}
  \leqslant{}&
  C_{b,\sigma,S,p}
  \bigl((c_u-c_t)+(c_u-c_t)^{1/2}\bigr)
  \mathbb{E}\left[
    \|D\|_{\infty;[t,u]}^2
  \right]^{1/2}.
  \label{random-initial-X-basic-estimate}
\end{align}
Moreover, $\|D\|_{\infty;[t,u]}
  \leqslant
  |\zeta-\zeta'|+\|D\|_{p;[t,u]}$.
Choose $\delta>0$ such that
$C_{b,\sigma,S,p}(\delta+\delta^{1/2})$ is sufficiently small. By
continuity and monotonicity of $c$, the interval $[t,S]$ admits a partition
$t=t_0<\cdots<t_N=S$ such that
$c_{t_{j+1}}-c_{t_j}\leq\delta$, with $N$ bounded only in terms of
$c_S$ and $\delta$. Applying \eqref{random-initial-X-basic-estimate} on
each subinterval, absorbing the local variation term, and iterating yields
\[
  \mathbb E\left[
    \|D\|_{p;[t,S]}^2
  \right]^{1/2}
  \leq
  C_{S,c_S,p,b,\sigma}\|\zeta-\zeta'\|_{L^2}.
\]
Finally,
\[
  \|D\|_{\infty;[t,S]}
  \leq
  |\zeta-\zeta'|+\|D\|_{p;[t,S]},
\]
which proves \eqref{random-initial-X-quantitative}.
\end{proof}

In the rough FBSDE, the backward component is viewed as a rough BSDE
whose random Young vector field is obtained by composing \(h\) with the
forward process. The following two lemmas verify the corresponding
well-posedness and stability assumptions in
\cite{becherer_rough_2026}. More precisely,
Lemma~\ref{Lemma-hX-Assumption-Ad} verifies conditions~A.d and B.f,
while Lemma~\ref{Lemma-hX-Assumption-Bg} verifies condition~B.g.

\begin{lemma}
\label{Lemma-hX-Assumption-Ad}
Suppose that Assumption~\hyperlink{A:h}{\((h)\)} holds on $[0,S]$, let
$a\in[0,S]$, and let
\(X:\Omega\times[a,S]\to\mathbb R^d\) be an adapted continuous process
such that $\|X\|_{p,2;[a,S]}<\infty$. Define
$h_t^X(y)\assign h(t,X_t,y)$. Then
\begin{equation}
\label{hX-Assumption-Ad}
\sup_{t\in[a,S]}
\bigl\||h_t^X|_{C_b^2}\bigr\|_{L^\infty}
+
[[h^X]]_{p,2;[a,S]}
+
[[\mathrm D_yh^X]]_{p,2;[a,S]}
\leq
C_h\bigl(1+\|X\|_{p,2;[a,S]}\bigr)
<\infty,
\end{equation}
where, for finite-dimensional Banach spaces \(W,V\) and a progressively
measurable random map $g:\Omega\times[a,S]\longrightarrow C_b(W,V),$
the quantity $[[g]]_{p,2;[a,S]}$ is defined by
\[
[[g]]_{p,2;[a,S]}
\assign
\sup_{s\in[a,S]}\esssup_{\omega\in\Omega}
\mathbb E_s
\left[
\|g(\omega)\|_{p;[s,S];C_b(W;V)}^2
\right]^{1/2},
\]
where $\|g(\omega)\|_{p;[s,S];C_b(W;V)}
\assign
\left(
\sup_{(t_i)\in\mathcal P([s,S])}
\sum_i\sup_{w\in W}
|g_{t_{i+1}}(\omega,w)-g_{t_i}(\omega,w)|^p
\right)^{1/p}.$
\end{lemma}

\begin{proof}
We have $\sup_{t\in[a,S]}
\bigl\||h_t^X|_{C_b^2}\bigr\|_{L^\infty}
\leq C_h$, which follows directly from
\(D_y^kh_t^X(y)=D_y^kh(t,X_t,y)\), \(k=0,1,2\) and 
\(\sup_{t,x}|h_t(x,\cdot)|_{C_b^2}\leq C_h\). For
\(s\in[a,S]\) and every partition \((t_i)\in\mathcal P([s,S])\), we have
\[
h_{t_{i+1}}^X(y)-h_{t_i}^X(y)
=
h(t_{i+1},X_{t_{i+1}},y)-h(t_{i+1},X_{t_i},y)
+h(t_{i+1},X_{t_i},y)-h(t_i,X_{t_i},y).
\]
This implies the following $p$-variation bound
\[
\|h^X\|_{p;[s,S];C_b(\mathbb R^k)}
\leq
\sup_{t,x,y}|D_xh_t(x,y)|\,\|X\|_{p;[s,S]}
+\|h\|_{\rho;[s,S];C_b(\mathbb R^d\times\mathbb R^k)},
\]
where we used $\rho<p$ and hence that the $\rho$-variation seminorm controls
the $p$-variation seminorm.
By Assumption~\hyperlink{A:h}{\((h)\)}, taking conditional
\(L^2\)-norms yields
\[
[[h^X]]_{p,2;[a,S]}
\leq
C_h\bigl(1+\|X\|_{p,2;[a,S]}\bigr).
\]
The same argument applied to \(D_yh\), using the uniform bound on
\(D_xD_yh\), yields
\[
[[D_yh^X]]_{p,2;[a,S]}
\leq
C_h\bigl(1+\|X\|_{p,2;[a,S]}\bigr).
\]
This proves \eqref{hX-Assumption-Ad}.
\end{proof}

\begin{lemma}
\label{Lemma-hX-Assumption-Bg}
Suppose that Assumption~\hyperlink{A:h}{\((h)\)} holds on $[0,S]$.
For every $n\in\mathbb N$, let $a_n\in[0,S]$ and let
$X^{n,1},X^{n,2}$ be adapted continuous processes on $[a_n,S]$ satisfying
\begin{equation*}
\mathbb E\left[
  \|X^{n,1}-X^{n,2}\|_{\infty;[a_n,S]}^2
  +
  \|X^{n,1}-X^{n,2}\|_{p;[a_n,S]}^2
\right]\longrightarrow0
\end{equation*}
and
\begin{equation*}
\sup_{n\in\mathbb N}\sup_{i\in\{1,2\}}
\|X^{n,i}\|_{p,2;[a_n,S]}<\infty.
\end{equation*}
For
\[
h_t^{n,i}(y):=h(t,X_t^{n,i},y),
\qquad i\in\{1,2\},
\]
one has
\begin{equation*}
\sup_y
\|h^{n,1}(y)-h^{n,2}(y)\|_{p;[a_n,S]}
\longrightarrow0
\qquad\text{in probability},
\end{equation*}
as well as
\begin{align*}
\sup_{t\in[a_n,S]}\sup_y
|h_t^{n,1}(y)-h_t^{n,2}(y)|
&\longrightarrow0
&&\text{in }L^2,
\\
\sup_{t\in[a_n,S]}
|D_yh_t^{n,1}-D_yh_t^{n,2}|_\infty
&\longrightarrow0
&&\text{in }L^2.
\end{align*}
\end{lemma}

\begin{proof}
Put $\Delta X^n:=X^{n,1}-X^{n,2}$.
For each fixed $y$, apply the pathwise difference estimate
\cite[Lemma~2.3, equation~(2.1)]{becherer_rough_2026} with
\[
g_t^1(x)=g_t^2(x)=h(t,x,y),
\qquad
x^1=X^{n,1},
\qquad
x^2=X^{n,2}.
\]
Since $g^1=g^2$, all coefficient-difference terms in that estimate
vanish. This pathwise estimate holds with constants uniform in $y$ by
Assumption~\hyperlink{A:h}{\((h)\)}, and hence taking the supremum over
$y$ gives
\[
\begin{aligned}
\sup_y
\|h^{n,1}(y)-h^{n,2}(y)\|_{p;[a_n,S]}
&\lesssim
\|D_xh\|_\infty
\|\Delta X^n\|_{p;[a_n,S]}
\\
&\quad+
\|\Delta X^n\|_{\infty;[a_n,S]}
\|D_xh\|_{p;[a_n,S];C_b}
\\
&\quad+
\|D_x^2h\|_\infty
\|\Delta X^n\|_{\infty;[a_n,S]}
\bigl(
\|X^{n,1}\|_{p;[a_n,S]}+
\|X^{n,2}\|_{p;[a_n,S]}
\bigr).
\end{aligned}
\]
The first term on the right-hand side converges to zero in $L^2$. In
the remaining terms,
$\|\Delta X^n\|_{\infty;[a_n,S]}$ converges to zero in
probability, while the other factors are bounded in probability by the
uniform conditional $(p,2)$-bounds. Therefore,
\[
\sup_y
\|h^{n,1}(y)-h^{n,2}(y)\|_{p;[a_n,S]}
\longrightarrow0
\]
in probability. Finally, boundedness of $D_xh$ and $D_xD_yh$ gives
\[
\sup_{t\in[a_n,S]}\sup_y
|h_t^{n,1}(y)-h_t^{n,2}(y)|
\le
\|D_xh\|_\infty
\|\Delta X^n\|_{\infty;[a_n,S]}
\]
and
\[
\sup_{t\in[a_n,S]}
|D_yh_t^{n,1}-D_yh_t^{n,2}|_\infty
\le
\|D_xD_yh\|_\infty
\|\Delta X^n\|_{\infty;[a_n,S]}.
\]
Both right-hand sides converge to zero in \(L^2\).
\end{proof}

\textbf{AI usage statement:} Generative AI tools, specifically ChatGPT (versions GPT-5.4--GPT-5.6), were used during the preparation of this manuscript to assist with language editing, exposition, literature searches, and feedback on mathematical arguments. All mathematical arguments were developed by the authors, with AI tools serving a supporting role. The authors independently checked all claims, proofs, and references and take full responsibility for the correctness and content of the manuscript.

\textbf{Acknowledgments:} The authors acknowledge funding by the
Deutsche Forschungsgemeinschaft (DFG, German Research Foundation) – CRC/TRR 388
"Rough Analysis, Stochastic Dynamics and Related Fields" – Project ID 516748464. 

\section*{Notation}
For the reader's convenience, we collect the main notation used throughout the
paper.
\begin{itemize}
  \item $C(\mathbb{R}^d;\mathbb{R}^k)$: continuous maps with the compact-open
  topology; $\mathrm{BUC}(\mathbb{R}^d;\mathbb{R}^k)$: bounded uniformly
  continuous maps equipped with the uniform norm.
  {In Section~\ref{section-markov-process}, the measurable
  state space $(E,\mathcal E)$ is $\mathrm{BUC}$ equipped with the
  $\sigma$-algebra generated by spatial evaluations, equivalently the
  Borel $\sigma$-algebra of the restricted compact-open topology.}
  \item $D(I;E)$: the space of c\`agl\`ad $E$-valued paths on $I$.
  \item $D^p(I;E)$: paths in $D(I;E)$ of finite $p$-variation, with seminorm
  $\|x\|_{p;[a,b];E}$; $C^p(I;E)$: the continuous subspace;
  $\mathcal P(I)$: finite partitions of $I$; see
  p.~\pageref{section-decorated-path}.
  \item $C^{0,p}(I;\mathbb{R}^e)$ and $D^{0,p}(I;\mathbb{R}^e)$: closures of
  smooth paths in $C^p$ and $D^p$ for the $p$-variation and $p$-variation-type
  $M1$ metrics; see p.~\pageref{smooth-path-closures}.
  \item $\bar{\mathcal{D}}(I;E)$: representatives $(\Phi,\Pi)$ of decorated
  paths; $\mathfrak{D}(I;E)$: their quotient under reparametrization;
  $\bar{\mathcal{D}}^{p\text{-var}}(I;E)$ and
  $\mathfrak{D}^{p\text{-var}}(I;E)$: the corresponding finite
  $p$-variation spaces before and after taking the quotient; see
  p.~\pageref{Def-decorated}.
  \item $\Phi^\delta$: the $\delta$-extension; $\tau^\delta$ and $c^\delta$:
  its time change and left inverse; see p.~\pageref{delta-extention}.
  \item $\alpha_p$, $\alpha_\infty$: Skorokhod-type metrics on decorated
  paths; $\sigma^p_{J1}$, $\sigma^p_{M1}$: the corresponding metrics on
  $D^p$; see p.~\pageref{Skorokhod-metric}.
  \item $\imath$, $\jmath$: the constant and linear embeddings into the
  decorated-path space; see p.~\pageref{embedding-i}.
  \item $\Delta W_t\assign W(t+)-W(t)$: the right jump of $W$;
  {$\Delta u(t,x)\assign u(t+,x)-u(t,x)$: the right jump of $u$;}
  $\mathcal C(W)$: its continuity points; $\diamond\,dW$: Marcus integration.
  \item $\varphi(V,x,s)$: the backward ODE flow with terminal value $x$;
  $\kappa_{V,W}$: the Marcus embedding for a vector field $V$ and driver
  $W$; $\kappa_{h,W}$: the Marcus lift for RPDE solution paths; see
  pp.~\pageref{Marcus-embedding} and~\pageref{Def-Marcus-lift}.
  \item $\mathcal L_t$: the second-order operator in \eqref{semilinear-RPDE}.
  \item $u^W\assign\mathcal U(\xi,W)$: the robust viscosity solution;
  $\mathbf u^W\assign\kappa_{-h,W}(u^W)$: its decorated Marcus lift.
  \item $B$: the Brownian motion in the rough FBSDE; $c$: a deterministic
  continuous non-decreasing clock; $M=B\circ c$: the clocked martingale;
  $B^\delta=B\circ c^\delta$: the time-extended Brownian motion; see
  \eqref{extended-rough-FBSDE} and~\eqref{pairwise-time-stretched-FBSDE}.
  \item $\mathcal B_p([0,S];\mathbb{R}^k)$ and
  $\mathrm{BMO}([0,S];\mathbb{R}^{k\times d})$: the solution spaces for the
  clocked rough BSDE, with norms $\|Y\|_{p,2;[0,S]}$ and
  $\|Z\|_{\mathrm{BMO};[0,S]}$, respectively; the BMO norm is taken with
  respect to $dc$; see p.~\pageref{YZ-statespace}.
  \item $L$: the stochastic driver in the SPDE;
  $\mathcal F^L_{s,t}$: its increment $\sigma$-field;
  $u^L(s;t,\eta)$: the randomised RPDE solution with terminal value $\eta$
  at time $t$.
  \item $\tilde L_s=L_T-L_{T-s}$: the time-reversed driver;
  $v^L(t;s,\eta)=u^L(T-t;T-s,\eta)$: the forward-in-time solution process;
  $P_{s,t}$: its transition operators;
  {$B_b(E,\mathcal E)$: bounded $\mathcal E$-measurable
  real-valued functions on the state space $E$;} see
  Section~\ref{Section-SPDE-Markov}.
\end{itemize}

\

\


\begin{thebibliography}{}
\bibitem[App09]{applebaum_levy_2009} David Applebaum.
\newblock \textit{L\'evy processes and stochastic calculus}.
\newblock Cambridge Studies in Advanced Mathematics. Cambridge University Press, Cambridge, 2nd edition, 2009.

\bibitem[BBP97]{barles_backward_1997}Guy Barles, Rainer Buckdahn, and  Etienne Pardoux.
\newblock Backward stochastic differential equations and integral-partial differential equations.
\newblock \textit{Stochastics and Stochastic Reports}, 60(1--2):57--83, 1997.

\bibitem[Bil99]{billingsley_convergence_1999}Patrick Billingsley.
\newblock \textit{Convergence of probability measures}.
\newblock John Wiley \& Sons, New York, 2nd edition, 1999. 

\bibitem[BHM20]{brzezniak_weak_2020}Zdzis{\l}aw Brze{\'z}niak, Fabian Hornung, and Utpal Manna.
\newblock Weak martingale solutions for the stochastic nonlinear Schr{\"o}dinger equation driven by pure jump noise.
\newblock \textit{Stochastics and Partial Differential Equations: Analysis and Computations}, 8(1):1--53, 2020.

\bibitem[BLZ14]{brzezniak_strong_2014}Zdzis{\l}aw Brze{\'z}niak, Wei Liu, and Jiahui Zhu.
\newblock Strong solutions for SPDE with locally monotone coefficients driven by L{\'e}vy noise.
\newblock \textit{Nonlinear Analysis: Real World Applications}, 17:283--310, 2014.

\bibitem[BM19]{brzezniak_weak_2019}Zdzis{\l}aw Brze{\'z}niak and Utpal Manna.
\newblock Weak solutions of a stochastic Landau--Lifshitz--Gilbert equation driven by pure jump noise.
\newblock \textit{Communications in Mathematical Physics}, 371(3):1071--1129, 2019.

\bibitem[BMM19]{brzezniak_wongzakai_2019} Zdzis{\l}aw Brze{\'z}niak, Utpal Manna, and Debopriya Mukherjee.
\newblock Wong--Zakai approximation for the stochastic Landau--Lifshitz--Gilbert equations.
\newblock \textit{Journal of Differential Equations}, 267(2):776--825, 2019.

\bibitem[BMP19]{brzezniak_martingale_2019} Zdzis{\l}aw Brze{\'z}niak, Utpal Manna, and Akash Ashirbad Panda.
\newblock Martingale solutions of nematic liquid crystals driven by pure jump noise in the Marcus canonical form.
\newblock \textit{Journal of Differential Equations}, 266(10):6204--6283, 2019.

\bibitem[BS26]{becherer_rough_2026}Dirk Becherer and Yuchen Sun.
\newblock Rough backward SDEs with discontinuous Young drivers.
\newblock \href{https://arxiv.org/html/2505.20437v2}{arXiv:2505.20437v2}, 2026.

\bibitem[CF09]{caruana_partial_2009}Michael Caruana and Peter Friz.
\newblock Partial differential equations driven by rough paths.
\newblock \textit{Journal of Differential Equations}, 247(1):140--173, 2009.

\bibitem[CT04]{cont_financial_2004} Rama Cont and Peter Tankov.
\newblock \textit{Financial modelling with jump processes}.
\newblock Chapman \& Hall/CRC Financial Mathematics Series. Chapman \& Hall/CRC, 2004.

\bibitem[CP14]{chechkin_marcus_2014} A.~Chechkin and I.~Pavlyukevich.
\newblock Marcus versus Stratonovich for systems with jump noise.
\newblock \textit{Journal of Physics A: Mathematical and Theoretical}, 47(34):342001, 2014.

\bibitem[CF19]{chevyrev_canonical_2019}Ilya Chevyrev and Peter~K.~Friz.
\newblock Canonical RDEs and general semimartingales as rough paths.
\newblock \textit{The Annals of Probability}, 47(1):420--463, 2019.

\bibitem[CFKM20]{chevyrev_superdiffusive_2020}Ilya Chevyrev, Peter~K.~Friz, Alexey Korepanov, and Ian Melbourne.
\newblock Superdiffusive limits for deterministic fast-slow dynamical systems.
\newblock \textit{Probability Theory and Related Fields}, 178(3--4):735--770, 2020.

\bibitem[CFO11]{caruana_rough_2011}Michael Caruana, Peter~K.~Friz, and  Harald Oberhauser.
\newblock A (rough) pathwise approach to a class of non-linear stochastic partial differential equations.
\newblock \textit{Annales de l'Institut Henri Poincar{\'e} C, Analyse non lin{\'e}aire}, 28(1):27--46, 2011.

\bibitem[CKM24]{chevyrev_superdiffusive_2024}Ilya Chevyrev, Alexey Korepanov, and Ian Melbourne.
\newblock Superdiffusive limits beyond the Marcus regime for deterministic fast-slow systems.
\newblock \textit{Communications of the American Mathematical Society}, 4(16):746--786, 2024.

\bibitem[{CIL92}]{crandall_users_1992} {Michael~G. Crandall, Hitoshi Ishii, and Pierre-Louis Lions. \newblock User's guide to viscosity solutions of second order partial differential equations. \newblock\textit{Bulletin of the American Mathematical Society}, 27(1):1--67, 1992.}

\bibitem[DN98]{dudley_introduction_1998-1}Richard~M.~Dudley and Rimas Norvai{\v s}a.
\newblock \textit{An introduction to $p$-variation and Young integrals: With emphasis on sample functions of stochastic processes}.
\newblock Centre for Mathematical Physics and Stochastics, University of Aarhus, 1998. 

\bibitem[DPZ14]{da_prato_stochastic_2014} Giuseppe Da Prato and Jerzy Zabczyk.
\newblock \textit{Stochastic equations in infinite dimensions}.
\newblock Encyclopedia of Mathematics and its Applications. Cambridge University Press, Cambridge, 2nd edition, 2014.

\bibitem[DF12]{diehl_backward_2012}Joscha Diehl and Peter Friz.
\newblock Backward stochastic differential equations with rough drivers.
\newblock \textit{The Annals of Probability}, 40(4):1715--1758, 2012.

\bibitem[DZ17]{diehl_backward_2017}Joscha Diehl and Jianfeng Zhang.
\newblock Backward stochastic differential equations with Young drift.
\newblock \textit{Probability, Uncertainty and Quantitative Risk}, 2(1):5, 2017.

\bibitem[EK86]{ethier_markov_1986}Stewart~N.~Ethier and Thomas~G.~Kurtz.
\newblock \textit{Markov processes}.
\newblock Wiley Series in Probability and Mathematical Statistics. John Wiley \& Sons, Inc., New York, 1986.

\bibitem[FH20]{friz_course_2020} Peter K. Friz and Martin Hairer.
\newblock \textit{A course on rough paths: With an introduction to regularity structures}.
\newblock Universitext. Springer, Cham, 2nd edition, 2020.

\bibitem[FO14]{friz_rough_2014}Peter Friz and Harald Oberhauser.
\newblock Rough path stability of (semi-)linear SPDEs.
\newblock \textit{Probability Theory and Related Fields}, 158(1):401--434, 2014.

\bibitem[FGLS17]{friz_eikonal_2017}Peter~K.~Friz, Paul Gassiat, Pierre-Louis Lions, and Panagiotis~E.~Souganidis.
\newblock Eikonal equations and pathwise solutions to fully non-linear SPDEs.
\newblock \textit{Stochastics and Partial Differential Equations: Analysis and Computations}, 5(2):256--277, 2017.

\bibitem[FV10]{friz_multidimensional_2010}Peter~K.~Friz and  Nicolas~B.~Victoir.
\newblock \textit{Multidimensional stochastic processes as rough paths: Theory and applications}.
\newblock Cambridge Studies in Advanced Mathematics. Cambridge University Press, Cambridge, 2010.

\bibitem[FZ18]{friz_differential_2018}Peter~K.~Friz and Huilin Zhang.
\newblock Differential equations driven by rough paths with jumps.
\newblock \textit{Journal of Differential Equations}, 264(10):6226--6301, 2018.

\bibitem[HP23]{hartmann_first-order_2023}Lena-Susanne Hartmann and Ilya Pavlyukevich.
\newblock First-order linear Marcus SPDEs.
\newblock \textit{Stochastics and Dynamics}, 23(6):2350054, 2023.

\bibitem[IL90]{ishii_viscosity_1990} Hitoshi Ishii and Pierre-Louis Lions.
\newblock Viscosity solutions of fully nonlinear second-order elliptic partial differential equations.
\newblock \textit{Journal of Differential Equations}, 83(1):26--78, 1990.

\bibitem[JS03]{jacod_limit_2003}Jean Jacod and Albert~N.~Shiryaev.
\newblock \textit{Limit theorems for stochastic processes}.
\newblock Grundlehren der mathematischen Wissenschaften, vol.~288.
\newblock Springer-Verlag, Berlin, 2nd edition, 2003.

\bibitem[{Kal21}]{kallenberg_foundations_2021} {Olav Kallenberg. \newblock \textit{Foundations of modern probability}. \newblock Probability Theory and Stochastic Modelling, vol.~99. Springer, Cham, 3rd edition, 2021.}

\bibitem[Kel75]{kelley_general_1975} John~L. Kelley.
\newblock \textit{General topology}.
\newblock Graduate Texts in Mathematics, vol.~27. Springer, New York, NY, 1975.

\bibitem[{Kun90}]{kunita1990stochastic} {Hiroshi Kunita. \newblock \textit{Stochastic flows and stochastic differential equations}. \newblock Cambridge Studies in Advanced Mathematics, vol.~24. Cambridge University Press, 1990.}

\bibitem[Kun04]{kunita_stochastic_2004} Hiroshi Kunita.
\newblock Stochastic differential equations based on L{\'e}vy processes and stochastic flows of diffeomorphisms.
\newblock In M.~M.~Rao, editor, \textit{Real and stochastic analysis: New perspectives}, pp.~305--373. Birkh{\"a}user, Boston, MA, 2004.

\bibitem[Kun19]{kunita_stochastic_2019} Hiroshi Kunita.
\newblock \textit{Stochastic flows and jump-diffusions}.
\newblock Probability Theory and Stochastic Modelling, vol.~92. Springer, Singapore, 2019.

\bibitem[KPP95]{kurtz_stratonovich_1995}Thomas~G.~Kurtz, {\'E}tienne Pardoux, and Philip Protter.
\newblock Stratonovich stochastic differential equations driven by general semimartingales.
\newblock \textit{Annales de l'Institut Henri Poincar{\'e}, Probabilit{\'e}s et Statistiques}, 31(2):351--377, 1995.

\bibitem[LS98]{lions_fully_1998}Pierre-Louis Lions and  Panagiotis~E.~Souganidis.
\newblock Fully nonlinear stochastic partial differential equations.
\newblock \textit{Comptes Rendus de l'Acad{\'e}mie des Sciences - Series I - Mathematics}, 326(9):1085--1092, 1998.

\bibitem[LT25]{liang_multidimensional_2023} Jiahao Liang and Shanjian Tang.
\newblock Multidimensional backward stochastic differential equations with rough drifts.
\newblock \textit{Transactions of the American Mathematical Society}, 378(1):201--257, 2025.

\bibitem[Mar80]{marcus_modeling_1980}Steven~I.~Marcus.
\newblock Modeling and approximation of stochastic differential equations driven by semimartingales.
\newblock \textit{Stochastics}, 4(3):223--245, 1980.

\bibitem[Mon72]{monroe_gamma-variation_1972} Itrel Monroe.
\newblock On the $\gamma$-variation of processes with stationary independent increments.
\newblock \textit{The Annals of Mathematical Statistics}, 43(4):1213--1220, 1972.

\bibitem[MR10]{marinelli_well-posedness_2010}Carlo Marinelli and Michael Röckner.
\newblock Well-posedness and asymptotic behavior for stochastic reaction-diffusion equations with multiplicative Poisson noise.
\newblock \textit{Electronic Journal of Probability}, 15, article no.~49, pp.~1528--1555, 2010.

\bibitem[Par98]{decreusefond_backward_1998}{\'E}tienne Pardoux.
\newblock Backward stochastic differential equations and viscosity solutions of systems of semilinear parabolic and elliptic PDEs of second order.
\newblock In Laurent Decreusefond, Bernt {\O}ksendal, Jon Gjerde, and Ali~S{\"u}leyman {\"U}st{\"u}nel, editors, \textit{Stochastic analysis and related topics VI}, pp.~79--127. Birkh{\"a}user, Boston, MA, 1998.

\bibitem[Pha09]{pham_continuous-time_2009}Huy{\^e}n Pham.
\newblock \textit{Continuous-time stochastic control and optimization with financial applications}. Stochastic Modelling and Applied Probability, vol.~61.
\newblock Springer, Berlin, Heidelberg, 2009.

\bibitem[PP92]{rozovskii_backward_1992}E.~Pardoux and S.~Peng.
\newblock Backward stochastic differential equations and quasilinear parabolic partial differential equations.
\newblock In Boris~L.~Rozovskii  and Richard~B.~Sowers, editors, \textit{Stochastic partial differential equations and their applications}, Lecture Notes in Control and Information Sciences, vol.~176, pp.~200--217. Springer-Verlag, Berlin, Heidelberg, 1992.

\bibitem[PPR97]{pardoux_probabilistic_1997}E.~Pardoux, F.~Pradeilles, and Z.~Rao.
\newblock Probabilistic interpretation of a system of semi-linear parabolic partial differential equations.
\newblock \textit{Annales de l'Institut Henri Poincar{\'e}, Probabilit{\'e}s et Statistiques}, 33(4):467--490, 1997.

\bibitem[PZ07]{peszat_stochastic_2007}S.~Peszat and J.~Zabczyk.
\newblock \textit{Stochastic partial differential equations with L{\'e}vy noise: An evolution equation approach}.
\newblock Encyclopedia of Mathematics and its Applications. Cambridge University Press, Cambridge, 2007.

\bibitem[SZZ25]{song_backward_2025-1} Jian Song, Huilin Zhang, and Kuan Zhang.
\newblock Backward stochastic differential equations with nonlinear Young drivers II.
\newblock \href{https://arxiv.org/abs/2509.05183}{arXiv:2509.05183}, 2025.

\bibitem[Wal98]{walter_ordinary_1998} Wolfgang Walter.
\newblock \textit{Ordinary differential equations}.
\newblock Graduate Texts in Mathematics, vol.~182. Springer, New York, NY, 1998.

\bibitem[Whi02]{whitt_stochastic-process_2002}Ward Whitt.
\newblock \textit{Stochastic-process limits: An introduction to stochastic-process limits and their application to queues}.
\newblock Springer Series in Operations Research. Springer, New York, Berlin, Heidelberg, 2002.

\bibitem[Zha17]{zhang_backward_2017}Jianfeng Zhang.
\newblock \textit{Backward stochastic differential equations}.
\newblock Probability Theory and Stochastic Modelling, vol.~86. Springer, New York, 2017.
\end{thebibliography}
\end{document}